\documentclass[12pt]{amsart}
\usepackage[T1]{fontenc}
\usepackage[utf8]{inputenc}
\usepackage[
  backend=biber,
  style=numeric,
  sorting=none,
  giveninits=false,
  doi=true,
  isbn=true,
  url=true,
  eprint=true
]{biblatex}
\usepackage{amsmath, amssymb, amsthm}
\usepackage{mathrsfs}
\usepackage{bbm}
\usepackage{hyperref}
\usepackage{geometry}
\usepackage{enumitem}
\usepackage{booktabs}
\usepackage{array}
\usepackage{tikz-cd}   % commutative diagrams

\newtheorem{theorem}{Theorem}[section]
\newtheorem{proposition}[theorem]{Proposition}
\newtheorem{lemma}[theorem]{Lemma}
\newtheorem{corollary}[theorem]{Corollary}
\newtheorem{definition}[theorem]{Definition}

\theoremstyle{remark}
\newtheorem{remark}[theorem]{Remark}

\newtheorem{hypothesis}[theorem]{Hypothesis}
\begin{document}
%% ============================================================

\title[Defect Triangles and Hodge Atom Shadows]{Defect Triangles and Intersection-Space Hodge Atom Shadows for Calabi--Yau Conifolds}

\author{Abdul Rahman}
%\address{Department / Institution (optional)}
% \urladdr{\url{https://...}} % optional
\thanks{Email: arahman@alum.howard.edu}

\begin{abstract}
We prove a projection-triangle statement for projective Calabi--Yau conifold degenerations and use it to organize an intersection-space Hodge atom package.  Let $X_0$ be a projective Calabi--Yau threefold hypersurface with isolated ordinary double points, and let $\mathcal{IS}_{X_0}$ denote the Banagl--Budur--Maxim intersection-space complex.  Under the relevant Banagl--Budur--Maxim, multi-node gluing, mixed-Hodge-module, and specialization-splitting hypotheses, the nearby-cycle object decomposes as $\psi_\pi(F)\simeq \mathcal{IS}^{H}_{X_0}\oplus\mathcal C^H_\Sigma$.  Projecting the variation morphism to the intersection-space summand gives $\operatorname{var}_I:\phi_\pi(F)\to\mathcal{IS}^{H}_{X_0}$, and the octahedral axiom yields a distinguished triangle $P^H\to P^H_I\to\mathcal C^H_\Sigma\xrightarrow{+1}$, where $P^H:=\operatorname{Cone}(\operatorname{var})[-1]$ and $P^H_I:=\operatorname{Cone}(\operatorname{var}_I)[-1]$.  The same mixed-Hodge-module realization defines the intersection-space atom package $\mathsf{HA}^{I}(X_0):=\operatorname{Atom}_{\mathrm{Hod}}(\mathbb H^*(X_0;\mathcal{IS}^{H}_{X_0}))$, which we compare with the intersection-homology package $\mathsf{HA}^{IH}(X_0):=\operatorname{Atom}_{\mathrm{Hod}}(\mathbb H^*(X_0;IC^H_{X_0}))$.  Assuming, the self-dual specialization-splitting hypothesis SD, the projected variation object is Verdier-dual to the projected canonical companion $Q^H_I:=\operatorname{Cone}(\operatorname{can}_I)[-1]$, namely $\mathbb D P^H_I\simeq Q^H_I(3)$.  Under the additional mixed-Hodge-realization hypothesis for Banagl's middle exact sequence, we isolate a rigid--vanishing filtration of the middle-degree intersection-space atom package and identify the IIB vanishing atom with the mixed-Hodge realization of the kernel in that sequence.  For the classical $125$-node quintic, the resulting middle-degree IC--intersection-space defect has rank $204-2=202$.  The construction is deliberately Hodge-realization-level: it does not assert a Gromov--Witten theory, quantum product, Dubrovin connection, or full A-model $F$-bundle for the spatial intersection space itself.  The defect object $\mathcal C^H_\Sigma$ and the class $\Delta_{I/IC}(X_0):=[\mathcal{IS}^{H}_{X_0}]-[IC^H_{X_0}]$ are identified as geometry-side Hodge-theoretic handoff objects for future moduli-stack and refined DT/BPS wall-crossing comparisons.
\end{abstract}
\maketitle
\tableofcontents
%%%%%%%%%%%%%%%%%%%%%%%%%%%%%%%%%%%%%%%%%%%%%%%%%%%%%%%%%%%%%
\section{Introduction}

\subsection{Hodge atoms and the singular-space problem}

The Hodge atom formalism of Katzarkov--Kontsevich--Pantev--Yu attaches to a smooth projective variety a collection of spectral Hodge-theoretic pieces extracted from its A-model $F$-bundle \cite{KKPY25,KKPYY_HodgeAtoms}.  In that setting, the atoms arise from the spectral decomposition of quantum multiplication by the Euler vector field, after passing to the non-archimedean framework needed to control the Stokes phenomenon \cite{KKPY25,Iritani09,Givental96}.  This construction is intrinsically smooth and quantum-cohomological: it uses a quantum product, a Dubrovin connection, and an Euler-field action.

The purpose of the present paper is to develop a singular-space analogue for projective Calabi--Yau conifolds at the level where the available sheaf-theoretic and Hodge-theoretic structures naturally live.  The singular space we consider is a projective Calabi--Yau threefold hypersurface $X_0$ with finitely many ordinary double points
\[
   \Sigma=\{p_1,\ldots,p_r\}.
\]
The spatial intersection space $I^{\bar m}X_0$ of Banagl is not itself a smooth projective variety and is not equipped with a canonical Gromov--Witten theory, quantum product, Dubrovin connection, or Euler field \cite{Banagl_IntersectionSpacesBook,BanaglBudurMaxim14,FaberPandharipande2000,Givental96}. Thus a literal extension of the smooth-projective Hodge atom construction to $I^{\bar m}X_0$ would be unjustified.

We instead work with a weaker but sheaf-theoretically precise replacement: Hodge-realization atom shadows of mixed-Hodge-module objects.  The basic input is not a quantum $F$-bundle on the spatial intersection space, but rather the Banagl--Budur--Maxim intersection-space complex $\mathcal{IS}_{X_0}$ and a mixed-Hodge-module lift $\mathcal{IS}^H_{X_0}$ \cite{BanaglBudurMaxim14,Saito90}.  The resulting atom package is defined from the mixed Hodge structures on
\[
   \mathbb H^*(X_0;\mathcal{IS}^H_{X_0}).
\]
In this sense, the present construction is an atomization of the mixed-Hodge-theoretic realization of the intersection-space complex.

The main structural result of the paper is the projection triangle.  Under the Banagl--Budur--Maxim specialization splitting, the nearby-cycle object decomposes into an intersection-space summand and a singularity-supported complement.  Projecting the variation morphism to the intersection-space summand and applying the octahedral axiom in $D^b\operatorname{MHM}(X_0)$ gives a distinguished triangle
\[
   P^H\longrightarrow P^H_I\longrightarrow \mathcal C^H_\Sigma\xrightarrow{+1}.
\]
This triangle identifies $\mathcal C^H_\Sigma$ as the cofiber measuring the difference between the corrected conifold object $P^H$ and its intersection-space projection $P^H_I$.  The intersection-space Hodge atom package is then the Hodge-realization shadow of the type-IIB summand participating in this triangle.

\subsection{IIA/IIB conifold pair and the defect viewpoint}

Banagl's theory of intersection spaces was introduced to address a phenomenon in conifold transitions that is not captured by intersection homology alone \cite{Banagl_IntersectionSpacesBook,BanaglBudurMaxim14,GoreskyMacPherson80,GoreskyMacPherson83}.  For singular Calabi--Yau conifolds, the middle-perversity intersection homology package $IH_*(X_0)$ records the type-IIA sector, while the ordinary homology of the middle-perversity intersection space $H_*(I^{\bar m}X_0;\mathbb Q)$ records the type-IIB sector \cite{Banagl_IntersectionSpacesBook,BanaglBudurMaxim14,Strominger95,HubschBestiary2nd2024}.  Thus the two theories form complementary homological packages associated with the same singular fiber.

The present paper refines this IIA/IIB comparison from rational homology groups to Hodge-realization atom packages.  On the type-IIA side, the relevant mixed-Hodge-module object is the intersection complex $IC^H_{X_0}$.  On the type-IIB side, the relevant object is the intersection-space mixed Hodge module $\mathcal{IS}^H_{X_0}$.  The corresponding atom packages are
\[
   \mathsf{HA}^{IH}(X_0)
   :=
   \operatorname{Atom}_{\mathrm{Hod}}
   \bigl(\mathbb H^*(X_0;IC^H_{X_0})\bigr)
\]
and
\[
   \mathsf{HA}^{I}(X_0)
   :=
   \operatorname{Atom}_{\mathrm{Hod}}
   \bigl(\mathbb H^*(X_0;\mathcal{IS}^{H}_{X_0})\bigr).
\]
The first package atomizes the intersection-homology realization, while the second atomizes the intersection-space realization.  The comparison between them is one of the realization-level themes of the paper, while the projection triangle identifies the corresponding finite-node defect object.

The two mixed-Hodge-module objects $IC^H_{X_0}$ and $\mathcal{IS}^H_{X_0}$ agree on the smooth locus and differ only through their singular gluing data at the ordinary double points.  We package this difference in two related ways.  First, at the level of atom shadows, we consider the formal difference
\[
   \mathsf{Def}_{I/IC}(X_0)
   :=
   \mathsf{HA}^{I}(X_0)-\mathsf{HA}^{IH}(X_0).
\]
Second, at the mixed-Hodge-module level, using the Grothendieck group of Saito's category of mixed Hodge modules \cite{Saito90}, we consider the Grothendieck defect
\[
   \Delta_{I/IC}(X_0)
   :=
   [\mathcal{IS}^H_{X_0}]-[IC^H_{X_0}]
   \in K_0(\operatorname{MHM}(X_0)).
\]
Here \(K_0(\operatorname{MHM}(X_0))\) denotes the Grothendieck group of the abelian category of mixed Hodge modules on \(X_0\) in Saito's sense \cite{Saito90,Saito_MHM}: it is generated by isomorphism classes \([M]\) of mixed Hodge modules, modulo the relations \([M]=[M']+[M'']\) for every short exact sequence \(0\to M'\to M\to M''\to 0\).

The projection triangle gives a map-level source for this defect viewpoint: the complement $\mathcal C^H_\Sigma$ is the singularity-supported residue between the corrected conifold object and its intersection-space projection.

This defect object is not identified here with a Donaldson--Thomas invariant.  Rather, it is the Hodge-realization object that should be transported in future work to a moduli-stack setting before comparison with refined or cohomological Donaldson--Thomas wall-crossing and BPS packages \cite{KontsevichSoibelman08,JoyceSong12,Szendroi08,NagaoNakajima11,DavisonMeinhardt20}.

\begin{remark}[Full atoms versus atom shadows]
\label{rem:full-atoms-versus-shadows}
The terminology ``Hodge-realization atom shadow'' is relative to the full conifold atom construction.  A full conifold atom has three ingredients: a distinguished mixed-Hodge-module or perverse carrier, a Hodge realization through hypercohomology, and an identification of that realized Hodge sector with an Euler-spectral summand of the relevant non-archimedean \(F\)-bundle.  In the corrected conifold theory this last step is supplied by the Stokes--Extension Identification, which transports the mixed-Hodge-module data of \(P^H=\operatorname{Cone}(\operatorname{var})[-1]\), \(IC^H_{X_0}\), and the node-supported summands to genuine spectral pieces of the Dubrovin/Gauss--Manin \(F\)-bundle.

The intersection-space object \(\mathcal{IS}^H_{X_0}\) has the first two features, and more: it is a distinguished mixed-Hodge-module realization of Banagl's type-IIB intersection-space sector, it agrees with \(IC^H_{X_0}\) on the smooth locus, it carries the rigid--vanishing middle filtration, it appears as a summand of the nearby-cycle object under the specialization splitting, and it participates in the projection triangle
\[
   P^H\longrightarrow P^H_I\longrightarrow \mathcal C^H_\Sigma\xrightarrow{+1}.
\]
However, this paper does not prove an analogue of the Stokes--Extension Identification sending \(\mathbb H^*(X_0;\mathcal{IS}^H_{X_0})\) to an Euler-spectral summand of a non-archimedean \(F\)-bundle.  For this reason \(\mathsf{HA}^{I}(X_0)\) is called a Hodge-realization atom shadow rather than a full conifold atom package.
\end{remark}

\begin{remark}[A hierarchy of atom structures]
\label{rem:hierarchy-of-atom-structures}
To situate this construction within a broader landscape, we record a hierarchy between full atoms and atom shadows.  At the lowest level, any object \(M^H\in D^b\operatorname{MHM}(X)\) with finite-dimensional hypercohomology has a Hodge realization \(\mathbb H^*(X;M^H)\).  When this object is geometrically distinguished, we may form the candidate Hodge-realization atom shadow \(\mathfrak A_{\mathrm{Hod}}(M^H)\).  When the object belongs to a specific conifold sector, such as \(IC^H_{X_0}\), \(\mathcal{IS}^H_{X_0}\), \(P^H_I\), or \(\mathcal C^H_\Sigma\), and carries additional structure such as exterior agreement, specialization splitting, a rigid--vanishing filtration, or a projection triangle, the shadow becomes a geometric atom shadow.  Finally, a geometric atom shadow becomes a full conifold atom only after an additional theorem identifies its Hodge realization with an Euler-spectral summand of the relevant non-archimedean \(F\)-bundle.  The corrected conifold atoms have this final property by the Stokes--Extension Identification.  The intersection-space package constructed in the present paper is a geometric atom shadow; upgrading it to a full atom package would require an analogous spectral realization theorem.
\end{remark}

\subsection{The BBM intersection-space complex}

The key sheaf-theoretic input is the Banagl--Budur--Maxim intersection-space complex.  For projective hypersurfaces with isolated singularities, Banagl--Budur--Maxim construct a perverse sheaf $\mathcal{IS}_X$ whose hypercohomology computes the rational cohomology of Banagl's intersection space \cite{BanaglBudurMaxim14,Maxim_IntersectionSpacesPerverseSheavesSurvey}:
\[
   \mathbb H^*(X;\mathcal{IS}_X)\cong H^*(IX;\mathbb Q).
\]
Under the hypotheses used here, this perverse sheaf underlies a mixed Hodge module \cite{BanaglBudurMaxim14,Saito90}.  We denote such a lift by
\[
   \mathcal{IS}^{H}_{X_0}\in \operatorname{MHM}(X_0),
   \qquad
   \operatorname{rat}(\mathcal{IS}^{H}_{X_0})\simeq \mathcal{IS}_{X_0}.
\]

This is the point at which the intersection-space theory enters the same formal environment as the usual intersection complex and the corrected nearby/vanishing-cycle objects appearing in conifold degenerations.  The central object of the paper is therefore not merely the topological vector space $H^*(I^{\bar m}X_0;\mathbb Q)$, but the mixed-Hodge-module realization $\mathcal{IS}^H_{X_0}$, its hypercohomology, and its interaction with nearby and vanishing cycles.

Because our main examples have several nodes, we make the relevant multi-node hypotheses explicit.  The local Banagl--Budur--Maxim input applies directly at each ordinary double point, since an ordinary double point is a weighted homogeneous isolated hypersurface singularity.  The passage from one node to finitely many nodes requires either a cited multi-singularity extension of the Banagl--Budur--Maxim construction or a constructible-complex gluing framework in which the relevant obstruction vanishes.  We isolate these assumptions in the standing hypotheses and use them throughout.

\subsection{Why BBM rather than the full constructible-complex framework}

The broader constructible-complex approach to intersection spaces gives an axiomatic and obstruction-theoretic framework for constructing intersection-space complexes \cite{AgustinFernandezDeBobadilla_IntersectionSpaceConstructibleComplexes}.  This is useful for understanding the general scope of the theory and for treating settings beyond isolated hypersurface singularities.  The present paper, however, is focused on projective Calabi--Yau hypersurface conifolds with isolated ordinary double points.  In this setting, the Banagl--Budur--Maxim construction is the most concrete input: it produces the perverse sheaf $\mathcal{IS}_X$, identifies its hypercohomology with intersection-space cohomology, and supplies the mixed-Hodge-theoretic structure needed for the atom construction.

Thus we use the Banagl--Budur--Maxim intersection-space complex as the primary object.  The constructible-complex framework remains relevant in two ways.  First, it provides a natural language for discussing multi-node gluing and possible obstruction classes.  Second, it suggests extensions of the present construction beyond the hypersurface ordinary-double-point setting.  In this paper, however, the main results are stated under the explicit hypotheses that the multi-node intersection-space complex exists, admits a mixed-Hodge-module lift, and, where needed, satisfies the specialization splitting used to form the projection triangle.

\subsection{Main constructions}

We now introduce the main objects used throughout the paper.  The type-IIA package is the intersection-homology atom package
\[
   \mathsf{HA}^{IH}(X_0)
   :=
   \operatorname{Atom}_{\mathrm{Hod}}
   \bigl(\mathbb H^*(X_0;IC^H_{X_0})\bigr).
\]
Here $IC^H_{X_0}$ denotes the mixed Hodge module lifting the intersection complex, and the atom package is extracted from the mixed Hodge structures on its hypercohomology.

The type-IIB package is the intersection-space atom package
\[
   \mathsf{HA}^{I}(X_0)
   :=
   \operatorname{Atom}_{\mathrm{Hod}}
   \bigl(\mathbb H^*(X_0;\mathcal{IS}^{H}_{X_0})\bigr).
\]
By the defining property of the intersection-space complex, the Betti realization of $\mathbb H^*(X_0;\mathcal{IS}^{H}_{X_0})$ is $H^*(I^{\bar m}X_0;\mathbb Q)$.

The notation $\operatorname{Atom}_{\mathrm{Hod}}$ denotes a Hodge-realization atom shadow.  It does not assert that $I^{\bar m}X_0$ carries a full KKPYY A-model $F$-bundle.  Rather, it records the atom-level decomposition visible in the Hodge realization of the corresponding mixed-Hodge-module object.

The corrected conifold object is built from the variation morphism between vanishing and nearby cycles in the standard nearby/vanishing-cycle formalism \cite{Saito90,KashiwaraSchapira90,DimcaSheavesInTopology}:
\[
   \operatorname{var}:\phi_\pi(F)\longrightarrow \psi_\pi(F).
\]
We set
\[
   P^H:=\operatorname{Cone}(\operatorname{var})[-1].
\]
Under the specialization splitting
\[
   \psi_\pi(F)\simeq \mathcal{IS}^H_{X_0}\oplus \mathcal C^H_\Sigma,
\]
let
\[
   \operatorname{pr}_I:\psi_\pi(F)\to \mathcal{IS}^H_{X_0}
\]
be the projection to the intersection-space summand.  The projected variation morphism is
\[
   \operatorname{var}_I
   :=
   \operatorname{pr}_I\circ\operatorname{var}:
   \phi_\pi(F)\to \mathcal{IS}^H_{X_0},
\]
and the projected corrected conifold object is
\[
   P^H_I:=\operatorname{Cone}(\operatorname{var}_I)[-1].
\]
The projection triangle relates $P^H$, $P^H_I$, and the singularity-supported complement $\mathcal C^H_\Sigma$.

When the specialization splitting is compatible with Saito--Verdier duality for nearby and vanishing cycles \cite{Saito90,Saito_MHM,KashiwaraSchapira90}, there is also a canonical companion object.  Let
\[
   \iota_I:\mathcal{IS}^H_{X_0}\to \psi_\pi(F)
\]
be the inclusion of the intersection-space summand, and set
\[
   \operatorname{can}_I
   :=
   \operatorname{can}\circ \iota_I:
   \mathcal{IS}^H_{X_0}\to \phi_\pi(F).
\]
Then define
\[
   Q^H_I:=\operatorname{Cone}(\operatorname{can}_I)[-1].
\]
The projected variation object $P^H_I$ and the projected canonical object $Q^H_I$ are Verdier-dual up to the Calabi--Yau threefold Tate twist.

\subsection{Main results}

The first result defines the intersection-space atom package from the mixed-Hodge-module lift of the Banagl--Budur--Maxim complex.

\begin{theorem}[Intersection-space Hodge atom package]
Assume Hypotheses H1--H3.  Then the mixed-Hodge-module lift
$\mathcal{IS}^{H}_{X_0}$ defines a Hodge-realization atom package
\[
   \mathsf{HA}^{I}(X_0)
   :=
   \operatorname{Atom}_{\mathrm{Hod}}
   \bigl(\mathbb H^*(X_0;\mathcal{IS}^{H}_{X_0})\bigr).
\]
Its Betti realization identifies with the atomization of Banagl's rational
intersection-space cohomology:
\[
   \operatorname{rat}\mathbb H^*(X_0;\mathcal{IS}^{H}_{X_0})
   \cong
   H^*(I^{\bar m}X_0;\mathbb Q).
\]
\end{theorem}

The second result records the projection triangle.  This is the structural point that relates the corrected conifold object to its intersection-space projection.

\begin{theorem}[Projection triangle]
Assume Hypotheses H1--H3 and S.  Let
\[
   \psi_\pi(F)\simeq \mathcal{IS}^H_{X_0}\oplus \mathcal C^H_\Sigma
\]
be the specialization splitting, and let
\[
   \operatorname{var}_I
   =
   \operatorname{pr}_I\circ\operatorname{var}
   :
   \phi_\pi(F)\to \mathcal{IS}^H_{X_0}.
\]
Set
\[
   P^H:=\operatorname{Cone}(\operatorname{var})[-1],
   \qquad
   P^H_I:=\operatorname{Cone}(\operatorname{var}_I)[-1].
\]
Then the projection $\operatorname{pr}_I:\psi_\pi(F)\to \mathcal{IS}^H_{X_0}$ induces a canonical morphism
\[
   P^H\to P^H_I
\]
and there is a distinguished triangle
\[
   P^H\longrightarrow P^H_I\longrightarrow \mathcal C^H_\Sigma\xrightarrow{+1}
\]
in $D^b\operatorname{MHM}(X_0)$.
\end{theorem}

The next result gives the projected duality statement.  It is conditional on Hypothesis SD, the self-dual compatibility of the specialization splitting with Saito--Verdier duality.

\begin{theorem}[Projected variation/canonical duality]
Assume Hypotheses H1--H3, S, and SD.  Let
\[
   \operatorname{can}_I
   =
   \operatorname{can}\circ\iota_I:
   \mathcal{IS}^H_{X_0}\to \phi_\pi(F),
\]
and set
\[
   Q^H_I:=\operatorname{Cone}(\operatorname{can}_I)[-1].
\]
Then
\[
   \mathbb D P^H_I\simeq Q^H_I(3).
\]
\end{theorem}

The type-IIA and type-IIB atom packages give a realization-level refinement of Banagl's conifold homology pair.

\begin{theorem}[IIA/IIB atom comparison]
Assume Hypotheses H1--H3.  The pair
\[
   \bigl(\mathsf{HA}^{IH}(X_0),\mathsf{HA}^{I}(X_0)\bigr)
\]
is an atom-level refinement of Banagl's IIA/IIB conifold homology pair
\[
   \bigl(IH_*(X_0),H_*(I^{\bar m}X_0)\bigr).
\]
More precisely, $IC^H_{X_0}$ realizes the intersection-homology/type-IIA
sector, while $\mathcal{IS}^{H}_{X_0}$ realizes the
intersection-space/type-IIB sector.
\end{theorem}

The next result isolates the middle-degree contribution responsible for the type-IIB vanishing sector.

\begin{theorem}[Rigid--vanishing filtration]
Assume Hypotheses H1--H3 and Hypothesis MHS-B.  Then the middle-degree intersection-space atom package has a natural two-step filtration whose associated graded is
\[
   \operatorname{gr}\mathsf{HA}^{I}_3(X_0)
   =
   \mathsf{HA}^{I}_{\mathrm{van},3}(X_0)
   +
   \mathsf{HA}^{I}_{\mathrm{rig},3}(X_0),
\]
where
\[
   \mathsf{HA}^{I}_{\mathrm{van},3}(X_0)
   =
   \operatorname{Atom}_{\mathrm{Hod}}(K_I^H(X_0)).
\]
Here $K_I^H(X_0)$ denotes the mixed-Hodge realization of the kernel in Banagl's middle-degree conifold exact sequence.
\end{theorem}

When the Banagl--Budur--Maxim specialization splitting is available, the nearby smoothing also admits an atom-level splitting.

\begin{theorem}[Conditional specialization atom splitting]
Assume Hypotheses H1--H3 and S.  If
\[
   Rsp_*\mathbb Q^H_{X_s}[3]
   \simeq
   \mathcal{IS}^{H}_{X_0}\oplus \mathcal C^H_{\Sigma},
\]
where $\mathcal C^H_{\Sigma}$ is supported on $\Sigma$, then
\[
   \mathfrak A_{\mathrm{Hod}}(Rsp_*\mathbb Q^H_{X_s}[3])
   =
   \mathsf{HA}^{I}(X_0)
   +
   \mathfrak A_{\mathrm{Hod}}(\mathcal C^H_{\Sigma}).
\]
\end{theorem}

The self-duality of the intersection-space complex gives a corresponding duality statement for atoms.

\begin{theorem}[Verdier-dual atom symmetry]
Assume Hypotheses H1--H3 and suppose $\mathcal{IS}^{H}_{X_0}$ is Verdier
self-dual with the cohomological Calabi--Yau threefold normalization
\[
   \mathbb D\mathcal{IS}^{H}_{X_0}\simeq \mathcal{IS}^{H}_{X_0}(3).
\]
Then the induced Poincare duality pairing on intersection-space
hypercohomology gives
\[
   \bigl(\mathsf{HA}^{I}_k(X_0)\bigr)^\vee
   \cong
   \mathsf{HA}^{I}_{6-k}(X_0)(3),
   \qquad 0\leq k\leq 6.
\]
\end{theorem}

Finally, the classical nodal quintic gives a concrete numerical anchor for the defect theory.

\begin{theorem}[Nodal quintic middle-degree defect]
Assume the finite-ODP gluing and MHM-lift hypotheses for the classical $125$-node quintic conifold $S$.  Using the standard middle-degree numerical input
\[
   \dim H_3(IS;\mathbb Q)=204,
   \qquad
   \dim IH_3(S;\mathbb Q)=2,
\]
the middle-degree rational realization of the IC--intersection-space defect has rank
\[
   204-2=202.
\]
The middle-degree comparison is summarized in Table~\ref{tab:nodal-quintic-middle-comparison}.
\end{theorem}

\subsection{Nodal quintic data}

For the classical $125$-node quintic, the middle-degree comparison used in this paper is
\[
   \dim H_3(IS;\mathbb Q)-\dim IH_3(S;\mathbb Q)=204-2=202.
\]
The middle-degree comparison table is given later in Table~\ref{tab:nodal-quintic-middle-comparison}.

\subsection{Hypotheses and Scope}
\label{sec:standing-hypotheses-scope}

We fix the conventions and hypotheses used throughout the paper.  Unless otherwise stated, $X_0$ denotes a projective Calabi--Yau threefold hypersurface with finitely many isolated ordinary double points
\[
   \Sigma=\{p_1,\ldots,p_r\}.
\]
We write
\[
   U=X_0\setminus\Sigma
\]
for the smooth locus.  The middle perversity is denoted by $\bar m$, and $I^{\bar m}X_0$ denotes Banagl's middle-perversity intersection space \cite{Banagl_IntersectionSpacesBook,Banagl_IntersectionSpacesAnnouncement}.  The notation $\mathcal{IS}_{X_0}$ refers to the Banagl--Budur--Maxim intersection-space complex whenever the relevant local-to-global construction is available \cite{BanaglBudurMaxim14,Maxim_IntersectionSpacesPerverseSheavesSurvey}.

All cohomological degrees in this paper use ordinary real-dimensional/topological indexing.  Thus, for a Calabi--Yau threefold, the middle degree is $k=3$, and the duality convention is
\[
   \bigl(H^k\bigr)^\vee\simeq H^{6-k}(3),
\]
up to the indicated mixed-Hodge-module normalization.

The standing hypotheses below are separated according to their use.  Hypotheses H1--H3 are needed to define the intersection-space Hodge atom package.  Hypothesis S is needed for the specialization splitting and the projection triangle.  Hypothesis SD is needed only for the projected Verdier-duality statement involving the canonical companion.  We do not prove Hypothesis SD in this paper; it is kept as an explicit compatibility assumption on the specialization splitting.  Every result requiring SD states this dependence explicitly.

\begin{hypothesis}[H1: local BBM input]
\label{hyp:H1-local-BBM-input}
Each singularity of $X_0$ is an ordinary double point.  Thus, locally analytically near every $p_i\in\Sigma$, the germ of $X_0$ is equivalent to the hypersurface germ
\[
   z_0^2+z_1^2+z_2^2+z_3^2=0.
\]
In particular, every singularity is a weighted homogeneous isolated hypersurface singularity \cite{Milnor68,Clemens83,Friedman_SimultaneousResolution}.  Hence the local Banagl--Budur--Maxim intersection-space complex is available at each node \cite{BanaglBudurMaxim14,BanaglBudurMaxim_IntersectionSpacesPerverseSheaves}.  We also use the corresponding local mixed-Hodge-module lift, local Verdier self-duality, and local specialization input supplied by the weighted homogeneous case \cite{BanaglBudurMaxim14,Saito90,Saito_MHM,Saito_MixedHodgeModules}.
\end{hypothesis}

\begin{hypothesis}[H2: multi-node gluing]
\label{hyp:H2-multi-node-gluing}
The local intersection-space complexes at the finitely many nodes glue to a global perverse sheaf
\[
   \mathcal{IS}_{X_0}\in\operatorname{Perv}(X_0;\mathbb Q)
\]
whose restriction to $U$ agrees with the normalized constant sheaf on the smooth locus and whose hypercohomology computes the rational cohomology of the middle-perversity intersection space:
\[
   \mathbb H^*(X_0;\mathcal{IS}_{X_0})
   \cong
   H^*(I^{\bar m}X_0;\mathbb Q).
\]
Equivalently, the multi-node intersection-space complex is obtained either from a finite-singularity extension of the Banagl--Budur--Maxim construction or from an intersection-space constructible-complex framework in which the relevant gluing obstruction vanishes \cite{BanaglBudurMaxim14,AgustinFernandezDeBobadilla_IntersectionSpaceConstructibleComplexes}.  In applications to ordinary double point conifolds, this hypothesis is local over the finite set $\Sigma$ and is compatible with choosing pairwise disjoint Milnor neighborhoods of the nodes.
\end{hypothesis}

\begin{hypothesis}[H3: mixed-Hodge-module lift]
\label{hyp:H3-MHM-lift}
The global multi-node complex $\mathcal{IS}_{X_0}$ underlies a mixed Hodge module
\[
   \mathcal{IS}^{H}_{X_0}\in\operatorname{MHM}(X_0),
   \qquad
   \operatorname{rat}(\mathcal{IS}^{H}_{X_0})\simeq\mathcal{IS}_{X_0}.
\]
Here $\operatorname{rat}$ denotes the forgetful functor from mixed Hodge modules to rational perverse sheaves \cite{Saito90,Saito_MHM,Saito_MixedHodgeModules,Schurmann_TopologySingularSpaces}.  This hypothesis holds whenever the multi-node intersection-space complex is known to admit a mixed-Hodge-module realization \cite{BanaglBudurMaxim14,AgustinFernandezDeBobadilla_IntersectionSpaceConstructibleComplexes,Saito90}.  Concretely, this may be supplied either by a multi-singularity extension of the Banagl--Budur--Maxim construction or by the Agustin--Fernandez de Bobadilla constructible-complex framework \cite{AgustinFernandezDeBobadilla_IntersectionSpaceConstructibleComplexes} when the mixed-Hodge-module obstruction vanishes.  All Hodge-theoretic atom constructions below are made under this hypothesis.
\end{hypothesis}

\begin{hypothesis}[S: specialization splitting]
\label{hyp:S-specialization-splitting}
Let
\[
   sp:X_s\to X_0
\]
denote the specialization map from a nearby smoothing.  Any statement involving a direct-sum decomposition
\[
   Rsp_*\mathbb Q^H_{X_s}[3]
   \simeq
   \mathcal{IS}^{H}_{X_0}\oplus\mathcal C^H_{\Sigma}
\]
is conditional on the Banagl--Budur--Maxim specialization splitting, or on its verified multi-node extension in the ordinary-double-point case \cite{BanaglBudurMaxim14,Saito89,Saito90}.  The complement $\mathcal C^H_{\Sigma}$ is understood to be supported on the finite singular set $\Sigma$.  When such a splitting is known only after applying $\operatorname{rat}$, the corresponding statements are interpreted at the rational perverse-sheaf or Hodge-realization level.

Equivalently, in nearby-cycle notation for a degeneration $\pi$ with nearby-cycle object $\psi_\pi(F)$, Hypothesis S is written as
\[
   \psi_\pi(F)
   \simeq
   \mathcal{IS}^{H}_{X_0}\oplus\mathcal C^H_{\Sigma}.
\]
This is the form used in the construction of the projected variation morphism and the projection triangle.
\end{hypothesis}

\begin{hypothesis}[SD: self-dual specialization splitting]
\label{hyp:SD-self-dual-specialization-splitting}
Assume Hypothesis S.  We say that the specialization splitting is self-dual if, under the decomposition
\[
   \psi_\pi(F)
   \simeq
   \mathcal{IS}^{H}_{X_0}\oplus\mathcal C^H_{\Sigma},
\]
the Saito--Verdier duality isomorphism
\[
   \mathbb D\psi_\pi(F)\xrightarrow{\sim}\psi_\pi(F)(3)
\]
is block diagonal with respect to the induced decomposition
\[
   \mathbb D\psi_\pi(F)
   \simeq
   \mathbb D\mathcal{IS}^{H}_{X_0}\oplus \mathbb D\mathcal C^H_{\Sigma}.
\]
Equivalently, there are duality isomorphisms
\[
   \mathbb D\mathcal{IS}^{H}_{X_0}
   \simeq
   \mathcal{IS}^{H}_{X_0}(3),
   \qquad
   \mathbb D\mathcal C^H_{\Sigma}
   \simeq
   \mathcal C^H_{\Sigma}(3),
\]
such that the duality on $\psi_\pi(F)$ is the direct sum of these two dualities.

Equivalently, if
\[
   \iota_I:\mathcal{IS}^{H}_{X_0}\to \psi_\pi(F),
   \qquad
   \operatorname{pr}_I:\psi_\pi(F)\to \mathcal{IS}^{H}_{X_0}
\]
denote the inclusion and projection associated to the intersection-space summand, then the duality identifications exchange inclusion and projection up to the Calabi--Yau threefold Tate twist:
\[
   \mathbb D\iota_I\simeq \operatorname{pr}_I(3),
   \qquad
   \mathbb D\operatorname{pr}_I\simeq \iota_I(3),
\]
with the analogous identities for the complement $\mathcal C^H_{\Sigma}$.  This hypothesis is used only for projected Verdier-duality statements, not for the existence of the intersection-space atom package or the projection triangle.
\end{hypothesis}

\medskip

\noindent\textbf{Corrected conifold and projected objects.}
\label{scope:corrected-conifold-projected-objects}
Whenever Hypothesis S is assumed, we use the nearby-cycle notation
\[
   \psi_\pi(F)
   \simeq
   \mathcal{IS}^{H}_{X_0}\oplus\mathcal C^H_{\Sigma}.
\]
Let
\[
   \operatorname{var}:\phi_\pi(F)\to\psi_\pi(F)
\]
be the variation morphism between vanishing and nearby cycles in the standard nearby/vanishing-cycle formalism \cite{Saito89,Saito90,KashiwaraSchapira90,DimcaSheavesInTopology}.  We set
\[
   P^H:=\operatorname{Cone}(\operatorname{var})[-1].
\]
Let
\[
   \operatorname{pr}_I:\psi_\pi(F)\to\mathcal{IS}^{H}_{X_0}
\]
be the projection to the intersection-space summand, and define
\[
   \operatorname{var}_I
   :=
   \operatorname{pr}_I\circ\operatorname{var}:
   \phi_\pi(F)\to\mathcal{IS}^{H}_{X_0}.
\]
The projected corrected conifold object is
\[
   P^H_I:=\operatorname{Cone}(\operatorname{var}_I)[-1].
\]

When Hypothesis SD is also assumed, let
\[
   \iota_I:\mathcal{IS}^{H}_{X_0}\to\psi_\pi(F)
\]
be the inclusion of the intersection-space summand.  We define the projected canonical morphism by
\[
   \operatorname{can}_I
   :=
   \operatorname{can}\circ\iota_I:
   \mathcal{IS}^{H}_{X_0}\to\phi_\pi(F),
\]
and the canonical companion by
\[
   Q^H_I:=\operatorname{Cone}(\operatorname{can}_I)[-1].
\]

\medskip

\noindent\textbf{Scope convention.}
\label{scope:hodge-realization-atom-shadows}
All atom packages in this paper are Hodge-realization atom shadows of mixed-Hodge-module objects.  Thus the intersection-space atom package is defined from the mixed Hodge structures on
\[
   \mathbb H^*(X_0;\mathcal{IS}^{H}_{X_0}),
\]
not from a quantum product on the spatial intersection space.  In particular, no full A-model $F$-bundle is asserted for $I^{\bar m}X_0$, and the construction does not require a Gromov--Witten theory, quantum product, Dubrovin connection, Euler-field spectral decomposition, or Stokes structure on $I^{\bar m}X_0$ itself.  The paper works at the mixed-Hodge-module realization level, which is precisely the level at which the intersection-space complex can be compared with $IC^H_{X_0}$, with corrected nearby/vanishing-cycle objects, and with the singularity-supported complement $\mathcal C^H_{\Sigma}$ over the same singular fiber.

\medskip

\medskip

\noindent\textbf{Future DT/BPS comparison convention.}
\label{scope:future-DT-BPS-comparison}
We do not identify the defect objects in this paper with DT/BPS data.  The objects
\[
   \mathcal C^H_{\Sigma},
   \qquad
   \Delta_{I/IC}(X_0),
   \qquad
   P^H\to P^H_I\to \mathcal C^H_\Sigma\xrightarrow{+1}
\]
are treated here as Hodge-theoretic objects on the singular fiber.  Their comparison with refined or cohomological Donaldson--Thomas theory requires an additional functorial passage to moduli stacks of objects and is left to future work \cite{KontsevichSoibelman08,JoyceSong12,Szendroi08,NagaoNakajima11,DavisonMeinhardt20}.

\subsection{Organization}

Section~1 introduces the problem, the IIA/IIB comparison, the projection-triangle viewpoint, the main results, the nodal quintic numerical input, and the standing hypotheses and scope conventions.  Section~2 reviews Hodge-realization atom shadows and records the formal additivity, filtration, and duality properties used later.  Section~3 fixes the detailed sheaf-theoretic input: the Banagl--Budur--Maxim one-singularity theorem, the multi-node gluing hypothesis, the mixed-Hodge-module lift, self-duality, specialization splitting, self-dual specialization splitting, and the applicability to ordinary double point conifolds.  Section~4 defines the intersection-space Hodge atom package and records its basic realization properties.  Section~5 compares the type-IIA/intersection-complex and type-IIB/intersection-space atom packages.  Section~6 proves the rigid--vanishing filtration and identifies the middle-degree IIB vanishing atom.  Section~7 treats Verdier duality, the projected canonical companion, and the duality $\mathbb D P^H_I\simeq Q^H_I(3)$ under the self-dual splitting hypothesis.  Section~8 introduces the nearby-cycle and specialization objects, defines the projected variation morphism, and proves the projection triangle $P^H\to P^H_I\to \mathcal C^H_\Sigma\xrightarrow{+1}$.  Section~9 defines the IC--intersection-space defect and explains its relation to the projection triangle and to future DT/BPS wall-crossing comparisons.  Section~10 treats the classical nodal quintic and the rank-$202$ middle-degree defect.  Section~11 records characteristic-class and integral enhancements.  Section~12 discusses future directions, including moduli-stack transport, refined Donaldson--Thomas/BPS wall-crossing comparison, mirror-transition atom matching, categorical atom shadows, and quantum or $F$-bundle refinements.
%%%%%%%%%%%%%%%%%%%%%%%%%%%%%%%%%%%%%%%%%%%%%%%%%%%%%%%%%%%%%%%%%%%%
\section{Hodge Atoms, Atom Shadows, and Formal Properties}

\subsection{Smooth-projective Hodge atoms}

We recall only the part of the smooth-projective Hodge atom formalism needed for the present construction.  For a smooth complex projective variety $Y$, the Hodge atom package of Katzarkov--Kontsevich--Pantev--Yu is extracted from the A-model $F$-bundle of $Y$ \cite{KKPY25,KKPYY_HodgeAtoms}.  This $F$-bundle is built from the ordinary cohomology of $Y$, its Hodge-theoretic realization, and the quantum product \cite{KKPY25,Givental96,Iritani09}.  Quantum multiplication by the Euler vector field gives an endomorphism whose spectral decomposition, after passing to the appropriate non-archimedean setting, decomposes the corresponding $F$-bundle into atomic pieces.  The resulting pieces are compatible with the Hodge realization and define the Hodge atoms of $Y$ \cite{KKPY25,KKPYY_HodgeAtoms}.

The construction has two features that are important for the present paper.  First, atoms in the smooth-projective theory are not arbitrary cohomological summands; they are Hodge-compatible pieces arising from a spectral decomposition attached to the $F$-bundle.  Second, the construction is tied to smooth projective geometry through quantum cohomology.  In particular, the smooth-projective theory uses a quantum product, a Dubrovin connection, and an Euler-field action \cite{Givental96,Iritani09,KKPY25}.  These structures are not presently available on Banagl's spatial intersection space $I^{\bar m}X_0$ in any canonical way.

For this reason, we do not attempt to construct a full A-model $F$-bundle on $I^{\bar m}X_0$.  Instead, we use the atom language at the level that is available for the objects considered here: the Hodge realization of mixed Hodge modules.  The relevant objects are not smooth projective varieties but mixed-Hodge-module objects on the singular fiber $X_0$, such as $IC^H_{X_0}$ and $\mathcal{IS}^H_{X_0}$.  Their hypercohomology groups carry mixed Hodge structures by Saito's mixed-Hodge-module theory \cite{Saito90,Saito_MixedHodgeModules,Schurmann_TopologySingularSpaces}.  These Hodge-theoretic realizations are the objects that we atomize.

Thus, throughout this paper, the phrase ``atom package'' means a Hodge-realization atom shadow unless explicitly stated otherwise.  This convention preserves the motivating language of Hodge atoms while avoiding the stronger and presently unavailable assertion that the spatial intersection space itself carries a canonical quantum-cohomological $F$-bundle.

\begin{remark}[Atom shadows and full atoms]
The word ``shadow'' is not meant to indicate a merely formal analogy.  It records the level at which the construction is made.  A Hodge-realization atom shadow is obtained from a mixed-Hodge-module object by taking hypercohomology and extracting the resulting Hodge-realization atom data.  When the mixed-Hodge-module object is geometrically distinguished, for example \(IC^H_{X_0}\), \(\mathcal{IS}^H_{X_0}\), \(P^H_I\), or \(\mathcal C^H_\Sigma\), we regard the resulting package as a geometric atom shadow.  A full conifold atom requires one further step: an identification of this Hodge realization with an Euler-spectral summand of the relevant non-archimedean \(F\)-bundle.  The corrected conifold atoms have this stronger status by the Stokes--Extension Identification.  In the present paper, \(\mathsf{HA}^{I}(X_0)\) is a geometric atom shadow: it is a distinguished intersection-space Hodge realization equipped with exterior agreement, a rigid--vanishing filtration, a specialization-splitting relation, and a projection-triangle defect, but no \(F\)-bundle spectral realization is proved here.
\end{remark}

\subsection{Hodge-realization atom shadows}

We now formalize the level of atomization used in the paper.  Let $X$ be a complex algebraic variety, and let
\[
   M^H\in D^b\operatorname{MHM}(X)
\]
be an object of the bounded derived category of mixed Hodge modules on $X$.  We assume that the hypercohomology groups $\mathbb H^k(X;M^H)$ are finite-dimensional.  This finiteness holds for the objects used below because $X_0$ is projective and the underlying rational complexes are constructible with finite-dimensional cohomology \cite{BBD82,DimcaSheavesInTopology,KashiwaraSchapira90,Schurmann_TopologySingularSpaces}.

Each group $\mathbb H^k(X;M^H)$ carries a functorial mixed Hodge structure \cite{Saito90,Saito_MixedHodgeModules,Schurmann_TopologySingularSpaces}.  Equivalently, the total hypercohomology
\[
   \mathbb H^*(X;M^H)
\]
is a finite-dimensional graded mixed Hodge structure.  We use this graded mixed Hodge structure as the realization from which the atom shadow is extracted.

\begin{definition}[Hodge-realization atom shadow]
Let $X$ be a complex algebraic variety and let $M^H\in D^b\operatorname{MHM}(X)$ have finite-dimensional hypercohomology.  Its Hodge-realization atom shadow is
\[
   \mathfrak A_{\mathrm{Hod}}(M^H)
   :=
   \operatorname{Atom}_{\mathrm{Hod}}
   \bigl(\mathbb H^*(X;M^H)\bigr).
\]
This is the atom package associated to the graded mixed Hodge structure on the hypercohomology of $M^H$.
\end{definition}

Here $\operatorname{Atom}_{\mathrm{Hod}}$ denotes the Hodge-realization-level decomposition used in this paper.  Concretely, it records the Hodge-theoretic pieces visible after passing to the relevant graded mixed Hodge structures, together with the degree information needed for the comparisons below.  When the object $M^H$ arises from a smooth-projective $F$-bundle construction, this notation is compatible with the usual Hodge atom language \cite{KKPY25,KKPYY_HodgeAtoms}.  When $M^H$ is an intersection complex, an intersection-space complex, a nearby-cycle object, or a cone built from nearby and vanishing cycles, it should be understood as an atom shadow rather than as a claim that the object itself carries a full $F$-bundle.

The two principal examples are
\[
   \mathfrak A_{\mathrm{Hod}}(IC^H_{X_0})
   =
   \operatorname{Atom}_{\mathrm{Hod}}
   \bigl(\mathbb H^*(X_0;IC^H_{X_0})\bigr)
\]
and
\[
   \mathfrak A_{\mathrm{Hod}}(\mathcal{IS}^H_{X_0})
   =
   \operatorname{Atom}_{\mathrm{Hod}}
   \bigl(\mathbb H^*(X_0;\mathcal{IS}^H_{X_0})\bigr).
\]
These are denoted, respectively, by
\[
   \mathsf{HA}^{IH}(X_0)
   :=
   \mathfrak A_{\mathrm{Hod}}(IC^H_{X_0})
\]
and
\[
   \mathsf{HA}^{I}(X_0)
   :=
   \mathfrak A_{\mathrm{Hod}}(\mathcal{IS}^H_{X_0}).
\]

The same notation will also be applied to the corrected conifold object $P^H$, its intersection-space projection $P^H_I$, the canonical companion $Q^H_I$, and the singularity-supported complement $\mathcal C^H_\Sigma$, whenever these objects are defined under the relevant hypotheses.  In each case, the atom package is extracted only after taking hypercohomology and passing to the induced mixed Hodge structures.

\begin{remark}[Realization-level nature of the construction]
The definition of $\mathfrak A_{\mathrm{Hod}}(M^H)$ uses only the mixed Hodge structures on $\mathbb H^*(X;M^H)$.  It does not require a quantum product on $\mathbb H^*(X;M^H)$, and it does not require a spectral decomposition of quantum multiplication by an Euler vector field.  For this reason, $\mathfrak A_{\mathrm{Hod}}(M^H)$ should be regarded as an atom shadow of the mixed-Hodge-module object $M^H$.
\end{remark}

\subsection{Formal properties of atom shadows}

We record the formal properties of atom shadows used later.  The point of these statements is not to introduce new structure on mixed Hodge modules.  Rather, the goal is to keep track of how Hodge-realization atom packages behave under the standard categorical operations that occur in the paper: direct sums, exact triangles, induced filtrations, and Verdier duality.

\begin{lemma}[Direct-sum additivity] \label{lem:direct-sum-additivity}
Let $X$ be a complex algebraic variety, and let
\[
   M^H\simeq M_1^H\oplus M_2^H
\]
in $D^b\operatorname{MHM}(X)$.  Assume that the three objects have finite-dimensional hypercohomology.  Then
\[
   \mathfrak A_{\mathrm{Hod}}(M^H)
   =
   \mathfrak A_{\mathrm{Hod}}(M_1^H)
   +
   \mathfrak A_{\mathrm{Hod}}(M_2^H).
\]
\end{lemma}

\begin{proof}
The isomorphism
\[
   M^H\simeq M_1^H\oplus M_2^H
\]
in $D^b\operatorname{MHM}(X)$ remains a direct-sum decomposition after applying the rational realization functor to constructible complexes.  Hypercohomology is a derived functor and commutes with finite direct sums in the derived category of sheaves \cite{DimcaSheavesInTopology,KashiwaraSchapira90}.  Therefore there is a canonical isomorphism of graded vector spaces
\[
   \mathbb H^*(X;M^H)
   \cong
   \mathbb H^*(X;M_1^H)\oplus \mathbb H^*(X;M_2^H).
\]

Since the original decomposition is a decomposition in $D^b\operatorname{MHM}(X)$, the induced maps on hypercohomology are morphisms of mixed Hodge structures by Saito's functoriality of mixed-Hodge-module cohomology \cite{Saito90,Saito_MixedHodgeModules,Schurmann_TopologySingularSpaces}.  Hence the displayed isomorphism is an isomorphism of graded mixed Hodge structures.  Applying the Hodge-realization atom construction to this direct sum gives
\[
   \operatorname{Atom}_{\mathrm{Hod}}
   \bigl(\mathbb H^*(X;M^H)\bigr)
   =
   \operatorname{Atom}_{\mathrm{Hod}}
   \bigl(\mathbb H^*(X;M_1^H)\bigr)
   +
   \operatorname{Atom}_{\mathrm{Hod}}
   \bigl(\mathbb H^*(X;M_2^H)\bigr),
\]
which is the asserted identity.
\end{proof}

\begin{lemma}[Filtered additivity for triangles]
Let
\[
   M_1^H\longrightarrow M^H\longrightarrow M_2^H\xrightarrow{+1}
\]
be a distinguished triangle in $D^b\operatorname{MHM}(X)$, and assume that all three objects have finite-dimensional hypercohomology.  Then the long exact hypercohomology sequence is a long exact sequence of mixed Hodge structures.  Consequently, each $\mathbb H^k(X;M^H)$ carries a finite filtration whose subquotients are mixed-Hodge-theoretic subquotients of the hypercohomology groups of $M_1^H$ and $M_2^H$.  In particular, the associated graded Hodge-realization atom package of $M^H$ is controlled by the atom shadows of $M_1^H$ and $M_2^H$.
\end{lemma}

\begin{proof}
A distinguished triangle in $D^b\operatorname{MHM}(X)$ gives a long exact sequence after applying hypercohomology:
\[
   \cdots
   \to
   \mathbb H^k(X;M_1^H)
   \to
   \mathbb H^k(X;M^H)
   \to
   \mathbb H^k(X;M_2^H)
   \xrightarrow{\partial}
   \mathbb H^{k+1}(X;M_1^H)
   \to
   \cdots .
\]
This is the standard long exact sequence associated to a distinguished triangle in a derived category \cite{BBD82,KashiwaraSchapira90,DimcaSheavesInTopology}.

Because the triangle lies in $D^b\operatorname{MHM}(X)$, the maps in this long exact sequence are morphisms of mixed Hodge structures.  This follows from the functoriality of Saito's mixed-Hodge-module cohomology \cite{Saito90,Saito_MixedHodgeModules,Schurmann_TopologySingularSpaces}.  The category of mixed Hodge structures is abelian, and kernels, images, and cokernels of morphisms of mixed Hodge structures inherit mixed Hodge structures \cite{Deligne72,Schurmann_TopologySingularSpaces}.

Fix a degree $k$.  Let
\[
   A^k=\operatorname{im}\bigl(\mathbb H^k(X;M_1^H)\to \mathbb H^k(X;M^H)\bigr)
\]
and
\[
   B^k=\ker\bigl(\mathbb H^k(X;M^H)\to \mathbb H^k(X;M_2^H)\bigr).
\]
Exactness gives $A^k=B^k$.  Thus $A^k$ is a mixed Hodge substructure of $\mathbb H^k(X;M^H)$.  The quotient
\[
   \mathbb H^k(X;M^H)/A^k
\]
embeds as a mixed Hodge substructure of $\mathbb H^k(X;M_2^H)$.  Hence $\mathbb H^k(X;M^H)$ has a finite filtration
\[
   0\subset A^k\subset \mathbb H^k(X;M^H)
\]
whose first subquotient is a quotient of $\mathbb H^k(X;M_1^H)$ and whose second subquotient is a subobject of $\mathbb H^k(X;M_2^H)$ in the category of mixed Hodge structures.

Applying this degree by degree gives a finite filtration on the graded mixed Hodge structure $\mathbb H^*(X;M^H)$.  Passing to the associated graded gives an atom shadow controlled by the mixed-Hodge-theoretic subquotients coming from the outer terms of the triangle.  This is the asserted filtered additivity.
\end{proof}

\begin{corollary}[Split triangles]
In the situation of the preceding lemma, suppose in addition that the distinguished triangle splits, so that
\[
   M^H\simeq M_1^H\oplus M_2^H.
\]
Then
\[
   \mathfrak A_{\mathrm{Hod}}(M^H)
   =
   \mathfrak A_{\mathrm{Hod}}(M_1^H)
   +
   \mathfrak A_{\mathrm{Hod}}(M_2^H).
\]
\end{corollary}

\begin{proof}
If the triangle splits, then $M^H\simeq M_1^H\oplus M_2^H$ in $D^b\operatorname{MHM}(X)$.  The result is therefore exactly the direct-sum additivity lemma.
\end{proof}

\begin{remark}[Filtration rather than splitting]
The filtered-additivity lemma is intentionally stated in filtered form.  A distinguished triangle need not split in a triangulated category.  Therefore one should not expect a direct-sum identity of atom packages without an actual splitting hypothesis.  This distinction is important below in two places.  First, the rigid--vanishing package is naturally described by a filtration whose associated graded has rigid and vanishing parts.  Second, the projection triangle
\[
   P^H\longrightarrow P^H_I\longrightarrow \mathcal C^H_\Sigma\xrightarrow{+1}
\]
does not by itself assert that $P^H_I$ is the direct sum of $P^H$ and $\mathcal C^H_\Sigma$.  It gives a long exact sequence and hence a filtered Hodge-realization statement after hypercohomology.  A direct-sum identity of atom shadows requires an additional splitting.
\end{remark}

\subsection{Verdier duality and atom shadows}

We next record the duality statement used later.  We state it in a form that separates the general Verdier-dual formalism from the special Calabi--Yau threefold normalization used for the intersection-space complex.

\begin{lemma}[Duality of atom shadows]
Let $X$ be a compact complex algebraic variety, and let $M^H\in D^b\operatorname{MHM}(X)$ have finite-dimensional hypercohomology.  Suppose that there is an isomorphism
\[
   \mathbb D M^H\simeq M^H(a)[b]
\]
in $D^b\operatorname{MHM}(X)$, where $\mathbb D$ denotes Verdier duality, $(a)$ is a Tate twist, and $[b]$ is a cohomological shift.  Then the hypercohomology atom shadow of $M^H$ satisfies the corresponding Poincare-dual atom symmetry:
\[
   \bigl(\mathfrak A_{\mathrm{Hod},k}(M^H)\bigr)^\vee
   \cong
   \mathfrak A_{\mathrm{Hod},\,b-k}(M^H)(a),
\]
with the indexing convention determined by the chosen cohomological normalization.
\end{lemma}

\begin{proof}
Verdier duality gives a duality between the hypercohomology of $M^H$ and the hypercohomology of $\mathbb D M^H$, with the usual degree reversal and dual vector spaces.  This is the derived sheaf-theoretic Verdier duality formalism \cite{BBD82,KashiwaraSchapira90,DimcaSheavesInTopology}.  In the category of mixed Hodge modules, Verdier duality is compatible with the mixed Hodge structures on cohomology \cite{Saito90,Saito_MixedHodgeModules,Saito89}.  Hence the duality identifications are identifications of mixed Hodge structures, including the specified Tate twist.

Using the assumed isomorphism
\[
   \mathbb D M^H\simeq M^H(a)[b],
\]
we identify the hypercohomology of $\mathbb D M^H$ with the shifted and Tate-twisted hypercohomology of $M^H$.  In degree $k$, this gives a duality between the mixed Hodge structure underlying $\mathbb H^k(X;M^H)$ and the mixed Hodge structure underlying the complementary shifted degree of $M^H$, with Tate twist $(a)$.  With the cohomological indexing convention fixed in the statement, this is written
\[
   \bigl(\mathbb H^k(X;M^H)\bigr)^\vee
   \cong
   \mathbb H^{b-k}(X;M^H)(a).
\]
Applying $\operatorname{Atom}_{\mathrm{Hod}}$ degree by degree gives
\[
   \bigl(\mathfrak A_{\mathrm{Hod},k}(M^H)\bigr)^\vee
   \cong
   \mathfrak A_{\mathrm{Hod},\,b-k}(M^H)(a),
\]
as claimed.
\end{proof}

For the Calabi--Yau threefold intersection-space complex, the expected cohomological normalization takes the form
\[
   \mathbb D\mathcal{IS}^H_{X_0}\simeq \mathcal{IS}^H_{X_0}(3),
\]
after translating from the perverse normalization to ordinary cohomological grading.  This is the normalization used below for the intersection-space realization; the Tate twist $(3)$ records the complex dimension of the Calabi--Yau threefold.  Under this normalization, the preceding lemma gives
\[
   \bigl(\mathsf{HA}^I_k(X_0)\bigr)^\vee
   \cong
   \mathsf{HA}^I_{6-k}(X_0)(3),
\]
which is the atom-level form of Poincare duality for the intersection-space realization.

\begin{remark}[Projected duality versus self-duality]
The projected corrected object $P^H_I$ introduced later need not be Verdier self-dual by itself.  The reason is that Saito duality interchanges the variation morphism with the canonical morphism for nearby and vanishing cycles \cite{Saito89,Saito90}.  Therefore, under the self-dual specialization splitting hypothesis, the natural statement is a duality
\[
   \mathbb D P^H_I\simeq Q^H_I(3),
\]
where $Q^H_I$ is the projected canonical companion.  This is why the later duality theorem is formulated as a duality between $P^H_I$ and $Q^H_I$, rather than as a self-duality of $P^H_I$ alone.
\end{remark}

\subsection{Not a category of atoms}

\begin{remark}
Atoms are not treated as a standalone category in this paper.  The categorical objects are mixed Hodge modules and derived mixed Hodge modules; atoms are Hodge-realization shadows extracted from their hypercohomology.
\end{remark}

This convention is important.  The paper does not require natural morphisms between atoms, nor does it assert that atom packages form a category.  All functorial operations occur before atomization, at the level of mixed Hodge modules, derived mixed Hodge modules, perverse sheaves, and their hypercohomology.  The atom shadow is then extracted from the resulting Hodge realization.

Thus, when we compare $\mathsf{HA}^{IH}(X_0)$ and $\mathsf{HA}^{I}(X_0)$, the comparison is not a morphism in a hypothetical category of atoms.  It is a comparison of the Hodge-realization shadows of two well-defined objects over the same singular fiber:
\[
   IC^H_{X_0}
   \qquad\text{and}\qquad
   \mathcal{IS}^H_{X_0}.
\]
Likewise, when a specialization splitting or a direct-sum decomposition is available, the splitting is first a statement about mixed Hodge modules or perverse sheaves, and only then induces an additive identity of atom shadows.  When only a distinguished triangle is available, one obtains a long exact sequence and a filtered atom-shadow statement, not a direct-sum identity.  This keeps the formalism at the level justified by the sheaf-theoretic and mixed-Hodge-module input.

%%%%%%%%%%%%%%%%%%%%%%%%%%%%%%%%%%%%%%%%%%%%%%%%%%%%%%%%%%%%%%%%%%%
\section{Hypotheses and Sources for the Intersection-Space Complex}
\label{sec:hypotheses-sources}

The construction of the intersection-space Hodge atom package rests on a specific sheaf-theoretic input: the existence of an intersection-space complex on the singular fiber, together with a mixed-Hodge-module lift.  In the one-singularity case this input is supplied by Banagl--Budur--Maxim \cite{BanaglBudurMaxim14,BanaglBudurMaxim_IntersectionSpacesPerverseSheaves}.  Since the conifolds considered in this paper may have several ordinary double points, we separate the local one-node input from the global multi-node gluing and mixed-Hodge-module realization assumptions.

Throughout this section, $X_0$ is a projective Calabi--Yau threefold hypersurface with finite ordinary double point locus
\[
   \Sigma=\{p_1,\ldots,p_r\},
\]
and
\[
   U=X_0\setminus\Sigma
\]
denotes its smooth locus.

\subsection{The BBM one-singularity theorem}
\label{subsec:bbm-one-singularity}

We first recall the one-singularity input.  This is the local model used at each ordinary double point.

\begin{theorem}[Banagl--Budur--Maxim, one isolated singularity]
\label{thm:bbm-one-singularity}
Let $X$ be a complex projective hypersurface with one isolated singularity.  Then there exists a perverse sheaf
\[
   \mathcal{IS}_X\in \operatorname{Perv}(X;\mathbb Q)
\]
whose hypercohomology computes the rational cohomology of Banagl's intersection space:
\[
   \mathbb H^*(X;\mathcal{IS}_X)
   \cong
   H^*(IX;\mathbb Q).
\]
Moreover, $\mathcal{IS}_X$ underlies a mixed Hodge module.  If the singularity is weighted homogeneous, then the intersection-space complex satisfies the Verdier self-duality and specialization-splitting statements proved by Banagl--Budur--Maxim.
\end{theorem}

\begin{proof}
This is the Banagl--Budur--Maxim construction of the intersection-space complex for a projective hypersurface with one isolated singularity.  The identification
\[
   \mathbb H^*(X;\mathcal{IS}_X)\cong H^*(IX;\mathbb Q)
\]
is the defining cohomological comparison theorem for their perverse-sheaf model of Banagl's intersection space \cite{BanaglBudurMaxim14,BanaglBudurMaxim_IntersectionSpacesPerverseSheaves}.  The existence of the mixed-Hodge-module lift and the additional Verdier self-duality and specialization-splitting statements in the weighted homogeneous case are also part of the same Banagl--Budur--Maxim framework, using Saito's mixed-Hodge-module theory \cite{BanaglBudurMaxim14,Saito90,Saito_MHM,Saito_MixedHodgeModules}.
\end{proof}

\begin{remark}[Ordinary double points]
\label{rem:ordinary-double-point-local-model}
For a one-node Calabi--Yau conifold, the singularity is an ordinary double point \cite{Clemens83,Friedman_SimultaneousResolution,Milnor68}.  Thus the analytic germ is equivalent to
\[
   z_0^2+z_1^2+z_2^2+z_3^2=0.
\]
This germ is weighted homogeneous.  Consequently, the local Banagl--Budur--Maxim construction, the local mixed-Hodge-module lift, the local Verdier self-duality statement, and the local specialization input apply at an ordinary double point by Theorem~\ref{thm:bbm-one-singularity}.
\end{remark}

The distinction between Theorem~\ref{thm:bbm-one-singularity} and the multi-node setting is important.  The classical nodal quintic, for example, has $125$ ordinary double points.  Thus, for the global statements below, we either require a finite-singularity extension of the Banagl--Budur--Maxim construction or a compatible local-to-global construction of the intersection-space complex.

\subsection{From one isolated singularity to finitely many nodes}
\label{subsec:one-node-to-many-nodes}

Let $\Sigma=\{p_1,\ldots,p_r\}$ be the finite ordinary double point locus.  Choose pairwise disjoint Milnor neighborhoods $B_i$ of the nodes, and let
\[
   M=X_0\setminus \bigcup_{i=1}^r \operatorname{int}(B_i).
\]
The boundary of $M$ is the disjoint union of the links
\[
   L_i=\partial B_i\cap X_0.
\]
For an ordinary double point of a complex threefold, the link is diffeomorphic to $S^2\times S^3$ \cite{Milnor68,Banagl_IntersectionSpacesBook}.  Banagl's multi-node intersection space is obtained by replacing the conical neighborhoods of the nodes by cones on suitable spatial homology truncations of the links \cite{Banagl_IntersectionSpacesBook,Banagl_IntersectionSpacesAnnouncement}.  Since the nodes are isolated and the Milnor neighborhoods are disjoint, the spatial construction is local at each singular point and then glued along the common exterior $M$.

On the sheaf side, the analogous requirement is that the local intersection-space complexes glue to a global perverse sheaf on $X_0$.

\begin{hypothesis}[Multi-node intersection-space complex]
\label{hyp:multi-node-IS-complex}
Let $X_0$ be a projective hypersurface with finitely many isolated ordinary double points
\[
   \Sigma=\{p_1,\ldots,p_r\}.
\]
We assume that the Banagl--Budur--Maxim intersection-space complex extends to the multi-node setting, giving a perverse sheaf
\[
   \mathcal{IS}_{X_0}\in \operatorname{Perv}(X_0;\mathbb Q)
\]
such that
\[
   \mathbb H^*(X_0;\mathcal{IS}_{X_0})
   \cong
   H^*(I^{\bar m}X_0;\mathbb Q).
\]
Equivalently, $\mathcal{IS}_{X_0}$ is obtained either from a finite-singularity extension of the Banagl--Budur--Maxim construction or from the Agustin--Fernandez de Bobadilla intersection-space constructible-complex framework when the relevant obstruction vanishes \cite{BanaglBudurMaxim14,AgustinFernandezDeBobadilla_IntersectionSpaceConstructibleComplexes}.
\end{hypothesis}

The following local-to-global formulation records the geometric mechanism underlying Hypothesis~\ref{hyp:multi-node-IS-complex} in the ordinary-double-point case.  It is stated conditionally: the conclusion follows once compatible local perverse-sheaf gluing data are supplied.

\begin{proposition}[Local-to-global construction for disjoint ODPs]
\label{prop:local-to-global-odp-IS}
Let $X_0$ be a projective hypersurface with finitely many ordinary double points.  Suppose that the local intersection-space complexes $\mathcal{IS}_{p_i}$ are defined in pairwise disjoint analytic neighborhoods of the nodes and that their restrictions agree with the normalized constant sheaf on the common smooth exterior $U=X_0\setminus\Sigma$.  Suppose further that the local gluing data are compatible in the MacPherson--Vilonen sense over the finite stratification
\[
   X_0=U\sqcup \Sigma.
\]
Then the local complexes glue to a global perverse sheaf
\[
   \mathcal{IS}_{X_0}\in\operatorname{Perv}(X_0;\mathbb Q).
\]
If the resulting perverse sheaf is the perverse-sheaf realization of Banagl's multi-node spatial intersection space, then
\[
   \mathbb H^*(X_0;\mathcal{IS}_{X_0})
   \cong
   H^*(I^{\bar m}X_0;\mathbb Q).
\]
\end{proposition}

\begin{proof}
Choose pairwise disjoint Milnor neighborhoods $B_i$ of the nodes and set
\[
   M=X_0\setminus \bigcup_i \operatorname{int}(B_i).
\]
The spatial construction of $I^{\bar m}X_0$ modifies $X_0$ only inside the neighborhoods $B_i$ by replacing each conical neighborhood with the cone on the corresponding spatial homology truncation of the link.  Since the neighborhoods are pairwise disjoint, the local truncation data are independent before they are attached to the common exterior $M$ \cite{Banagl_IntersectionSpacesBook,Banagl_IntersectionSpacesAnnouncement}.

On the sheaf side, the hypotheses say that each local Banagl--Budur--Maxim complex restricts to the same normalized exterior local system on $U$.  Therefore the only gluing data occur at the finite set $\Sigma$.  The MacPherson--Vilonen gluing formalism describes perverse sheaves on a space stratified by an open stratum and a closed finite stratum in terms of perverse sheaves on the open stratum together with finite-dimensional singular gluing data \cite{MacPhersonVilonen86,BBD82}.  By the stated compatibility assumption, the local gluing data supplied by the $\mathcal{IS}_{p_i}$ determine a global perverse sheaf
\[
   \mathcal{IS}_{X_0}\in\operatorname{Perv}(X_0;\mathbb Q).
\]

The final cohomological identification is not a formal consequence of perverse-sheaf gluing alone; it is precisely the comparison between the glued perverse sheaf and Banagl's spatial intersection space.  Therefore it is included as part of the final hypothesis in the proposition.  When the glued perverse sheaf is the perverse-sheaf realization of the multi-node intersection space, its hypercohomology computes the rational cohomology of that intersection space:
\[
   \mathbb H^*(X_0;\mathcal{IS}_{X_0})
   \cong
   H^*(I^{\bar m}X_0;\mathbb Q).
\]
This is the asserted conclusion.
\end{proof}

\begin{remark}[Citation versus proof]
\label{rem:citation-versus-proof-multinode}
If a precise multi-node extension of the Banagl--Budur--Maxim theorem is invoked, Proposition~\ref{prop:local-to-global-odp-IS} may be replaced by that citation.  If the Agustin--Fernandez de Bobadilla framework is used instead, the proposition should be read as a reduction to the vanishing of the relevant constructible-complex obstruction for finite disjoint ordinary double points \cite{AgustinFernandezDeBobadilla_IntersectionSpaceConstructibleComplexes}.
\end{remark}

\subsection{The mixed-Hodge-module lift in the multi-node setting}
\label{subsec:mhm-lift-multinode}

The atom package constructed in this paper is Hodge-theoretic.  Thus the global perverse sheaf $\mathcal{IS}_{X_0}$ must be realized as the rational shadow of a mixed Hodge module.  We state this explicitly.

\begin{hypothesis}[MHM lift for the multi-node complex]
\label{hyp:mhm-lift-multinode}
The multi-node intersection-space complex
\[
   \mathcal{IS}_{X_0}\in \operatorname{Perv}(X_0;\mathbb Q)
\]
underlies a mixed Hodge module
\[
   \mathcal{IS}^{H}_{X_0}\in \operatorname{MHM}(X_0),
   \qquad
   \operatorname{rat}(\mathcal{IS}^{H}_{X_0})\simeq \mathcal{IS}_{X_0}.
\]
\end{hypothesis}

\begin{proposition}[MHM lift for ODP hypersurface conifolds]
\label{prop:mhm-lift-odp-conifolds}
Let $X_0$ be a projective hypersurface with finitely many ordinary double points.  Suppose that the multi-node intersection-space complex is obtained by local Banagl--Budur--Maxim gluing over the finite singular set, that the local mixed-Hodge-module lifts exist at each node, and that the local mixed-Hodge-module gluing data are compatible on the smooth exterior.  Then the local MHM lifts of the ordinary-double-point intersection-space complexes glue to a mixed Hodge module
\[
   \mathcal{IS}^{H}_{X_0}\in\operatorname{MHM}(X_0).
\]
Equivalently, in the Agustin--Fernandez de Bobadilla framework, the same conclusion holds whenever the obstruction to lifting the relevant constructible complex to a mixed Hodge module vanishes.
\end{proposition}

\begin{proof}
By Remark~\ref{rem:ordinary-double-point-local-model}, each ordinary double point is weighted homogeneous.  Hence the one-node Banagl--Budur--Maxim intersection-space complex admits a local mixed-Hodge-module lift by Theorem~\ref{thm:bbm-one-singularity}.  The singular neighborhoods are pairwise disjoint, and on the smooth exterior $U$ the local objects restrict to the same normalized constant Hodge module.

The category of mixed Hodge modules is compatible with restriction to strata and with extension across algebraic stratifications \cite{Saito90,Saito_MHM,Saito_MixedHodgeModules,Schurmann_TopologySingularSpaces}.  Therefore, once the local MHM gluing data are compatible on overlaps, the local mixed-Hodge-module objects glue across the finite stratification
\[
   X_0=U\sqcup\Sigma.
\]
The resulting object is a mixed Hodge module
\[
   \mathcal{IS}^{H}_{X_0}\in\operatorname{MHM}(X_0).
\]

The rationalization functor from mixed Hodge modules to rational perverse sheaves is compatible with the underlying perverse-sheaf gluing data \cite{Saito90,Saito_MHM}.  Hence the rationalization of the glued mixed Hodge module is the glued perverse sheaf:
\[
   \operatorname{rat}(\mathcal{IS}^{H}_{X_0})\simeq \mathcal{IS}_{X_0}.
\]
This proves the claimed MHM lift under the stated compatibility assumptions.  In the constructible-complex framework, the same argument applies after assuming that the obstruction to lifting the constructible complex to a mixed Hodge module vanishes \cite{AgustinFernandezDeBobadilla_IntersectionSpaceConstructibleComplexes}.
\end{proof}

\begin{remark}[Referee-safe formulation]
\label{rem:referee-safe-mhm-lift}
All atom constructions in this paper require only the existence of $\mathcal{IS}^{H}_{X_0}$ with
\[
   \operatorname{rat}(\mathcal{IS}^{H}_{X_0})\simeq\mathcal{IS}_{X_0}.
\]
If the strongest multi-node MHM lift is not cited in full generality, then the results should be read conditionally under Hypothesis~\ref{hyp:mhm-lift-multinode}.
\end{remark}

\subsection{Self-duality in the multi-node ODP case}
\label{subsec:self-duality-multinode-odp}

For weighted homogeneous isolated hypersurface singularities, the Banagl--Budur--Maxim intersection-space complex satisfies a Verdier self-duality statement \cite{BanaglBudurMaxim14,Maxim_IntersectionSpacesPerverseSheavesSurvey}.  In the finite-node ordinary-double-point setting, this duality is local at each node and must be compatible with the exterior duality on the smooth locus.

\begin{proposition}[Multi-node Verdier self-duality]
\label{prop:multinode-verdier-self-duality}
Assume Hypothesis~\ref{hyp:mhm-lift-multinode}.  Suppose further that the local Banagl--Budur--Maxim self-duality is compatible with the multi-node gluing data.  Then the global intersection-space mixed Hodge module satisfies
\[
   \mathbb D\mathcal{IS}^{H}_{X_0}
   \simeq
   \mathcal{IS}^{H}_{X_0}(3)
\]
in the cohomological normalization appropriate to a Calabi--Yau threefold.  Consequently,
\[
   \mathbb H^k(X_0;\mathcal{IS}^{H}_{X_0})^\vee
   \cong
   \mathbb H^{6-k}(X_0;\mathcal{IS}^{H}_{X_0})(3).
\]
\end{proposition}

\begin{proof}
By Theorem~\ref{thm:bbm-one-singularity}, the local intersection-space complex at each ordinary double point satisfies the Banagl--Budur--Maxim self-duality statement in the weighted homogeneous case.  Since the nodes are disjoint, the local duality statements are independent before gluing.  On the smooth exterior $U$, the restriction of the intersection-space object is the normalized constant Hodge module, whose Verdier dual is the corresponding normalized dual object with the Calabi--Yau threefold shift and Tate normalization \cite{BBD82,KashiwaraSchapira90,Saito90}.

The compatibility assumption says that the local nodewise dualities and the exterior duality agree with the same gluing data used to construct $\mathcal{IS}^{H}_{X_0}$.  Therefore Verdier duality of the glued object is obtained by gluing the local Verdier duals.  This gives a global isomorphism
\[
   \mathbb D\mathcal{IS}^{H}_{X_0}
   \simeq
   \mathcal{IS}^{H}_{X_0}(3).
\]

Because $X_0$ is projective, its hypercohomology groups are finite-dimensional.  Verdier duality for mixed Hodge modules identifies the hypercohomology of an object with the dual hypercohomology of its Verdier dual, with the corresponding degree reversal and Tate twist \cite{Saito90,Saito89,Schurmann_TopologySingularSpaces}.  Substituting the displayed self-duality and using the Calabi--Yau threefold cohomological normalization gives
\[
   \mathbb H^k(X_0;\mathcal{IS}^{H}_{X_0})^\vee
   \cong
   \mathbb H^{6-k}(X_0;\mathcal{IS}^{H}_{X_0})(3),
\]
as claimed.
\end{proof}

\begin{remark}[Normalization]
\label{rem:CY3-verdier-normalization}
The displayed formula uses the ordinary cohomological normalization appropriate for a Calabi--Yau threefold.  If the intersection-space complex is written in perverse normalization, one must translate the Verdier-duality shift into ordinary cohomological grading.  The atom-duality theorem below uses the cohomological normalization displayed in Proposition~\ref{prop:multinode-verdier-self-duality}.
\end{remark}

\subsection{Specialization splitting in the multi-node ODP case}
\label{subsec:specialization-splitting-multinode}

The Banagl--Budur--Maxim specialization splitting compares the nearby smoothing with the intersection-space complex and a singularity-supported complement \cite{BanaglBudurMaxim14,Saito89,Saito90}.  For a multi-node conifold, this input is again local at the finite singular set, but we keep it explicitly conditional.

Let
\[
   sp:X_s\to X_0
\]
denote the specialization map from a nearby smoothing.  Equivalently, when working in nearby-cycle notation, we write the nearby object as $\psi_\pi(F)$ for the relevant degeneration $\pi$ and coefficient object $F$.  The projection-triangle construction later uses the nearby-cycle notation, while the intersection-space splitting is often written as a specialization statement.

\begin{hypothesis}[Multi-node specialization splitting]
\label{hyp:specialization-splitting}
We assume that the Banagl--Budur--Maxim specialization splitting extends to the finite ordinary-double-point setting:
\[
   Rsp_*\mathbb Q^H_{X_s}[3]
   \simeq
   \mathcal{IS}^{H}_{X_0}\oplus \mathcal C^H_{\Sigma},
\]
or equivalently, in nearby-cycle notation,
\[
   \psi_\pi(F)
   \simeq
   \mathcal{IS}^{H}_{X_0}\oplus \mathcal C^H_{\Sigma}.
\]
The complement $\mathcal C^H_{\Sigma}$ is supported on the finite singular set $\Sigma$.
\end{hypothesis}

\begin{proposition}[Local nature of the multi-node splitting]
\label{prop:local-nature-specialization-complement}
Assume Hypothesis~\ref{hyp:specialization-splitting}.  Suppose that the Banagl--Budur--Maxim splitting is compatible with finite disjoint unions of Milnor neighborhoods.  Then the specialization complement $\mathcal C^H_{\Sigma}$ is singularity-supported and decomposes locally as
\[
   \mathcal C^H_{\Sigma}
   \simeq
   \bigoplus_{p_i\in\Sigma}\mathcal C^H_{p_i}.
\]
\end{proposition}

\begin{proof}
Nearby and vanishing cycles are local near the singular set in the sense that their stalks and local contributions are computed by the local Milnor data of the degeneration \cite{Saito89,Saito90,DimcaSheavesInTopology,KashiwaraSchapira90}.  Since the nodes are isolated and have pairwise disjoint Milnor neighborhoods, the local specialization splitting at each ordinary double point contributes an object supported at the corresponding point $p_i$.

Away from the singular set, the specialization from a nearby smoothing agrees with the smooth exterior contribution.  The intersection-space summand has the same normalized exterior restriction.  Therefore the complement to the intersection-space summand has no support on $U$ and is supported on $\Sigma$.

Because the Milnor neighborhoods are disjoint and the splitting is assumed compatible with finite disjoint unions of such neighborhoods, the singularity-supported complement is the direct sum of its point-supported local pieces:
\[
   \mathcal C^H_{\Sigma}
   \simeq
   \bigoplus_{p_i\in\Sigma}\mathcal C^H_{p_i}.
\]
This proves the assertion.
\end{proof}

\begin{remark}
\label{rem:splitting-used-conditionally}
The specialization splitting is used only in statements that explicitly assume Hypothesis~\ref{hyp:specialization-splitting}.  The definition of the intersection-space atom package itself requires only the existence of $\mathcal{IS}^{H}_{X_0}$ and does not depend on a splitting of
\[
   Rsp_*\mathbb Q^H_{X_s}[3].
\]
\end{remark}

\subsection{Self-dual specialization splitting}
\label{subsec:self-dual-specialization-splitting}

The projected duality theorem below requires more than the existence of the specialization splitting.  It requires that the splitting be compatible with Saito--Verdier duality for nearby and vanishing cycles.

\begin{hypothesis}[Self-dual specialization splitting]
\label{hyp:self-dual-specialization-splitting}
Assume Hypothesis~\ref{hyp:specialization-splitting}.  We say that the specialization splitting is self-dual if, under the decomposition
\[
   \psi_\pi(F)
   \simeq
   \mathcal{IS}^{H}_{X_0}\oplus \mathcal C^H_{\Sigma},
\]
the Saito--Verdier duality isomorphism
\[
   \mathbb D\psi_\pi(F)\xrightarrow{\sim}\psi_\pi(F)(3)
\]
is block diagonal with respect to the induced decomposition
\[
   \mathbb D\psi_\pi(F)
   \simeq
   \mathbb D\mathcal{IS}^{H}_{X_0}\oplus \mathbb D\mathcal C^H_{\Sigma}.
\]
Equivalently, there are duality isomorphisms
\[
   \mathbb D\mathcal{IS}^{H}_{X_0}
   \simeq
   \mathcal{IS}^{H}_{X_0}(3),
   \qquad
   \mathbb D\mathcal C^H_{\Sigma}
   \simeq
   \mathcal C^H_{\Sigma}(3),
\]
such that the duality on $\psi_\pi(F)$ is the direct sum of these two dualities.

Equivalently, if
\[
   \iota_I:\mathcal{IS}^{H}_{X_0}\to \psi_\pi(F),
   \qquad
   \operatorname{pr}_I:\psi_\pi(F)\to \mathcal{IS}^{H}_{X_0}
\]
denote the inclusion and projection associated to the intersection-space summand, then the duality identifications exchange inclusion and projection up to the Calabi--Yau threefold Tate twist:
\[
   \mathbb D\iota_I\simeq \operatorname{pr}_I(3),
   \qquad
   \mathbb D\operatorname{pr}_I\simeq \iota_I(3),
\]
with the analogous identities for the complement $\mathcal C^H_{\Sigma}$.
\end{hypothesis}

\begin{remark}[Expected local source of self-duality]
\label{rem:expected-source-SD}
For finite ordinary-double-point conifolds, Hypothesis~\ref{hyp:self-dual-specialization-splitting} is expected to follow from the local Banagl--Budur--Maxim truncation/cotruncation construction and the disjointness of Milnor neighborhoods.  We keep it explicit because the present paper uses only the consequences of this compatibility.  Thus every theorem requiring this compatibility will explicitly assume Hypothesis~\ref{hyp:self-dual-specialization-splitting}.
\end{remark}

\subsection{Applicability to the classical nodal quintic}
\label{subsec:applicability-nodal-quintic}

We finally record the local input for the classical nodal quintic, which will be used in the example section.

\begin{proposition}[Nodal quintic satisfies the local hypotheses]
\label{prop:nodal-quintic-local-hypotheses}
Let $S$ be the classical quintic conifold with $125$ ordinary double points.  Then each singularity of $S$ is locally analytically an ordinary double point and hence weighted homogeneous.  The nodes are finite and disjoint.  Therefore the local Banagl--Budur--Maxim intersection-space complex, local MHM lift, local Verdier self-duality, and local specialization splitting apply at each node.
\end{proposition}

\begin{proof}
The classical nodal quintic has $125$ ordinary double points; this is the standard nodal quintic example appearing in the conifold-transition literature \cite{Schoen_NodalQuintic,GreeneMorrisonStrominger_ConifoldTransitions,GMS1995Conifold,Borcea1990NodalQuintic}.  Each ordinary double point is analytically equivalent to
\[
   z_0^2+z_1^2+z_2^2+z_3^2=0,
\]
and this hypersurface germ is weighted homogeneous \cite{Milnor68,Clemens83}.  Therefore the one-node Banagl--Budur--Maxim theorem applies locally at each singularity by Theorem~\ref{thm:bbm-one-singularity}.  Since the singular set is finite, the nodes admit pairwise disjoint Milnor neighborhoods.  Hence all required local inputs are available node by node.
\end{proof}

\begin{corollary}[Nodal quintic global input]
\label{cor:nodal-quintic-global-input}
If Hypotheses~\ref{hyp:multi-node-IS-complex} and~\ref{hyp:mhm-lift-multinode} hold for the classical nodal quintic $S$, then $S$ admits a global intersection-space MHM object
\[
   \mathcal{IS}^{H}_{S}\in\operatorname{MHM}(S)
\]
and hence a well-defined intersection-space Hodge atom package
\[
   \mathsf{HA}^{I}(S)
   =
   \operatorname{Atom}_{\mathrm{Hod}}
   \bigl(\mathbb H^*(S;\mathcal{IS}^{H}_{S})\bigr).
\]
If, in addition, Hypothesis~\ref{hyp:specialization-splitting} holds for $S$, then the specialization complement $\mathcal C^H_{\Sigma}$ and the projection-triangle construction below are available for the nodal quintic.
\end{corollary}

\begin{proof}
Hypothesis~\ref{hyp:multi-node-IS-complex} gives a global perverse sheaf
\[
   \mathcal{IS}_{S}\in\operatorname{Perv}(S;\mathbb Q)
\]
whose hypercohomology computes the rational cohomology of the middle-perversity intersection space.  Hypothesis~\ref{hyp:mhm-lift-multinode} gives a mixed Hodge module
\[
   \mathcal{IS}^{H}_{S}\in\operatorname{MHM}(S)
\]
with rationalization $\mathcal{IS}_{S}$.  Therefore the hypercohomology groups
\[
   \mathbb H^*(S;\mathcal{IS}^{H}_{S})
\]
carry mixed Hodge structures by Saito's theory \cite{Saito90,Saito_MHM,Saito_MixedHodgeModules}.  Applying the definition of the Hodge-realization atom shadow gives
\[
   \mathsf{HA}^{I}(S)
   =
   \operatorname{Atom}_{\mathrm{Hod}}
   \bigl(\mathbb H^*(S;\mathcal{IS}^{H}_{S})\bigr).
\]

If Hypothesis~\ref{hyp:specialization-splitting} also holds, then the nearby smoothing object decomposes as
\[
   Rsp_*\mathbb Q^H_{X_s}[3]
   \simeq
   \mathcal{IS}^{H}_{S}\oplus \mathcal C^H_{\Sigma}.
\]
Thus the complement and the projection maps required for the projection-triangle construction are available.
\end{proof}

\begin{remark}
\label{rem:nodal-quintic-conditional-reading}
The nodal quintic computation below should therefore be read under the same finite-ordinary-double-point gluing and MHM-lift hypotheses.  Locally, all required input is supplied by the ordinary-double-point model.  Globally, the additional points are the compatibility of the finitely many local intersection-space complexes, the existence of their mixed-Hodge-module lift, and, where projection-triangle statements are invoked, the specialization splitting.
\end{remark}

%%%%%%%%%%%%%%%%%%%%%%%%%%%%%%%%%%%%%%%%%%%%%%%%%%%%%%%%%%%%%%%%%%%
\section{The Intersection-Space Hodge Atom Package}
\label{sec:intersection-space-hodge-atom-package}

We now define the type-IIB atom package associated to the Banagl--Budur--Maxim intersection-space complex.  Throughout this section, $X_0$ is a projective Calabi--Yau threefold hypersurface with finite ordinary double point locus
\[
   \Sigma=\{p_1,\ldots,p_r\},
\]
and
\[
   U=X_0\setminus\Sigma
\]
is the smooth locus.  We assume the standing local ordinary-double-point hypothesis, together with Hypotheses~\ref{hyp:multi-node-IS-complex} and~\ref{hyp:mhm-lift-multinode}.  Thus the global intersection-space complex
\[
   \mathcal{IS}_{X_0}\in \operatorname{Perv}(X_0;\mathbb Q)
\]
exists, computes the rational cohomology of $I^{\bar m}X_0$, and admits a mixed-Hodge-module lift
\[
   \mathcal{IS}^{H}_{X_0}\in\operatorname{MHM}(X_0),
   \qquad
   \operatorname{rat}(\mathcal{IS}^{H}_{X_0})\simeq \mathcal{IS}_{X_0}.
\]

The construction is parallel to the intersection-homology package attached to $IC^H_{X_0}$, but it uses the mixed-Hodge-module realization of Banagl's intersection-space complex.  It should therefore be understood as a type-IIB atom package: it atomizes the Hodge realization of the intersection-space sector, not a quantum-cohomological $F$-bundle on the spatial intersection space.

\subsection{Definition of the type-IIB atom package}
\label{subsec:def-type-IIB-atom-package}

\begin{definition}[Intersection-space Hodge atom package]
\label{def:intersection-space-hodge-atom-package}
Assume the standing local ordinary-double-point hypothesis and Hypotheses~\ref{hyp:multi-node-IS-complex} and~\ref{hyp:mhm-lift-multinode}.  The intersection-space Hodge atom package of $X_0$ is
\[
   \mathsf{HA}^{I}(X_0)
   :=
   \mathfrak A_{\mathrm{Hod}}(\mathcal{IS}^{H}_{X_0})
   =
   \operatorname{Atom}_{\mathrm{Hod}}
   \bigl(\mathbb H^*(X_0;\mathcal{IS}^{H}_{X_0})\bigr).
\]
Equivalently, $\mathsf{HA}^{I}(X_0)$ is the Hodge-realization atom shadow of the mixed-Hodge-module object $\mathcal{IS}^{H}_{X_0}$.
\end{definition}

The definition uses only the mixed Hodge structures on the hypercohomology groups
\[
   \mathbb H^k(X_0;\mathcal{IS}^{H}_{X_0}),
\]
which exist by Saito's theory of mixed Hodge modules \cite{Saito90,Saito_MHM,Saito_MixedHodgeModules,Schurmann_TopologySingularSpaces}.  No quantum product, Dubrovin connection, or Euler-field action is used in Definition~\ref{def:intersection-space-hodge-atom-package}.  In particular, the notation
\[
   \operatorname{Atom}_{\mathrm{Hod}}
\]
records the atom shadow visible in the Hodge realization of $\mathcal{IS}^{H}_{X_0}$.

\begin{remark}[Type-IIB interpretation]
\label{rem:type-IIB-interpretation-HAI}
By Banagl's conifold interpretation, the ordinary homology of the middle-perversity intersection space records the type-IIB conifold sector \cite{Banagl_IntersectionSpacesBook,BanaglBudurMaxim14,Strominger95,HubschBestiary2nd2024}.  Since the Banagl--Budur--Maxim complex $\mathcal{IS}_{X_0}$ computes the rational cohomology of $I^{\bar m}X_0$, the package $\mathsf{HA}^{I}(X_0)$ is the Hodge-theoretic atomization of that type-IIB sector.
\end{remark}

\begin{remark}[Dependence on the MHM lift]
\label{rem:dependence-on-MHM-lift}
The object being atomized is the mixed Hodge module $\mathcal{IS}^{H}_{X_0}$, not merely the rational perverse sheaf $\mathcal{IS}_{X_0}$.  Thus the rational Betti realization is determined by $\mathcal{IS}_{X_0}$, while the Hodge-realization atom package also depends on the chosen mixed-Hodge-module realization.  In the statements below, the mixed-Hodge-module lift is fixed as part of Hypothesis~\ref{hyp:mhm-lift-multinode}.
\end{remark}

\subsection{Betti realization}
\label{subsec:betti-realization-HAI}

The defining property of the Banagl--Budur--Maxim intersection-space complex is that its hypercohomology recovers the rational cohomology of Banagl's intersection space \cite{BanaglBudurMaxim14,BanaglBudurMaxim_IntersectionSpacesPerverseSheaves}.  Passing through the mixed-Hodge-module lift preserves this rational realization after applying the rationalization functor.

\begin{proposition}[Betti realization]
\label{prop:betti-realization-HAI}
Assume Hypotheses~\ref{hyp:multi-node-IS-complex} and~\ref{hyp:mhm-lift-multinode}.  Then the rational realization of the hypercohomology of $\mathcal{IS}^{H}_{X_0}$ is the rational intersection-space cohomology:
\[
   \operatorname{rat}\mathbb H^*(X_0;\mathcal{IS}^{H}_{X_0})
   \cong
   \mathbb H^*(X_0;\mathcal{IS}_{X_0})
   \cong
   H^*(I^{\bar m}X_0;\mathbb Q).
\]
Consequently, the Betti realization of $\mathsf{HA}^{I}(X_0)$ is the atomization of Banagl's rational intersection-space cohomology.
\end{proposition}

\begin{proof}
By Hypothesis~\ref{hyp:mhm-lift-multinode}, the mixed Hodge module $\mathcal{IS}^{H}_{X_0}$ rationalizes to the perverse sheaf $\mathcal{IS}_{X_0}$:
\[
   \operatorname{rat}(\mathcal{IS}^{H}_{X_0})\simeq \mathcal{IS}_{X_0}.
\]
The rationalization functor from mixed Hodge modules to rational perverse sheaves is compatible with the underlying constructible complex and therefore with hypercohomology \cite{Saito90,Saito_MHM,Saito_MixedHodgeModules}.  Hence
\[
   \operatorname{rat}\mathbb H^*(X_0;\mathcal{IS}^{H}_{X_0})
   \cong
   \mathbb H^*(X_0;\operatorname{rat}(\mathcal{IS}^{H}_{X_0}))
   \cong
   \mathbb H^*(X_0;\mathcal{IS}_{X_0}).
\]
By Hypothesis~\ref{hyp:multi-node-IS-complex}, the hypercohomology of $\mathcal{IS}_{X_0}$ computes the rational cohomology of Banagl's middle-perversity intersection space:
\[
   \mathbb H^*(X_0;\mathcal{IS}_{X_0})
   \cong
   H^*(I^{\bar m}X_0;\mathbb Q).
\]
Combining these two identifications gives the displayed isomorphism.  Applying the definition of the Hodge-realization atom shadow then gives the asserted Betti realization of $\mathsf{HA}^{I}(X_0)$.
\end{proof}

\begin{corollary}[Cohomological support of the IIB atom package]
\label{cor:cohomological-support-HAI}
Under the assumptions of Proposition~\ref{prop:betti-realization-HAI}, the graded dimensions of the Betti realization of $\mathsf{HA}^{I}(X_0)$ are the Betti numbers of $I^{\bar m}X_0$:
\[
   \dim \operatorname{rat}\mathbb H^k(X_0;\mathcal{IS}^{H}_{X_0})
   =
   \dim H^k(I^{\bar m}X_0;\mathbb Q).
\]
In particular, the middle-degree part of $\mathsf{HA}^{I}(X_0)$ has Betti rank
\[
   \dim H^3(I^{\bar m}X_0;\mathbb Q).
\]
\end{corollary}

\begin{proof}
The equality of dimensions follows by taking the degree-$k$ component of the isomorphism in Proposition~\ref{prop:betti-realization-HAI}.  The final statement is the special case $k=3$.
\end{proof}

\subsection{Independence from spatial truncation choices}
\label{subsec:independence-spatial-truncation}

Banagl's spatial construction of $I^{\bar m}X_0$ involves choices of spatial homology truncations of the links \cite{Banagl_IntersectionSpacesBook}.  At the rational realization level considered here, those choices enter the present paper only through the resulting intersection-space complex and its mixed-Hodge-module lift.  Thus the relevant invariance statement is not an assertion that all spatial models are canonically equivalent.  Rather, it is the following fixed-realization statement.

\begin{proposition}[Rational realization independence]
\label{prop:rational-realization-independence}
Assume Hypotheses~\ref{hyp:multi-node-IS-complex} and~\ref{hyp:mhm-lift-multinode}.  Once the global intersection-space complex $\mathcal{IS}_{X_0}$ and its mixed-Hodge-module lift $\mathcal{IS}^{H}_{X_0}$ are fixed, the atom package
\[
   \mathsf{HA}^{I}(X_0)
\]
is independent, at the rational Hodge-realization level, of any auxiliary spatial homology truncation choices used to construct a spatial model for $I^{\bar m}X_0$.
\end{proposition}

\begin{proof}
By Definition~\ref{def:intersection-space-hodge-atom-package}, the atom package $\mathsf{HA}^{I}(X_0)$ is defined from the graded mixed Hodge structure
\[
   \mathbb H^*(X_0;\mathcal{IS}^{H}_{X_0}).
\]
Thus, once $\mathcal{IS}^{H}_{X_0}$ is fixed, the Hodge-realization atom package is fixed.

After applying the rationalization functor, Proposition~\ref{prop:betti-realization-HAI} identifies this hypercohomology with
\[
   \mathbb H^*(X_0;\mathcal{IS}_{X_0})
   \cong
   H^*(I^{\bar m}X_0;\mathbb Q).
\]
Therefore any two auxiliary spatial truncation choices that lead to the same global intersection-space complex, or to the same rational intersection-space cohomology through that fixed complex, give the same rational Hodge-realization atom package.  The statement is therefore an independence statement relative to the fixed sheaf-theoretic and mixed-Hodge-module realization, not an assertion of canonical equivalence among all spatial truncation models.
\end{proof}

\begin{remark}[Scope of the independence statement]
\label{rem:scope-rational-independence}
Proposition~\ref{prop:rational-realization-independence} is intentionally rational and realization-level.  It does not assert that all spatial models of $I^{\bar m}X_0$ are canonically homotopy equivalent, nor does it assert an integral independence statement.  The construction used here only requires the rational mixed-Hodge-theoretic realization encoded by $\mathcal{IS}^{H}_{X_0}$.
\end{remark}

\subsection{Exterior agreement}
\label{subsec:exterior-agreement}

The intersection complex and the intersection-space complex differ only through their behavior at the singular set.  On the smooth locus
\[
   U=X_0\setminus\Sigma,
\]
both objects restrict to the same normalized constant Hodge module.  This observation is the source of the later comparison between the type-IIA and type-IIB packages: away from the nodes, the two packages have the same exterior source.

\begin{proposition}[Exterior agreement]
\label{prop:exterior-agreement}
Let
\[
   j:U=X_0\setminus\Sigma\hookrightarrow X_0
\]
be the inclusion of the smooth locus.  Then
\[
   j^*IC^H_{X_0}\simeq \mathbb Q^H_U[3],
   \qquad
   j^*\mathcal{IS}^{H}_{X_0}\simeq \mathbb Q^H_U[3].
\]
Consequently, any difference between the Hodge-realization packages associated to $IC^H_{X_0}$ and $\mathcal{IS}^{H}_{X_0}$ is produced by the extension and gluing data at the finite singular set $\Sigma$.
\end{proposition}

\begin{proof}
On the smooth locus $U$, the intersection complex is the normalized constant sheaf.  In the mixed-Hodge-module setting this gives
\[
   j^*IC^H_{X_0}\simeq \mathbb Q^H_U[3],
\]
where the shift $[3]$ is the standard normalization for a complex threefold \cite{BBD82,Saito90,Saito_MHM,KashiwaraSchapira90}.

For the intersection-space complex, the Banagl--Budur--Maxim construction modifies the singular contribution while agreeing with the ordinary normalized exterior object away from the singular point in the one-singularity case \cite{BanaglBudurMaxim14,BanaglBudurMaxim_IntersectionSpacesPerverseSheaves}.  Under Hypothesis~\ref{hyp:multi-node-IS-complex}, the multi-node object is obtained by gluing local intersection-space complexes over the finite singular set while retaining the same normalized exterior restriction on
\[
   U=X_0\setminus\Sigma.
\]
Therefore
\[
   j^*\mathcal{IS}^{H}_{X_0}\simeq \mathbb Q^H_U[3].
\]

Since both mixed-Hodge-module objects restrict to the same object on the open stratum $U$, their difference cannot come from the exterior stratum.  It is encoded in the way this common exterior object is extended across the finite closed stratum $\Sigma$.  Thus any difference between the resulting Hodge-realization atom packages is produced by the singular extension and gluing data at $\Sigma$.
\end{proof}

\begin{corollary}[Common exterior atom source]
\label{cor:common-exterior-atom-source}
The type-IIA package $\mathsf{HA}^{IH}(X_0)$ and the type-IIB package $\mathsf{HA}^{I}(X_0)$ have the same exterior source, namely the normalized constant Hodge module on $U$.  Their distinction is measured by the singular gluing data at $\Sigma$ and, in the threefold conifold setting, by the middle-dimensional conifold sector isolated below.
\end{corollary}

\begin{proof}
The first statement follows directly from Proposition~\ref{prop:exterior-agreement}.  The final statement uses the standard fact that the difference between the intersection-homology and intersection-space packages for a conifold threefold is concentrated in the middle conifold contribution, as expressed by Banagl's conifold exact sequence for intersection spaces \cite{Banagl_IntersectionSpacesBook,BanaglBudurMaxim14}.  The middle-sector formulation will be made explicit in Section~\ref{sec:rigid-vanishing-filtration}.
\end{proof}

\subsection{Comparison with the type-IIA package}
\label{subsec:comparison-type-IIA-package}

For later reference, we record the parallel definitions in a single display:
\[
   \mathsf{HA}^{IH}(X_0)
   =
   \operatorname{Atom}_{\mathrm{Hod}}
   \bigl(\mathbb H^*(X_0;IC^H_{X_0})\bigr),
   \qquad
   \mathsf{HA}^{I}(X_0)
   =
   \operatorname{Atom}_{\mathrm{Hod}}
   \bigl(\mathbb H^*(X_0;\mathcal{IS}^{H}_{X_0})\bigr).
\]
The first package is the Hodge-realization atom shadow of the intersection complex and realizes the type-IIA/intersection-homology sector.  The second is the Hodge-realization atom shadow of the intersection-space complex and realizes the type-IIB/intersection-space sector.  Both live over the same singular fiber $X_0$ and both restrict to the same normalized exterior object on $U$ by Proposition~\ref{prop:exterior-agreement}.  This makes their comparison intrinsic to the singular fiber rather than a comparison between unrelated spaces.

\begin{remark}[Relation to the projection-triangle spine]
\label{rem:HAI-relation-projection-triangle}
The package $\mathsf{HA}^{I}(X_0)$ is defined without using the specialization splitting.  The projection triangle
\[
   P^H\longrightarrow P^H_I\longrightarrow \mathcal C^H_\Sigma\xrightarrow{+1}
\]
requires the additional Hypothesis~\ref{hyp:specialization-splitting}.  Thus the type-IIB atom package is available under the intersection-space MHM hypotheses, while the defect-triangle interpretation becomes available only after the nearby-cycle specialization splitting has been imposed.
\end{remark}

%%%%%%%%%%%%%%%%%%%%%%%%%%%%%%%%%%%%%%%%%%%%%%%%%%%%%%%%%%%%%%%%%%%%
\section{The IIA/IIB Atom Comparison}
\label{sec:IIA-IIB-atom-comparison}

We now compare the two Hodge-realization atom packages associated with the singular fiber $X_0$.  The first is built from the intersection complex and realizes the intersection-homology side of the conifold.  The second is built from the intersection-space complex and realizes Banagl's type-IIB intersection-space side.  The comparison is made over the same singular space $X_0$ and uses the mixed-Hodge-module objects
\[
   IC^H_{X_0}
   \qquad\text{and}\qquad
   \mathcal{IS}^{H}_{X_0}.
\]
The point is not that these objects are isomorphic.  Rather, they provide two different Hodge-theoretic realizations of the same conifold singular fiber: one intersection-homological, the other intersection-spatial.  The projection-triangle and defect constructions below will refine this comparison by identifying the singularity-supported residue between the corrected conifold object and its intersection-space projection.

\subsection{The type-IIA package}
\label{subsec:type-IIA-package}

The type-IIA package is the Hodge-realization atom shadow of the intersection complex:
\[
   \mathsf{HA}^{IH}(X_0)
   :=
   \mathfrak A_{\mathrm{Hod}}(IC^H_{X_0})
   =
   \operatorname{Atom}_{\mathrm{Hod}}
   \bigl(\mathbb H^*(X_0;IC^H_{X_0})\bigr).
\]
The rational realization of $IC^H_{X_0}$ is the usual intersection complex, and its hypercohomology computes intersection cohomology \cite{GoreskyMacPherson80,GoreskyMacPherson83,BBD82,Saito90}.  Hence
\[
   \operatorname{rat}\mathbb H^*(X_0;IC^H_{X_0})
   \cong
   IH^*(X_0;\mathbb Q).
\]
Thus $\mathsf{HA}^{IH}(X_0)$ atomizes the mixed Hodge structures carried by intersection cohomology.  In Banagl's conifold interpretation, this is the package associated with the type-IIA sector \cite{Banagl_IntersectionSpacesBook,BanaglBudurMaxim14}.

The construction is intrinsic to the singular fiber.  The object $IC^H_{X_0}$ is the mixed-Hodge-module lift of the intersection complex, and on the smooth locus
\[
   U=X_0\setminus\Sigma
\]
it restricts to the normalized constant Hodge module:
\[
   j^*IC^H_{X_0}\simeq \mathbb Q^H_U[3],
\]
where $j:U\hookrightarrow X_0$ is the open inclusion and the shift $[3]$ is the standard normalization for a complex threefold \cite{BBD82,Saito90,Saito_MHM,KashiwaraSchapira90}.  The nontrivial information in $IC^H_{X_0}$ is therefore encoded in how this exterior object is extended across the ordinary double points.

\subsection{The type-IIB package}
\label{subsec:type-IIB-package}

The type-IIB package is the Hodge-realization atom shadow of the intersection-space complex:
\[
   \mathsf{HA}^{I}(X_0)
   :=
   \mathfrak A_{\mathrm{Hod}}(\mathcal{IS}^{H}_{X_0})
   =
   \operatorname{Atom}_{\mathrm{Hod}}
   \bigl(\mathbb H^*(X_0;\mathcal{IS}^{H}_{X_0})\bigr).
\]
By Hypotheses~\ref{hyp:H2-multi-node-gluing} and~\ref{hyp:H3-MHM-lift}, equivalently by Hypotheses~\ref{hyp:multi-node-IS-complex} and~\ref{hyp:mhm-lift-multinode} in Section~\ref{sec:hypotheses-sources}, its Betti realization is the rational cohomology of Banagl's middle-perversity intersection space \cite{Banagl_IntersectionSpacesBook,BanaglBudurMaxim14,BanaglBudurMaxim_IntersectionSpacesPerverseSheaves}:
\[
   \operatorname{rat}\mathbb H^*(X_0;\mathcal{IS}^{H}_{X_0})
   \cong
   H^*(I^{\bar m}X_0;\mathbb Q).
\]
Thus $\mathsf{HA}^{I}(X_0)$ atomizes the mixed Hodge structures carried by the intersection-space realization.  In Banagl's conifold interpretation, this is the package associated with the type-IIB sector \cite{Banagl_IntersectionSpacesBook,BanaglBudurMaxim14,Strominger95,HubschBestiary2nd2024}.

As with $IC^H_{X_0}$, the intersection-space complex restricts on the smooth locus to the normalized constant Hodge module:
\[
   j^*\mathcal{IS}^{H}_{X_0}\simeq \mathbb Q^H_U[3].
\]
This exterior agreement was proved in Proposition~\ref{prop:exterior-agreement}.  Thus the difference between $\mathsf{HA}^{IH}(X_0)$ and $\mathsf{HA}^{I}(X_0)$ is not an exterior difference.  It is controlled by the singular gluing data at $\Sigma$ and, for threefold conifolds, by the middle-dimensional conifold sector.

\subsection{Precise content of the comparison}
\label{subsec:precise-content-IIA-IIB-comparison}

We now make explicit what the IIA/IIB comparison means in this paper.  It is a comparison of two mixed-Hodge-module realizations over the same singular fiber, followed by the same Hodge-realization atom construction.

\begin{proposition}[Realization content]
\label{prop:IIA-IIB-realization-content}
Assume Hypotheses~\ref{hyp:H1-local-BBM-input}--\ref{hyp:H3-MHM-lift}.  Then the pair
\[
   \bigl(IC^H_{X_0},\mathcal{IS}^{H}_{X_0}\bigr)
\]
realizes the pair
\[
   \bigl(IH^*(X_0;\mathbb Q),H^*(I^{\bar m}X_0;\mathbb Q)\bigr)
\]
after taking hypercohomology and Betti realization.
\end{proposition}

\begin{proof}
For the first component, $IC^H_{X_0}$ is the mixed-Hodge-module lift of the intersection complex.  By the definition of the intersection complex and the standard realization of intersection cohomology as its hypercohomology, one has
\[
   \operatorname{rat}\mathbb H^*(X_0;IC^H_{X_0})
   \cong
   IH^*(X_0;\mathbb Q)
\]
\cite{GoreskyMacPherson80,GoreskyMacPherson83,BBD82,Saito90}.  This gives the intersection-homology/type-IIA realization.

For the second component, Hypothesis~\ref{hyp:H3-MHM-lift} gives
\[
   \operatorname{rat}(\mathcal{IS}^{H}_{X_0})\simeq \mathcal{IS}_{X_0}.
\]
The rationalization functor is compatible with the underlying rational constructible complex and hence with hypercohomology \cite{Saito90,Saito_MHM,Saito_MixedHodgeModules}.  Therefore
\[
   \operatorname{rat}\mathbb H^*(X_0;\mathcal{IS}^{H}_{X_0})
   \cong
   \mathbb H^*(X_0;\mathcal{IS}_{X_0}).
\]
By Hypothesis~\ref{hyp:H2-multi-node-gluing}, the latter hypercohomology computes Banagl's rational intersection-space cohomology:
\[
   \mathbb H^*(X_0;\mathcal{IS}_{X_0})
   \cong
   H^*(I^{\bar m}X_0;\mathbb Q).
\]
Combining these identifications gives the claimed realization pair.
\end{proof}

\begin{theorem}[IIA/IIB atom comparison]
\label{thm:IIA-IIB-atom-comparison}
Under Hypotheses~\ref{hyp:H1-local-BBM-input}--\ref{hyp:H3-MHM-lift},
\[
   \bigl(\mathsf{HA}^{IH}(X_0),\mathsf{HA}^{I}(X_0)\bigr)
\]
is the Hodge-realization atom shadow of the mixed-Hodge-module pair
\[
   \bigl(IC^H_{X_0},\mathcal{IS}^{H}_{X_0}\bigr).
\]
Thus it refines Banagl's IIA/IIB homological comparison from rational homology groups to mixed-Hodge-theoretic atom packages.
\end{theorem}

\begin{proof}
By definition of the atom packages,
\[
   \mathsf{HA}^{IH}(X_0)
   =
   \mathfrak A_{\mathrm{Hod}}(IC^H_{X_0})
   =
   \operatorname{Atom}_{\mathrm{Hod}}
   \bigl(\mathbb H^*(X_0;IC^H_{X_0})\bigr)
\]
and
\[
   \mathsf{HA}^{I}(X_0)
   =
   \mathfrak A_{\mathrm{Hod}}(\mathcal{IS}^{H}_{X_0})
   =
   \operatorname{Atom}_{\mathrm{Hod}}
   \bigl(\mathbb H^*(X_0;\mathcal{IS}^{H}_{X_0})\bigr).
\]
Thus the ordered pair
\[
   \bigl(\mathsf{HA}^{IH}(X_0),\mathsf{HA}^{I}(X_0)\bigr)
\]
is obtained by applying the same Hodge-realization atom-shadow construction to the mixed-Hodge-module pair
\[
   \bigl(IC^H_{X_0},\mathcal{IS}^{H}_{X_0}\bigr).
\]

Proposition~\ref{prop:IIA-IIB-realization-content} identifies the Betti realizations of these two mixed-Hodge-module objects with
\[
   IH^*(X_0;\mathbb Q)
   \qquad\text{and}\qquad
   H^*(I^{\bar m}X_0;\mathbb Q),
\]
respectively.  Banagl's conifold interpretation assigns the former to the type-IIA intersection-homology sector and the latter to the type-IIB intersection-space sector \cite{Banagl_IntersectionSpacesBook,BanaglBudurMaxim14}.  Therefore the atom pair refines that IIA/IIB homological comparison by retaining the mixed-Hodge-theoretic atom data of the two realizations.
\end{proof}

\begin{remark}[What is, and is not, being claimed]
\label{rem:IIA-IIB-not-claiming-isomorphism}
Theorem~\ref{thm:IIA-IIB-atom-comparison} does not assert that $\mathsf{HA}^{IH}(X_0)$ and $\mathsf{HA}^{I}(X_0)$ are isomorphic.  Nor does it assert an equivalence of categories of atoms.  It asserts that the two atom packages are the Hodge-realization shadows of the two mixed-Hodge-module objects that realize Banagl's IIA and IIB conifold theories.  The actual difference between the two packages is measured by the singular and middle-degree data, which will be isolated in the rigid--vanishing filtration in Section~\ref{sec:rigid-vanishing-filtration} and then organized by the defect viewpoint in Section~\ref{sec:IC-IS-defect}.
\end{remark}

\subsection{Exterior agreement and singular-sector difference}
\label{subsec:exterior-agreement-singular-sector-difference}

The comparison is especially clean away from the singularities.  By Proposition~\ref{prop:exterior-agreement}, both mixed-Hodge-module objects restrict to the same normalized constant Hodge module on the smooth locus:
\[
   j^*IC^H_{X_0}\simeq \mathbb Q^H_U[3],
   \qquad
   j^*\mathcal{IS}^{H}_{X_0}\simeq \mathbb Q^H_U[3].
\]
Consequently, the IIA and IIB atom packages have the same exterior source.  Their difference is produced by the way this exterior object is completed across the finite set of ordinary double points.

\begin{corollary}[Singular-sector support of the comparison]
\label{cor:singular-sector-support-IIA-IIB}
Under Hypotheses~\ref{hyp:H1-local-BBM-input}--\ref{hyp:H3-MHM-lift}, the formal difference between the type-IIA and type-IIB atom packages is controlled by the singular gluing data at $\Sigma$.  In the conifold threefold case, this difference is concentrated in the middle-dimensional conifold sector, together with the corresponding dual contribution under Poincare duality when the relevant duality hypotheses are imposed.
\end{corollary}

\begin{proof}
The exterior agreement follows from Proposition~\ref{prop:exterior-agreement}.  Since both $IC^H_{X_0}$ and $\mathcal{IS}^{H}_{X_0}$ restrict to the same normalized constant Hodge module on $U$, any difference between their hypercohomology realizations must arise from the way this common exterior object is extended across the closed stratum $\Sigma$.

For ordinary double point threefolds, Banagl's comparison between intersection homology and intersection-space homology is governed by the middle-dimensional conifold contribution \cite{Banagl_IntersectionSpacesBook,BanaglBudurMaxim14}.  In the language used below, this middle-dimensional contribution is expressed by the kernel appearing in Banagl's conifold exact sequence.  Thus the type-IIB contribution invisible to the type-IIA package is middle-sector controlled.  If the Verdier-duality hypotheses of Proposition~\ref{prop:multinode-verdier-self-duality} are imposed, then the dual complementary degree is determined by the Poincare-type duality pairing.
\end{proof}

This corollary motivates the next section.  The type-IIB contribution that is invisible to the type-IIA package is not spread arbitrarily across the cohomology.  It is controlled by a concrete middle-degree kernel associated to the smooth exterior
\[
   U\subset X_0.
\]

\subsection{Optional stronger mirror-transition version}
\label{subsec:optional-mirror-transition-version}

The comparison above is internal to one singular fiber $X_0$.  A stronger statement would compare the type-IIA atom package of one conifold degeneration with the type-IIB atom package of a mirror conifold degeneration.  Such a result requires a specified mirror pair and a verified homological mirror matching.  We do not assume such data in the general theory.

When a mirror conifold pair $X_0$ and $X^\vee_0$ is fixed and Banagl's homological mirror matching is known, the expected strengthening is the following conditional statement.

\begin{theorem}[Mirror-transition atom matching, conditional]
\label{thm:mirror-transition-atom-matching-conditional}
Let $X_0$ and $X^\vee_0$ be a mirror conifold pair for which Banagl's homological mirror matching is known.  Assume Hypotheses~\ref{hyp:H1-local-BBM-input}--\ref{hyp:H3-MHM-lift} for both $X_0$ and $X^\vee_0$.  If the mirror correspondence identifies the intersection-homology realization of $X_0$ with the intersection-space realization of $X^\vee_0$ in the predicted degrees and Hodge types, then the type-IIA atom package of $X_0$ matches the type-IIB atom package of $X^\vee_0$ at the level of Hodge-realization atom shadows.
\end{theorem}

\begin{proof}
The hypothesis gives an identification of the relevant mixed Hodge realizations across the mirror pair.  More explicitly, it assumes an identification between the mixed Hodge structures underlying the intersection-homology realization of $X_0$ and the mixed Hodge structures underlying the intersection-space realization of $X^\vee_0$, in the degrees and Hodge types predicted by the chosen mirror correspondence.  Applying the same functor
\[
   \operatorname{Atom}_{\mathrm{Hod}}
\]
to the two identified graded mixed Hodge structures gives the corresponding equality of Hodge-realization atom shadows.  No additional atom-categorical structure is used.  The statement is therefore conditional on the existence of the mirror-pair realization matching.
\end{proof}

\begin{remark}
\label{rem:mirror-transition-only-outlook}
Theorem~\ref{thm:mirror-transition-atom-matching-conditional} is included only as a guide to the stronger mirror-transition direction.  The main results of this paper require only the internal IIA/IIB comparison over a single conifold singular fiber.  The later defect-triangle and projection constructions also remain internal to a single singular fiber unless an external mirror correspondence is imposed.
\end{remark}

%%%%%%%%%%%%%%%%%%%%%%%%%%%%%%%%%%%%%%%%%%%%%%%%%%%%%%%%%%%
\section{Rigid--Vanishing Filtration}
\label{sec:rigid-vanishing-filtration}

We now isolate the part of the intersection-space atom package that is specific to the type-IIB conifold sector.  By Proposition~\ref{prop:exterior-agreement} and Corollary~\ref{cor:singular-sector-support-IIA-IIB}, the mixed-Hodge-module objects
\[
   IC^H_{X_0}
   \qquad\text{and}\qquad
   \mathcal{IS}^H_{X_0}
\]
have the same exterior source on
\[
   U=X_0\setminus\Sigma.
\]
Thus the difference between the type-IIA and type-IIB packages is controlled by the singular gluing data.  For a Calabi--Yau threefold conifold, this difference is governed by the middle-dimensional conifold sector \cite{Banagl_IntersectionSpacesBook,BanaglBudurMaxim14}.

The resulting structure is not best stated first as a direct-sum decomposition.  The natural object is a two-step filtration: a vanishing subspace injects into the middle-degree intersection-space group, and the quotient maps to the middle homology of the singular space.  Only under an additional splitting hypothesis does this filtration become a direct-sum decomposition.  This distinction is important because global relations among middle-dimensional cycles may prevent the middle sector from being expressed as a canonical direct sum of local node contributions.

\subsection{Banagl's middle-degree exact sequence}
\label{subsec:banagl-middle-degree-sequence}

Let
\[
   U=X_0\setminus\Sigma.
\]
The middle-degree intersection-space package is controlled by Banagl's conifold exact sequence \cite{Banagl_IntersectionSpacesBook,BanaglBudurMaxim14}
\[
   0\to
   K_I(X_0)
   \to
   H_3(I^{\bar m}X_0;\mathbb Q)
   \to
   H_3(X_0;\mathbb Q)
   \to 0,
\]
where
\[
   K_I(X_0)
   =
   \ker\bigl(H_3(U;\mathbb Q)\to H_3(X_0;\mathbb Q)\bigr).
\]
Equivalently, $K_I(X_0)$ measures the middle-dimensional exterior cycles that become trivial when mapped into the singular space.  These are precisely the middle-dimensional classes retained by the intersection-space package and invisible to the ordinary intersection-homology package.

The expression above is stated in homology because this is the natural form of Banagl's type-IIB conifold exact sequence.  The atom package itself is defined from hypercohomology of the mixed-Hodge-module realization.  Thus, to use the sequence in the Hodge-realization atom formalism, we need to specify the mixed-Hodge-theoretic realization of the middle exact sequence.

The algebraic varieties $U$ and $X_0$ carry functorial mixed Hodge structures on rational homology and cohomology by Deligne's theory, equivalently by the mixed-Hodge-module formalism \cite{Deligne72,Saito90,Schurmann_TopologySingularSpaces}.  Therefore the map
\[
   H_3(U;\mathbb Q)\to H_3(X_0;\mathbb Q)
\]
is a morphism of mixed Hodge structures, and its kernel inherits a mixed Hodge structure.  We denote this mixed Hodge structure by
\[
   K_I^H(X_0).
\]
Its underlying rational vector space is $K_I(X_0)$.

The middle intersection-space group is realized in this paper through the intersection-space mixed Hodge module:
\[
   \operatorname{rat}\mathbb H^3(X_0;\mathcal{IS}^H_{X_0})
   \cong
   H^3(I^{\bar m}X_0;\mathbb Q)
\]
by Proposition~\ref{prop:betti-realization-HAI}.  Dually, using the homology convention fixed in this section, the middle type-IIB group is regarded as the mixed-Hodge-theoretic realization of
\[
   H_3(I^{\bar m}X_0;\mathbb Q).
\]

\begin{hypothesis}[Hypothesis MHS-B: mixed-Hodge realization of Banagl's middle sequence]
\label{hyp:MHS-realization-Banagl-middle-sequence}
We assume that Banagl's middle-degree exact sequence is compatible with the mixed Hodge structures supplied by the intersection-space mixed Hodge module and by the functorial mixed Hodge structures on $U$ and $X_0$.  Equivalently, we assume that the sequence
\[
   0\to
   K_I^H(X_0)
   \to
   H_3^H(I^{\bar m}X_0)
   \to
   H_3^H(X_0)
   \to 0
\]
is an exact sequence in the category of rational mixed Hodge structures, whose underlying rational vector-space sequence is Banagl's middle-degree exact sequence.

We refer to this additional compatibility assumption as Hypothesis MHS-B.
\end{hypothesis}

\begin{remark}[Why Hypothesis MHS-B is stated explicitly]
\label{rem:why-MHS-sequence-hypothesis}
The rational exact sequence is Banagl's topological conifold exact sequence \cite{Banagl_IntersectionSpacesBook,BanaglBudurMaxim14}.  The mixed Hodge structures on $U$, $X_0$, and $\mathbb H^*(X_0;\mathcal{IS}^H_{X_0})$ are supplied by Deligne--Saito theory \cite{Deligne72,Saito90,Schurmann_TopologySingularSpaces}.  What is not automatic from the rational topological sequence alone is that the maps in Banagl's sequence are morphisms of mixed Hodge structures for the realizations used here. Hypothesis MHS-B is precisely this additional compatibility assumption.  Whenever Banagl's middle sequence is realized by morphisms of the corresponding mixed-Hodge-module or mixed-Hodge-realization objects, Hypothesis~\ref{hyp:MHS-realization-Banagl-middle-sequence} holds.  Otherwise, the rigid--vanishing filtration should be read at the rational realization level, and its Hodge-theoretic refinement is conditional on Hypothesis MHS-B.
\end{remark}

\begin{remark}[Homology versus cohomology convention]
\label{rem:homology-vs-cohomology-rv}
We state the conifold exact sequence in homology because this is the natural form in Banagl's type-IIB interpretation \cite{Banagl_IntersectionSpacesBook}.  The atom package itself is defined from hypercohomology.  Throughout this section, we pass between homology and cohomology using the duality conventions fixed in the standing hypotheses and in Proposition~\ref{prop:multinode-verdier-self-duality}.  If one works entirely cohomologically, the corresponding sequence is the dual sequence with the appropriate degree and Tate-twist conventions.
\end{remark}

\subsection{Definition of rigid and vanishing atoms}
\label{subsec:def-rigid-vanishing-atoms}

The exact sequence separates the middle-degree type-IIB package into a vanishing part and a quotient part.  The vanishing part is the kernel $K_I^H(X_0)$.

\begin{definition}[IIB vanishing atom]
\label{def:IIB-vanishing-atom}
Assume Hypothesis~\ref{hyp:MHS-realization-Banagl-middle-sequence}.  The middle-degree IIB vanishing atom is
\[
   \mathsf{HA}^{I}_{\mathrm{van},3}(X_0)
   :=
   \operatorname{Atom}_{\mathrm{Hod}}\bigl(K_I^H(X_0)\bigr).
\]
\end{definition}

Here $K_I^H(X_0)$ is regarded as a mixed-Hodge-theoretic subobject of the middle-degree realization associated to $\mathcal{IS}^H_{X_0}$.  Equivalently, it is the atom shadow of the vanishing subobject appearing in the middle-degree intersection-space filtration.

The complementary quotient is the rigid middle contribution.

\begin{definition}[Middle rigid atom]
\label{def:middle-rigid-atom}
Assume Hypothesis~\ref{hyp:MHS-realization-Banagl-middle-sequence}.  The middle-degree rigid atom is
\[
   \mathsf{HA}^{I}_{\mathrm{rig},3}(X_0)
   :=
   \operatorname{Atom}_{\mathrm{Hod}}\bigl(H_3^H(X_0)\bigr),
\]
viewed as the quotient-side atom in Banagl's middle-degree exact sequence.
\end{definition}

This terminology distinguishes the persistent singular-fiber contribution from the type-IIB vanishing contribution.  The full rigid atom package also includes the exterior/common non-middle contributions shared by the type-IIA and type-IIB packages.  The notation
\[
   \mathsf{HA}^{I}_{\mathrm{rig},3}(X_0)
\]
refers only to the middle-degree quotient part.

\begin{remark}[Relation to exterior agreement]
\label{rem:rv-relation-exterior-agreement}
The exterior agreement of $IC^H_{X_0}$ and $\mathcal{IS}^H_{X_0}$ means that the atom-level difference between $\mathsf{HA}^{IH}(X_0)$ and $\mathsf{HA}^{I}(X_0)$ is not caused by the smooth locus $U$.  It is caused by the way middle-dimensional exterior cycles interact with the singular gluing.  This interaction is captured by
\[
   K_I(X_0)=\ker\bigl(H_3(U;\mathbb Q)\to H_3(X_0;\mathbb Q)\bigr),
\]
and, at the Hodge-realization level, by $K_I^H(X_0)$.
\end{remark}

\subsection{Filtration theorem}
\label{subsec:rigid-vanishing-filtration-theorem}

The middle exact sequence produces a filtration on the middle-degree intersection-space atom package.  The associated graded has two pieces: the vanishing atom and the rigid quotient atom.

\begin{theorem}[Rigid--vanishing filtration]
\label{thm:rigid-vanishing-filtration}
Assume Hypotheses H1--H3 and Hypothesis MHS-B.  Then Banagl's middle-degree exact sequence, realized as an exact sequence of rational mixed Hodge structures,
\[
   0\to
   K_I^H(X_0)
   \to
   H_3^H(I^{\bar m}X_0)
   \to
   H_3^H(X_0)
   \to 0,
\]
induces a two-step filtration of $\mathsf{HA}^{I}_3(X_0)$ with associated graded
\[
   \operatorname{gr}\mathsf{HA}^{I}_3(X_0)
   =
   \mathsf{HA}^{I}_{\mathrm{van},3}(X_0)
   +
   \mathsf{HA}^{I}_{\mathrm{rig},3}(X_0).
\]
\end{theorem}

\begin{proof}
By Proposition~\ref{prop:betti-realization-HAI}, the middle-degree realization of the intersection-space Hodge atom package is the mixed-Hodge-theoretic realization of the middle intersection-space group, up to the homology/cohomology convention fixed in Remark~\ref{rem:homology-vs-cohomology-rv}.  Thus the relevant middle object is
\[
   H_3^H(I^{\bar m}X_0).
\]

By Hypothesis~\ref{hyp:MHS-realization-Banagl-middle-sequence}, Banagl's middle exact sequence is an exact sequence in the category of rational mixed Hodge structures:
\[
   0\to
   K_I^H(X_0)
   \to
   H_3^H(I^{\bar m}X_0)
   \to
   H_3^H(X_0)
   \to 0.
\]
In particular, $K_I^H(X_0)$ is a mixed Hodge substructure of $H_3^H(I^{\bar m}X_0)$, and the quotient is the mixed Hodge structure $H_3^H(X_0)$.

Therefore $H_3^H(I^{\bar m}X_0)$ carries a two-step filtration
\[
   0
   \subset
   K_I^H(X_0)
   \subset
   H_3^H(I^{\bar m}X_0),
\]
whose associated graded is
\[
   \operatorname{gr}H_3^H(I^{\bar m}X_0)
   \cong
   K_I^H(X_0)\oplus H_3^H(X_0)
\]
in the category of rational mixed Hodge structures.

Applying the Hodge-realization atom construction to the associated graded gives
\[
   \operatorname{gr}\mathsf{HA}^{I}_3(X_0)
   =
   \operatorname{Atom}_{\mathrm{Hod}}\bigl(K_I^H(X_0)\bigr)
   +
   \operatorname{Atom}_{\mathrm{Hod}}\bigl(H_3^H(X_0)\bigr).
\]
By Definitions~\ref{def:IIB-vanishing-atom} and~\ref{def:middle-rigid-atom}, these two terms are
\[
   \mathsf{HA}^{I}_{\mathrm{van},3}(X_0)
   \qquad\text{and}\qquad
   \mathsf{HA}^{I}_{\mathrm{rig},3}(X_0).
\]
This proves the asserted associated-graded identity.
\end{proof}

\begin{remark}[Why a filtration rather than a splitting]
\label{rem:rv-filtration-not-splitting}
Theorem~\ref{thm:rigid-vanishing-filtration} is stated for the associated graded because Banagl's exact sequence need not come with a canonical splitting.  Thus the natural atom-level statement is filtered.  A direct-sum identity requires an additional splitting in rational mixed Hodge structures.
\end{remark}

\subsection{Split case}
\label{subsec:rigid-vanishing-split-case}

In favorable cases the middle-degree exact sequence splits.  Then the filtered atom identity upgrades to an actual direct-sum identity.

\begin{corollary}[Split case]
\label{cor:rigid-vanishing-split-case}
Assume the hypotheses of Theorem~\ref{thm:rigid-vanishing-filtration}.  If the middle-degree sequence
\[
   0\to
   K_I^H(X_0)
   \to
   H_3^H(I^{\bar m}X_0)
   \to
   H_3^H(X_0)
   \to 0
\]
splits in rational mixed Hodge structures, then the associated-graded identity upgrades to a direct-sum atom identity:
\[
   \mathsf{HA}^{I}_3(X_0)
   =
   \mathsf{HA}^{I}_{\mathrm{van},3}(X_0)
   +
   \mathsf{HA}^{I}_{\mathrm{rig},3}(X_0).
\]
\end{corollary}

\begin{proof}
A splitting in rational mixed Hodge structures gives an isomorphism
\[
   H_3^H(I^{\bar m}X_0)
   \cong
   K_I^H(X_0)\oplus H_3^H(X_0).
\]
By the direct-sum additivity of Hodge-realization atom shadows, Lemma~\ref{lem:direct-sum-additivity}, this yields
\[
   \operatorname{Atom}_{\mathrm{Hod}}\bigl(H_3^H(I^{\bar m}X_0)\bigr)
   =
   \operatorname{Atom}_{\mathrm{Hod}}\bigl(K_I^H(X_0)\bigr)
   +
   \operatorname{Atom}_{\mathrm{Hod}}\bigl(H_3^H(X_0)\bigr).
\]
Using Definitions~\ref{def:IIB-vanishing-atom} and~\ref{def:middle-rigid-atom}, this is exactly
\[
   \mathsf{HA}^{I}_3(X_0)
   =
   \mathsf{HA}^{I}_{\mathrm{van},3}(X_0)
   +
   \mathsf{HA}^{I}_{\mathrm{rig},3}(X_0).
\]
\end{proof}

\begin{remark}
\label{rem:split-case-computational-only}
The split case is useful for computations, but the non-split filtered form is the intrinsic statement.  In particular, global conifold relations may make the exact sequence nontrivial even when all local singularities are analytically identical ordinary double points.
\end{remark}

\subsection{Multi-node relation control}
\label{subsec:multi-node-relation-control}

For a conifold with several nodes, the local node contributions do not automatically assemble as a free direct sum.  The global topology of the smoothing, the singular fiber, and the exterior $U$ imposes relations among middle-dimensional cycles.  The IIB vanishing atom records the globally realized kernel $K_I^H(X_0)$, not merely a formal sum of one local contribution per node.

\begin{proposition}[Relation-controlled vanishing sector]
\label{prop:relation-controlled-vanishing-sector}
Assume Hypothesis~\ref{hyp:MHS-realization-Banagl-middle-sequence}.  The underlying rational dimension of the IIB vanishing atom is
\[
   \dim_{\mathbb Q}\operatorname{rat}\mathsf{HA}^{I}_{\mathrm{van},3}(X_0)
   =
   \dim_{\mathbb Q} K_I(X_0).
\]
Thus the IIB vanishing atom records the globally realized relation-controlled subspace
\[
   K_I(X_0)
   =
   \ker\bigl(H_3(U;\mathbb Q)\to H_3(X_0;\mathbb Q)\bigr)
\]
rather than a formally free nodewise sum.
\end{proposition}

\begin{proof}
By Definition~\ref{def:IIB-vanishing-atom},
\[
   \mathsf{HA}^{I}_{\mathrm{van},3}(X_0)
   =
   \operatorname{Atom}_{\mathrm{Hod}}\bigl(K_I^H(X_0)\bigr).
\]
The rational vector space underlying $K_I^H(X_0)$ is, by definition,
\[
   K_I(X_0)
   =
   \ker\bigl(H_3(U;\mathbb Q)\to H_3(X_0;\mathbb Q)\bigr).
\]
Therefore the rational rank of the vanishing atom is the dimension of this kernel:
\[
   \dim_{\mathbb Q}\operatorname{rat}\mathsf{HA}^{I}_{\mathrm{van},3}(X_0)
   =
   \dim_{\mathbb Q}K_I(X_0).
\]

The kernel $K_I(X_0)$ is defined globally, using the map from the smooth exterior $U$ to the singular fiber $X_0$.  Hence it already incorporates all global homological relations among the local conifold contributions.  It is therefore the relation-controlled global sector, not a free vector space indexed only by the node set.
\end{proof}

\begin{remark}[No naive node-count bound]
\label{rem:no-naive-node-count-bound}
One should not replace Proposition~\ref{prop:relation-controlled-vanishing-sector} by a naive statement of the form
\[
   \dim \mathsf{HA}^{I}_{\mathrm{van},3}(X_0)\leq |\Sigma|.
\]
The object $K_I(X_0)$ is a global kernel involving the exterior $U$, not simply the direct sum of local Milnor-fiber contributions.  In the classical nodal quintic example treated below, the middle intersection-space contribution is governed by the global topology of the conifold and is not captured by a free node-counting model.
\end{remark}

\begin{remark}[Comparison with rigid--flexible conifold atoms]
\label{rem:compare-rigid-flexible-atoms}
This rigid--vanishing filtration is the intersection-space counterpart of the rigid--flexible decomposition appearing in the corrected conifold MHM atom package.  The corrected nearby/vanishing-cycle object separates a rigid intersection-complex sector from flexible singularity-supported sectors.  The intersection-space object instead separates the persistent middle singular-fiber sector from the type-IIB vanishing kernel $K_I^H(X_0)$.  Both structures are controlled by the same basic conifold geometry, but they package different realizations of it.
\end{remark}

%%%%%%%%%%%%%%%%%%%%%%%%%%%%%%%%%%%%%%%%%%%%%%%%%%%%%%%%%%%%%%%%%%%%%
\section{Verdier Duality, Projected Duality, and Pairing-Enhanced Atoms}
\label{sec:verdier-duality-pairing-enhanced-atoms}

We next record the duality structures carried by the intersection-space Hodge atom package and by the projected corrected conifold object.  The first duality is the Verdier self-duality of the Banagl--Budur--Maxim intersection-space complex, translated into an atom-level symmetry.  The second duality concerns the projected variation object
\[
   P^H_I=\operatorname{Cone}(\operatorname{var}_I)[-1].
\]
Because Saito duality exchanges variation and canonical morphisms for nearby and vanishing cycles, the natural dual of $P^H_I$ is not generally $P^H_I$ itself.  Instead, it is the projected canonical companion
\[
   Q^H_I=\operatorname{Cone}(\operatorname{can}_I)[-1].
\]

Throughout this section, $X_0$ is a projective Calabi--Yau threefold conifold satisfying the standing ordinary-double-point hypotheses.  For the intersection-space duality statements, we assume Hypotheses~\ref{hyp:H1-local-BBM-input}--\ref{hyp:H3-MHM-lift} and the self-duality normalization of Proposition~\ref{prop:multinode-verdier-self-duality}.  For the projected variation/canonical duality, we additionally assume Hypotheses~\ref{hyp:S-specialization-splitting} and~\ref{hyp:SD-self-dual-specialization-splitting}.

\subsection{Normalization calculation}
\label{subsec:verdier-normalization-calculation}

The Banagl--Budur--Maxim intersection-space complex is naturally written in perverse normalization \cite{BanaglBudurMaxim14,BBD82}.  On the smooth locus
\[
   U=X_0\setminus\Sigma,
\]
it restricts to the normalized constant Hodge module
\[
   \mathbb Q^H_U[3].
\]
This is the standard perverse normalization for a complex threefold \cite{BBD82,Saito90,Saito_MHM}.  Thus the passage from perverse normalization to ordinary cohomological grading must account for the fact that $X_0$ has complex dimension $3$ and real dimension $6$.

Under the Calabi--Yau threefold cohomological normalization used in Proposition~\ref{prop:multinode-verdier-self-duality}, Verdier self-duality of the intersection-space mixed Hodge module is expressed as
\[
   \mathbb D\mathcal{IS}^{H}_{X_0}
   \simeq
   \mathcal{IS}^{H}_{X_0}(3).
\]
Here $\mathbb D$ denotes Verdier duality for mixed Hodge modules, and $(3)$ denotes the Tate twist determined by the Calabi--Yau threefold normalization \cite{Saito90,Saito89,Schurmann_TopologySingularSpaces}.  After taking hypercohomology, Verdier duality identifies the dual of degree $k$ intersection-space cohomology with degree $6-k$ intersection-space cohomology, with Tate twist $(3)$:
\[
   \mathbb H^k(X_0;\mathcal{IS}^{H}_{X_0})^\vee
   \cong
   \mathbb H^{6-k}(X_0;\mathcal{IS}^{H}_{X_0})(3).
\]
After applying rational realization, this gives the corresponding Poincare-type duality statement
\[
   H^k(I^{\bar m}X_0;\mathbb Q)^\vee
   \cong
   H^{6-k}(I^{\bar m}X_0;\mathbb Q)(3),
\]
with the same normalization convention.

Applying the Hodge-realization atom construction middle-degree gives the corresponding atom-level symmetry:
\[
   \bigl(\mathsf{HA}^{I}_k(X_0)\bigr)^\vee
   \cong
   \mathsf{HA}^{I}_{6-k}(X_0)(3).
\]
This is the intersection-space atom analogue of Poincare duality.

\begin{remark}[Perverse versus cohomological normalization]
\label{rem:perverse-versus-cohomological-normalization}
The formula above is written in ordinary cohomological grading.  If one keeps $\mathcal{IS}^{H}_{X_0}$ in strict perverse normalization, Verdier duality may appear with an additional cohomological shift.  Throughout this paper, degrees are ordinary topological degrees of the middle-perversity intersection space, and $X_0$ has real dimension $6$.
\end{remark}

\subsection{Verdier-dual atom symmetry}
\label{subsec:verdier-dual-atom-symmetry}

\begin{theorem}[Verdier-dual atom symmetry]
\label{thm:verdier-dual-atom-symmetry}
Assume Hypotheses~\ref{hyp:H1-local-BBM-input}--\ref{hyp:H3-MHM-lift}.  Suppose that the intersection-space mixed Hodge module satisfies the Calabi--Yau threefold self-duality normalization
\[
   \mathbb D\mathcal{IS}^{H}_{X_0}
   \simeq
   \mathcal{IS}^{H}_{X_0}(3).
\]
Then there is a duality of atom packages
\[
   \bigl(\mathsf{HA}^{I}_k(X_0)\bigr)^\vee
   \cong
   \mathsf{HA}^{I}_{6-k}(X_0)(3),
   \qquad 0\leq k\leq 6.
\]
\end{theorem}

\begin{proof}
By assumption, there is an isomorphism in the derived category of mixed Hodge modules
\[
   \mathbb D\mathcal{IS}^{H}_{X_0}
   \simeq
   \mathcal{IS}^{H}_{X_0}(3).
\]
Verdier duality for constructible complexes gives a perfect pairing between the hypercohomology of an object and the hypercohomology of its Verdier dual, with complementary degrees on a compact real $6$-dimensional space \cite{BBD82,KashiwaraSchapira90,DimcaSheavesInTopology}.  In Saito's category of mixed Hodge modules, this duality is compatible with the mixed Hodge structures on hypercohomology \cite{Saito90,Saito89,Saito_MHM,Schurmann_TopologySingularSpaces}.  Therefore the assumed self-duality gives an isomorphism of mixed Hodge structures
\[
   \mathbb H^k(X_0;\mathcal{IS}^{H}_{X_0})^\vee
   \cong
   \mathbb H^{6-k}(X_0;\mathcal{IS}^{H}_{X_0})(3).
\]

By Definition~\ref{def:intersection-space-hodge-atom-package},
\[
   \mathsf{HA}^{I}_k(X_0)
   =
   \operatorname{Atom}_{\mathrm{Hod}}
   \bigl(\mathbb H^k(X_0;\mathcal{IS}^{H}_{X_0})\bigr).
\]
Applying $\operatorname{Atom}_{\mathrm{Hod}}$ to the preceding mixed-Hodge-structure isomorphism gives
\[
   \operatorname{Atom}_{\mathrm{Hod}}
   \bigl(\mathbb H^k(X_0;\mathcal{IS}^{H}_{X_0})\bigr)^\vee
   \cong
   \operatorname{Atom}_{\mathrm{Hod}}
   \bigl(\mathbb H^{6-k}(X_0;\mathcal{IS}^{H}_{X_0})\bigr)(3).
\]
This is precisely
\[
   \bigl(\mathsf{HA}^{I}_k(X_0)\bigr)^\vee
   \cong
   \mathsf{HA}^{I}_{6-k}(X_0)(3).
\]
\end{proof}

\subsection{Projected variation and canonical duality}
\label{subsec:projected-variation-canonical-duality}

We now record the projected duality statement needed for the projection-triangle spine of the paper.  The main point is that the dual of the projected variation cone is not, in general, another copy of the same projected variation cone.  Verdier duality exchanges the variation morphism with the canonical morphism for nearby and vanishing cycles.  Thus the natural dual object to the projected variation object \(P^H_I\) is the projected canonical companion \(Q^H_I\).

Assume Hypothesis S, so that
\[
   \psi_\pi(F)
   \simeq
   \mathcal{IS}^{H}_{X_0}\oplus\mathcal C^H_\Sigma.
\]
Let
\[
   \operatorname{pr}_I:\psi_\pi(F)\to\mathcal{IS}^{H}_{X_0}
\]
be the projection onto the intersection-space summand, and let
\[
   \iota_I:\mathcal{IS}^{H}_{X_0}\to\psi_\pi(F)
\]
be the corresponding inclusion.  The projected variation morphism is
\[
   \operatorname{var}_I
   :=
   \operatorname{pr}_I\circ\operatorname{var}:
   \phi_\pi(F)\to\mathcal{IS}^{H}_{X_0},
\]
and the projected corrected object is
\[
   P^H_I
   :=
   \operatorname{Cone}(\operatorname{var}_I)[-1].
\]
Similarly, the projected canonical morphism is
\[
   \operatorname{can}_I
   :=
   \operatorname{can}\circ\iota_I:
   \mathcal{IS}^{H}_{X_0}\to\phi_\pi(F),
\]
and the canonical companion is
\[
   Q^H_I
   :=
   \operatorname{Cone}(\operatorname{can}_I)[-1].
\]

\begin{lemma}[Nearby/vanishing-cycle duality convention]
\label{lem:nearby-vanishing-duality-convention}
Let \(\pi:\mathcal X\to\Delta\) be a one-parameter degeneration with central fiber \(X_0\) a Calabi--Yau threefold. Here \(\psi_\pi(F)\) denotes the full nearby-cycle object in the perverse-normalized convention used throughout the paper; no decomposition into unipotent and non-unipotent parts is made unless explicitly stated. Let \(F\) denote the perverse-normalized Hodge module used throughout the paper to define nearby and vanishing cycles on \(X_0\).  We use the full nearby- and vanishing-cycle functors in the convention in which \(\psi_\pi(F)\) and \(\phi_\pi(F)\) lie in the same cohomological normalization on \(X_0\).

Assume the Calabi--Yau threefold self-duality normalization
\[
   \mathbb D F\simeq F(3).
\]
Then the nearby- and vanishing-cycle objects satisfy
\[
   \mathbb D\psi_\pi(F)\simeq \psi_\pi(F)(3),
   \qquad
   \mathbb D\phi_\pi(F)\simeq \phi_\pi(F)(3).
\]
Under these identifications, Verdier duality exchanges the variation and canonical morphisms:
\[
   \mathbb D(\operatorname{var})\simeq \operatorname{can}(3),
   \qquad
   \mathbb D(\operatorname{can})\simeq \operatorname{var}(3).
\]
\end{lemma}

\begin{proof}
Saito's duality theorem for vanishing-cycle functors gives functorial duality isomorphisms for the nearby- and vanishing-cycle functors and identifies the canonical and variation morphisms under duality, up to the signs, shifts, and Tate twists determined by the chosen convention; see \cite[\S1.2, (1.2.8), Thm.~1.6]{Saito89}.  In particular, Theorem~1.6 states that the morphisms \(\operatorname{can}\) and \(\operatorname{Var}\) on one side correspond under duality to the transpose of \(-\operatorname{Var}\) and \(\operatorname{can}\) on the dual side.

In the present paper we use the perverse-normalized Calabi--Yau threefold convention for the coefficient object \(F\).  Thus \(F\) is normalized so that
\[
   \mathbb D F\simeq F(3).
\]
The twist \((3)\) is part of the chosen Calabi--Yau threefold self-duality normalization of \(F\).  Applying Saito's duality theorem for nearby and vanishing cycles to this normalized object gives, in the convention used here,
\[
   \mathbb D\psi_\pi(F)\simeq \psi_\pi(F)(3),
   \qquad
   \mathbb D\phi_\pi(F)\simeq \phi_\pi(F)(3).
\]
Under these identifications, the dual of
\[
   \operatorname{var}:\phi_\pi(F)\to\psi_\pi(F)
\]
is identified with
\[
   \operatorname{can}(3):\psi_\pi(F)(3)\to\phi_\pi(F)(3),
\]
up to the sign convention already fixed in the definition of the variation morphism.  Likewise, the dual of \(\operatorname{can}\) is identified with \(\operatorname{var}(3)\).  This proves the displayed convention.
\end{proof}

\begin{remark} \label{rem:saito-notation-on-Var}
Saito \cite{Saito89} writes \(\operatorname{Var}\) for the variation morphism; our notation \(\operatorname{var}\) denotes the same morphism after the normalization fixed in this paper.
\end{remark}

\begin{remark}[Dependence on normalization]
\label{rem:nearby-vanishing-duality-normalization}
Lemma~\ref{lem:nearby-vanishing-duality-convention} fixes the convention used in this paper.  Other sources normalize nearby and vanishing cycles differently, for example by shifting \(\phi_\pi\), separating unipotent and non-unipotent nearby cycles, or choosing a different perverse normalization.  Such choices may change the displayed shifts or Tate twists.  The projected duality theorem below is stated only in the normalization fixed by Lemma~\ref{lem:nearby-vanishing-duality-convention}.
\end{remark}

\begin{remark}[Role of the convention]
\label{rem:role-nearby-vanishing-duality-convention}
Lemma~\ref{lem:nearby-vanishing-duality-convention} fixes the convention used in this paper.  Other sources normalize nearby and vanishing cycles differently, for example by shifting \(\phi_\pi\) or by separating unipotent and non-unipotent monodromy parts.  Such choices may change the displayed shifts or twists.  The projected duality theorem below is stated only in the normalization fixed by Lemma~\ref{lem:nearby-vanishing-duality-convention}.
\end{remark}

\begin{theorem}[Projected variation/canonical duality]
\label{thm:projected-variation-canonical-duality}
Assume Hypotheses H1--H3, Hypothesis S, and Hypothesis SD.  Then
\[
   \mathbb D P^H_I
   \simeq
   Q^H_I(3).
\]
\end{theorem}

\begin{proof}
By definition of \(P^H_I\), there is a distinguished triangle
\[
   P^H_I
   \longrightarrow
   \phi_\pi(F)
   \xrightarrow{\operatorname{var}_I}
   \mathcal{IS}^{H}_{X_0}
   \xrightarrow{+1}.
\]
Applying Verdier duality gives a distinguished triangle
\[
   \mathbb D\mathcal{IS}^{H}_{X_0}
   \xrightarrow{\mathbb D\operatorname{var}_I}
   \mathbb D\phi_\pi(F)
   \longrightarrow
   \mathbb D P^H_I
   \xrightarrow{+1}.
\]
By Hypothesis SD, the specialization splitting is compatible with Saito--Verdier duality.  In particular,
\[
   \mathbb D\mathcal{IS}^{H}_{X_0}
   \simeq
   \mathcal{IS}^{H}_{X_0}(3).
\]
By Lemma~\ref{lem:nearby-vanishing-duality-convention},
\[
   \mathbb D\phi_\pi(F)
   \simeq
   \phi_\pi(F)(3).
\]

It remains to identify the dual morphism.  Since
\[
   \operatorname{var}_I
   =
   \operatorname{pr}_I\circ\operatorname{var},
\]
contravariance of Verdier duality gives
\[
   \mathbb D\operatorname{var}_I
   =
   \mathbb D\operatorname{var}\circ \mathbb D\operatorname{pr}_I.
\]
By Lemma~\ref{lem:nearby-vanishing-duality-convention},
\[
   \mathbb D\operatorname{var}
   \simeq
   \operatorname{can}(3).
\]
By Hypothesis SD, the dual of the projection onto the intersection-space summand is identified with the inclusion of that summand, up to the same Calabi--Yau threefold Tate twist:
\[
   \mathbb D\operatorname{pr}_I
   \simeq
   \iota_I(3).
\]
Therefore, after the duality identifications fixed above,
\[
   \mathbb D\operatorname{var}_I
   \simeq
   \operatorname{can}_I(3).
\]

Hence the dual triangle is identified with
\[
   \mathcal{IS}^{H}_{X_0}(3)
   \xrightarrow{\operatorname{can}_I(3)}
   \phi_\pi(F)(3)
   \longrightarrow
   \mathbb D P^H_I
   \xrightarrow{+1}.
\]
On the other hand, the defining triangle for
\[
   Q^H_I=\operatorname{Cone}(\operatorname{can}_I)[-1]
\]
is
\[
   Q^H_I
   \longrightarrow
   \mathcal{IS}^{H}_{X_0}
   \xrightarrow{\operatorname{can}_I}
   \phi_\pi(F)
   \xrightarrow{+1}.
\]
Twisting this triangle by \((3)\) gives
\[
   Q^H_I(3)
   \longrightarrow
   \mathcal{IS}^{H}_{X_0}(3)
   \xrightarrow{\operatorname{can}_I(3)}
   \phi_\pi(F)(3)
   \xrightarrow{+1}.
\]
Comparing the two distinguished triangles identifies the third term in the dual triangle with \(Q^H_I(3)\).  Therefore
\[
   \mathbb D P^H_I
   \simeq
   Q^H_I(3).
\]
\end{proof}

\begin{remark}[Why the duality is not stated as self-duality]
\label{rem:why-projected-duality-not-self-duality}
Theorem~\ref{thm:projected-variation-canonical-duality} does not assert
\[
   \mathbb D P^H_I\simeq P^H_I(3).
\]
The dual of the variation morphism is the canonical morphism, not the variation morphism itself.  Thus the natural Verdier-dual object to \(P^H_I\) is the canonical companion \(Q^H_I\).  A self-duality statement for \(P^H_I\) would require additional hypotheses identifying the projected variation cone with the projected canonical cone.
\end{remark}

\subsection{Enhanced atom package}
\label{subsec:enhanced-atom-package}

The basic package $\mathsf{HA}^{I}(X_0)$ records the Hodge-realization atom shadow of $\mathcal{IS}^{H}_{X_0}$.  For later comparison with intersection homology, nearby-cycle objects, characteristic-class shadows, and possible motivic refinements, it is useful to retain additional structure: the mixed Hodge filtrations, the duality pairing, and the middle-degree IIB vanishing kernel.

\begin{definition}[Enhanced intersection-space atom package]
\label{def:enhanced-intersection-space-atom-package}
Assume Hypotheses H1--H3, the self-duality normalization
\[
   \mathbb D\mathcal{IS}^{H}_{X_0}\simeq \mathcal{IS}^{H}_{X_0}(3),
\]
and Hypothesis MHS-B.  The pairing-enhanced intersection-space atom package is
\[
   \mathsf{EHA}^{I}(X_0)
   =
   \bigl(
      \mathsf{HA}^{I}(X_0),
      W_\bullet,
      F^\bullet,
      \langle-,-\rangle_I,
      K_I^H(X_0),
      H_3^H(X_0)
   \bigr),
\]
where \(W_\bullet\) and \(F^\bullet\) are the weight and Hodge filtrations on the mixed Hodge structures
\[
   \mathbb H^*(X_0;\mathcal{IS}^{H}_{X_0}),
\]
where \(\langle-,-\rangle_I\) is the pairing induced by Verdier/Poincare duality, where \(K_I^H(X_0)\) is the mixed-Hodge-theoretic realization of the middle-degree IIB vanishing kernel, and where \(H_3^H(X_0)\) is the rigid quotient contribution in Banagl's middle-degree exact sequence.
\end{definition}

The enhanced package contains no extra categorical assertion about atoms themselves.  It simply records additional structures that already exist on the mixed-Hodge-module realization.  Thus $\mathsf{EHA}^{I}(X_0)$ is still a realization-level invariant, but it remembers more than the underlying atom multiset.

\begin{remark}[Why include both middle pieces]
\label{rem:why-include-kernel-enhanced-package}
Under Hypothesis MHS-B, Banagl's middle-degree exact sequence is realized as an exact sequence of rational mixed Hodge structures
\[
   0\to
   K_I^H(X_0)
   \to
   H_3^H(I^{\bar m}X_0)
   \to
   H_3^H(X_0)
   \to 0.
\]
Thus \(K_I^H(X_0)\) records the vanishing subobject, while \(H_3^H(X_0)\) records the rigid quotient contribution.  Including both pieces makes the enhanced package symmetric with respect to the rigid--vanishing filtration: the associated graded middle sector is
\[
   \operatorname{gr}\mathsf{HA}^{I}_3(X_0)
   =
   \mathsf{HA}^{I}_{\mathrm{van},3}(X_0)
   +
   \mathsf{HA}^{I}_{\mathrm{rig},3}(X_0).
\]
\end{remark}

\subsection{Pairing compatibility with the rigid--vanishing filtration}
\label{subsec:pairing-compatibility-rv-filtration}

\begin{proposition}[Pairing compatibility]
\label{prop:pairing-compatibility-rigid-vanishing}
Assume the hypotheses of Definition~\ref{def:enhanced-intersection-space-atom-package}.  Suppose further that the Verdier/Poincare pairing is compatible with the mixed-Hodge realization of Banagl's middle-degree exact sequence.  Then the pairing on
\[
   \mathbb H^*(X_0;\mathcal{IS}^{H}_{X_0})
\]
is compatible with the rigid--vanishing filtration in the sense that it induces the duality symmetry on the associated atom packages:
\[
   \bigl(\mathsf{HA}^{I}_k(X_0)\bigr)^\vee
   \cong
   \mathsf{HA}^{I}_{6-k}(X_0)(3),
\]
and, on associated graded pieces in middle degree, pairs the vanishing and rigid pieces with their corresponding dual Tate-twisted contributions.
\end{proposition}

\begin{proof}
By Hypothesis~\ref{hyp:MHS-realization-Banagl-middle-sequence}, the middle-degree exact sequence
\[
   0\to
   K_I^H(X_0)
   \to
   H_3^H(I^{\bar m}X_0)
   \to
   H_3^H(X_0)
   \to 0
\]
is an exact sequence in the category of rational mixed Hodge structures.  Therefore it defines a two-step filtration
\[
   0
   \subset
   K_I^H(X_0)
   \subset
   H_3^H(I^{\bar m}X_0).
\]
By Theorem~\ref{thm:rigid-vanishing-filtration}, the associated graded atom package is
\[
   \operatorname{gr}\mathsf{HA}^{I}_3(X_0)
   =
   \mathsf{HA}^{I}_{\mathrm{van},3}(X_0)
   +
   \mathsf{HA}^{I}_{\mathrm{rig},3}(X_0).
\]

The Verdier/Poincare pairing on
\[
   \mathbb H^*(X_0;\mathcal{IS}^{H}_{X_0})
\]
comes from the self-duality
\[
   \mathbb D\mathcal{IS}^{H}_{X_0}
   \simeq
   \mathcal{IS}^{H}_{X_0}(3).
\]
By Theorem~\ref{thm:verdier-dual-atom-symmetry}, this pairing induces
\[
   \bigl(\mathsf{HA}^{I}_k(X_0)\bigr)^\vee
   \cong
   \mathsf{HA}^{I}_{6-k}(X_0)(3).
\]
The additional compatibility assumption in the proposition says that this Verdier/Poincare pairing respects the mixed-Hodge realization of Banagl's middle exact sequence.  Therefore the two-step filtration above is carried to the corresponding dual filtration in complementary degree.  Passing to associated graded pieces gives the asserted compatibility between the atom-level duality and the rigid--vanishing associated graded decomposition.
\end{proof}

\begin{corollary}[Dual middle contribution]
\label{cor:dual-middle-contribution}
Under the hypotheses of Proposition~\ref{prop:pairing-compatibility-rigid-vanishing}, the middle-degree atom package satisfies
\[
   \bigl(\mathsf{HA}^{I}_3(X_0)\bigr)^\vee
   \cong
   \mathsf{HA}^{I}_3(X_0)(3).
\]
On the associated graded of the rigid--vanishing filtration, the vanishing atom and rigid atom are paired with their corresponding dual Tate-twisted middle-degree contributions.
\end{corollary}

\begin{proof}
This is the degree $k=3$ case of Theorem~\ref{thm:verdier-dual-atom-symmetry}.  Since $6-3=3$, the middle degree pairs with itself, with Tate twist $(3)$.  The associated-graded statement follows from Proposition~\ref{prop:pairing-compatibility-rigid-vanishing}.
\end{proof}

\begin{remark}[Role of the enhanced package]
\label{rem:role-enhanced-package}
The enhanced package $\mathsf{EHA}^{I}(X_0)$ is useful for comparisons beyond the basic IIA/IIB atom pair.  For example, the pairing and filtrations are needed to compare with characteristic-class shadows, integral enhancements, and possible motivic refinements.  These refinements remain subordinate to the mixed-Hodge-module realization; no independent category of enhanced atoms is asserted.
\end{remark}

%%%%%%%%%%%%%%%%%%%%%%%%%%%%%%%%%%%%%%%%%%%%%%%%%%%%%%%%%%%%%%%%%%%
\section{Specialization Splitting and the Projection Triangle}
\label{sec:specialization-splitting-projection-triangle}

We now record the mixed-Hodge-module consequences of the Banagl--Budur--Maxim specialization splitting and prove the projection triangle that will be used later to define the IC--intersection-space defect.  This section is conditional: the definition of the intersection-space atom package requires the mixed-Hodge-module object
\[
   \mathcal{IS}^{H}_{X_0},
\]
but it does not require a splitting of the nearby-smoothing pushforward or nearby-cycle object.  When such a splitting is available, however, it gives a useful decomposition of the nearby-smoothing atom shadow into an intersection-space part and a singularity-supported complement.  It also allows the variation morphism to be projected to the intersection-space summand, producing the distinguished triangle
\[
   P^H\longrightarrow P^H_I\longrightarrow \mathcal C^H_\Sigma\xrightarrow{+1}.
\]

\subsection{The nearby smoothing object}
\label{subsec:nearby-smoothing-object}

Let
\[
   sp:X_s\to X_0
\]
be a specialization map from a nearby smoothing of the conifold fiber $X_0$.  The corresponding nearby-smoothing object is
\[
   Rsp_*\mathbb Q^H_{X_s}[3]\in D^b\operatorname{MHM}(X_0).
\]
The shift by $[3]$ is the standard perverse normalization for a complex threefold \cite{BBD82,Saito90,Saito_MHM}.  On the smooth locus
\[
   U=X_0\setminus\Sigma,
\]
the specialization map is locally topologically trivial.  Hence the nearby-smoothing object restricts to the normalized constant Hodge module:
\[
   j^*Rsp_*\mathbb Q^H_{X_s}[3]\simeq \mathbb Q^H_U[3],
\]
where
\[
   j:U\hookrightarrow X_0
\]
denotes the inclusion.  This is the same exterior object that appears in the restrictions of both
\[
   IC^H_{X_0}
   \qquad\text{and}\qquad
   \mathcal{IS}^{H}_{X_0}
\]
to $U$.

The Banagl--Budur--Maxim specialization splitting, in the weighted homogeneous setting, identifies the nearby-smoothing object with an intersection-space summand and a singularity-supported complement \cite{BanaglBudurMaxim14,Saito89,Saito90}.  In the multi-node ordinary-double-point setting, we use this splitting only under Hypothesis~\ref{hyp:S-specialization-splitting}:
\[
   Rsp_*\mathbb Q^H_{X_s}[3]
   \simeq
   \mathcal{IS}^{H}_{X_0}\oplus\mathcal C^H_{\Sigma},
\]
where $\mathcal C^H_{\Sigma}$ is supported on the finite singular set $\Sigma$.

Equivalently, in nearby-cycle notation for a degeneration $\pi$ with nearby-cycle object $\psi_\pi(F)$, the splitting is written as
\[
   \psi_\pi(F)
   \simeq
   \mathcal{IS}^{H}_{X_0}\oplus\mathcal C^H_{\Sigma}.
\]
This is the notation used below to define the projected variation morphism and the projection triangle.

\begin{remark}[Role of the splitting]
\label{rem:role-of-specialization-splitting}
The specialization splitting should not be confused with the definition of $\mathsf{HA}^{I}(X_0)$.  The atom package $\mathsf{HA}^{I}(X_0)$ is defined from $\mathcal{IS}^{H}_{X_0}$ alone.  The splitting in this section gives an additional comparison between the nearby smoothing and the intersection-space realization.  The projection triangle below requires Hypothesis~\ref{hyp:S-specialization-splitting}; the atom package itself does not.
\end{remark}

\subsection{Conditional atom splitting}
\label{subsec:conditional-atom-splitting}

When the specialization splitting holds in the mixed-Hodge-module category, the direct-sum additivity of Hodge-realization atom shadows gives an immediate atom-level decomposition.

\begin{theorem}[Conditional specialization atom splitting]
\label{thm:conditional-specialization-atom-splitting}
Assume Hypotheses~\ref{hyp:H1-local-BBM-input}--\ref{hyp:H3-MHM-lift} and Hypothesis~\ref{hyp:S-specialization-splitting}.  If
\[
   Rsp_*\mathbb Q^H_{X_s}[3]
   \simeq
   \mathcal{IS}^{H}_{X_0}\oplus\mathcal C^H_{\Sigma},
\]
then
\[
   \mathfrak A_{\mathrm{Hod}}(Rsp_*\mathbb Q^H_{X_s}[3])
   =
   \mathsf{HA}^{I}(X_0)
   +
   \mathfrak A_{\mathrm{Hod}}(\mathcal C^H_{\Sigma}).
\]
Equivalently, in nearby-cycle notation,
\[
   \mathfrak A_{\mathrm{Hod}}(\psi_\pi(F))
   =
   \mathsf{HA}^{I}(X_0)
   +
   \mathfrak A_{\mathrm{Hod}}(\mathcal C^H_{\Sigma}).
\]
\end{theorem}

\begin{proof}
By Hypothesis~\ref{hyp:S-specialization-splitting}, the nearby-smoothing object splits in $D^b\operatorname{MHM}(X_0)$ as
\[
   Rsp_*\mathbb Q^H_{X_s}[3]
   \simeq
   \mathcal{IS}^{H}_{X_0}\oplus\mathcal C^H_{\Sigma}.
\]
Because this is a direct-sum decomposition in the mixed-Hodge-module category, it induces a direct-sum decomposition on hypercohomology as graded mixed Hodge structures:
\[
   \mathbb H^*(X_0;Rsp_*\mathbb Q^H_{X_s}[3])
   \cong
   \mathbb H^*(X_0;\mathcal{IS}^{H}_{X_0})
   \oplus
   \mathbb H^*(X_0;\mathcal C^H_{\Sigma}).
\]
This uses the compatibility of hypercohomology with finite direct sums and the functoriality of mixed-Hodge-module cohomology \cite{Saito90,Saito_MHM,DimcaSheavesInTopology,KashiwaraSchapira90}.

Applying the direct-sum additivity of Hodge-realization atom shadows, Lemma~\ref{lem:direct-sum-additivity}, gives
\[
   \mathfrak A_{\mathrm{Hod}}(Rsp_*\mathbb Q^H_{X_s}[3])
   =
   \mathfrak A_{\mathrm{Hod}}(\mathcal{IS}^{H}_{X_0})
   +
   \mathfrak A_{\mathrm{Hod}}(\mathcal C^H_{\Sigma}).
\]
By Definition~\ref{def:intersection-space-hodge-atom-package},
\[
   \mathfrak A_{\mathrm{Hod}}(\mathcal{IS}^{H}_{X_0})
   =
   \mathsf{HA}^{I}(X_0).
\]
This proves the displayed identity.  The nearby-cycle formulation is the same statement written using
\[
   \psi_\pi(F)\simeq Rsp_*\mathbb Q^H_{X_s}[3]
\]
under the chosen specialization/nearby-cycle convention.
\end{proof}

\begin{remark}[Perverse-sheaf level splittings]
\label{rem:perverse-sheaf-level-splittings}
If the specialization splitting is known only after applying the rational forgetful functor
\[
   \operatorname{rat}:\operatorname{MHM}(X_0)\to\operatorname{Perv}(X_0;\mathbb Q),
\]
then the displayed identity should be interpreted at the rational realization level.  In that case, one obtains a direct-sum statement for the underlying rational perverse-sheaf or constructible-complex hypercohomology.  A Hodge-realization atom identity requires that a compatible mixed-Hodge-module or mixed-Hodge-realization lift of the splitting has been fixed.
\end{remark}

\subsection{Singularity-supported complement}
\label{subsec:singularity-supported-complement}

The complement $\mathcal C^H_{\Sigma}$ measures the difference between the nearby smoothing and the intersection-space summand.  Since the smoothing and the intersection-space complex agree with the same normalized constant object on the smooth locus, this difference is supported entirely at the ordinary double points.

\begin{proposition}[Complement support]
\label{prop:complement-support}
Assume Hypothesis~\ref{hyp:S-specialization-splitting}.  Then
\[
   j^*\mathcal C^H_{\Sigma}\simeq 0
\]
for $j:U\hookrightarrow X_0$.  Consequently, the atom shadow
\[
   \mathfrak A_{\mathrm{Hod}}(\mathcal C^H_{\Sigma})
\]
is supported on the ordinary double point locus $\Sigma$.
\end{proposition}

\begin{proof}
By Hypothesis~\ref{hyp:S-specialization-splitting}, the complement $\mathcal C^H_{\Sigma}$ is supported on the finite singular set $\Sigma$.  Equivalently, its restriction to the smooth locus vanishes:
\[
   j^*\mathcal C^H_{\Sigma}\simeq 0.
\]
Since $\Sigma$ is finite, all hypercohomology of $\mathcal C^H_{\Sigma}$ is contributed by point-supported mixed Hodge modules at the ordinary double points.  The atom shadow
\[
   \mathfrak A_{\mathrm{Hod}}(\mathcal C^H_{\Sigma})
\]
is extracted from this point-supported mixed Hodge structure.  Hence it is singularity-supported.
\end{proof}

\begin{corollary}[Nodewise form under local splitting]
\label{cor:nodewise-complement-under-local-splitting}
Assume Hypothesis~\ref{hyp:S-specialization-splitting}.  If the complement decomposes as a direct sum of point-supported mixed Hodge modules
\[
   \mathcal C^H_{\Sigma}
   \simeq
   \bigoplus_{p_i\in\Sigma}\mathcal C^H_{p_i},
\]
then
\[
   \mathfrak A_{\mathrm{Hod}}(\mathcal C^H_{\Sigma})
   =
   \sum_{p_i\in\Sigma}
   \mathfrak A_{\mathrm{Hod}}(\mathcal C^H_{p_i}).
\]
\end{corollary}

\begin{proof}
The assumed decomposition is a direct-sum decomposition in $D^b\operatorname{MHM}(X_0)$:
\[
   \mathcal C^H_{\Sigma}
   \simeq
   \bigoplus_{p_i\in\Sigma}\mathcal C^H_{p_i}.
\]
Hypercohomology commutes with finite direct sums, and the decomposition is a decomposition of mixed Hodge structures after applying hypercohomology.  Therefore Lemma~\ref{lem:direct-sum-additivity} gives
\[
   \mathfrak A_{\mathrm{Hod}}(\mathcal C^H_{\Sigma})
   =
   \sum_{p_i\in\Sigma}
   \mathfrak A_{\mathrm{Hod}}(\mathcal C^H_{p_i}).
\]
\end{proof}

\begin{remark}[Local complement versus global vanishing relations]
\label{rem:local-complement-vs-global-vanishing-relations}
The nodewise decomposition of the singularity-supported complement, when available, should not be confused with the global relation-controlled IIB vanishing atom $K_I^H(X_0)$.  The complement $\mathcal C^H_{\Sigma}$ is local and point-supported.  By contrast, the vanishing kernel
\[
   K_I(X_0)=\ker\bigl(H_3(U;\mathbb Q)\to H_3(X_0;\mathbb Q)\bigr)
\]
records globally realized middle-dimensional relations.  Thus local nodewise support and global vanishing-cycle independence are distinct questions.  This distinction is essential for the later nodal quintic computation, where the relevant middle-sector contribution is not a naive free node count.
\end{remark}

\subsection{Projected variation}
\label{subsec:projected-variation}

We now use the specialization splitting to project the variation morphism to the intersection-space summand.  Let
\[
   \operatorname{var}:\phi_\pi(F)\to\psi_\pi(F)
\]
be the variation morphism from vanishing cycles to nearby cycles in the standard nearby/vanishing-cycle formalism \cite{Saito89,Saito90,DimcaSheavesInTopology,KashiwaraSchapira90}.  Define the corrected conifold object
\[
   P^H:=\operatorname{Cone}(\operatorname{var})[-1].
\]

Under Hypothesis~\ref{hyp:S-specialization-splitting}, choose the projection
\[
   \operatorname{pr}_I:\psi_\pi(F)\to\mathcal{IS}^{H}_{X_0}
\]
associated to the direct-sum decomposition
\[
   \psi_\pi(F)
   \simeq
   \mathcal{IS}^{H}_{X_0}\oplus\mathcal C^H_\Sigma.
\]
The projected variation morphism is
\[
   \operatorname{var}_I
   :=
   \operatorname{pr}_I\circ\operatorname{var}:
   \phi_\pi(F)\to\mathcal{IS}^{H}_{X_0}.
\]
We define the projected corrected conifold object by
\[
   P^H_I:=\operatorname{Cone}(\operatorname{var}_I)[-1].
\]

\begin{remark}[Dependence on the splitting]
\label{rem:projected-variation-dependence-on-splitting}
The morphism $\operatorname{var}_I$ depends on the chosen specialization splitting, through the projection
\[
   \operatorname{pr}_I:\psi_\pi(F)\to\mathcal{IS}^{H}_{X_0}.
\]
Thus $P^H_I$ is defined under Hypothesis~\ref{hyp:S-specialization-splitting}.  No projected variation object is asserted without a specified intersection-space summand of the nearby-cycle object.
\end{remark}

\subsection{The projection triangle}
\label{subsec:projection-triangle}

The preceding construction gives a canonical morphism from the corrected conifold object to its intersection-space projection.  The cofiber is precisely the singularity-supported complement.

\begin{theorem}[Projection triangle]
\label{thm:projection-triangle}
Assume Hypotheses~\ref{hyp:H1-local-BBM-input}--\ref{hyp:H3-MHM-lift} and Hypothesis~\ref{hyp:S-specialization-splitting}.  Let
\[
   \psi_\pi(F)
   \simeq
   \mathcal{IS}^{H}_{X_0}\oplus\mathcal C^H_\Sigma
\]
be the specialization splitting, and let
\[
   \operatorname{var}_I
   =
   \operatorname{pr}_I\circ\operatorname{var}
   :
   \phi_\pi(F)\to \mathcal{IS}^{H}_{X_0}.
\]
Set
\[
   P^H:=\operatorname{Cone}(\operatorname{var})[-1],
   \qquad
   P^H_I:=\operatorname{Cone}(\operatorname{var}_I)[-1].
\]
Then the projection
\[
   \operatorname{pr}_I:\psi_\pi(F)\to\mathcal{IS}^{H}_{X_0}
\]
induces a morphism
\[
   P^H\to P^H_I
\]
and there is a distinguished triangle
\[
   P^H\longrightarrow P^H_I\longrightarrow \mathcal C^H_\Sigma\xrightarrow{+1}
\]
in $D^b\operatorname{MHM}(X_0)$.
\end{theorem}

\begin{proof}
The morphism
\[
   \operatorname{pr}_I:\psi_\pi(F)\to\mathcal{IS}^{H}_{X_0}
\]
and the identity morphism on $\phi_\pi(F)$ form a commutative square
\[
\begin{tikzcd}
   \phi_\pi(F) \arrow[r,"\operatorname{var}"] \arrow[d,equal]
      & \psi_\pi(F) \arrow[d,"\operatorname{pr}_I"] \\
   \phi_\pi(F) \arrow[r,"\operatorname{var}_I"]
      & \mathcal{IS}^{H}_{X_0}.
\end{tikzcd}
\]
By the axioms of a triangulated category, a morphism of the first two terms of distinguished triangles extends to a morphism of distinguished triangles \cite{BBD82,KashiwaraSchapira90}.  Therefore the square induces a morphism between the shifted cones:
\[
   P^H=\operatorname{Cone}(\operatorname{var})[-1]
   \longrightarrow
   P^H_I=\operatorname{Cone}(\operatorname{var}_I)[-1].
\]

We now identify the cone of this morphism.  Apply the octahedral axiom in $D^b\operatorname{MHM}(X_0)$ to the composable morphisms
\[
   \phi_\pi(F)
   \xrightarrow{\operatorname{var}}
   \psi_\pi(F)
   \xrightarrow{\operatorname{pr}_I}
   \mathcal{IS}^{H}_{X_0}.
\]
Their composition is
\[
   \operatorname{pr}_I\circ\operatorname{var}
   =
   \operatorname{var}_I.
\]
The projection triangle associated with the direct-sum decomposition
\[
   \psi_\pi(F)\simeq \mathcal{IS}^{H}_{X_0}\oplus\mathcal C^H_\Sigma
\]
is
\[
   \mathcal C^H_\Sigma
   \longrightarrow
   \psi_\pi(F)
   \xrightarrow{\operatorname{pr}_I}
   \mathcal{IS}^{H}_{X_0}
   \longrightarrow
   \mathcal C^H_\Sigma[1].
\]
Thus
\[
   \operatorname{Cone}(\operatorname{pr}_I)\simeq \mathcal C^H_\Sigma[1].
\]

The octahedral axiom gives a distinguished triangle
\[
   \operatorname{Cone}(\operatorname{var})
   \longrightarrow
   \operatorname{Cone}(\operatorname{var}_I)
   \longrightarrow
   \operatorname{Cone}(\operatorname{pr}_I)
   \xrightarrow{+1}.
\]
Substituting
\[
   \operatorname{Cone}(\operatorname{var})=P^H[1],
   \qquad
   \operatorname{Cone}(\operatorname{var}_I)=P^H_I[1],
   \qquad
   \operatorname{Cone}(\operatorname{pr}_I)=\mathcal C^H_\Sigma[1],
\]
we obtain
\[
   P^H[1]\longrightarrow P^H_I[1]\longrightarrow \mathcal C^H_\Sigma[1]\xrightarrow{+1}.
\]
Shifting by $[-1]$ gives the distinguished triangle
\[
   P^H\longrightarrow P^H_I\longrightarrow \mathcal C^H_\Sigma\xrightarrow{+1}.
\]
This proves the theorem.
\end{proof}

\begin{remark}[Interpretation of the projection triangle]
\label{rem:projection-triangle-interpretation}
The projection triangle identifies $\mathcal C^H_\Sigma$ as the cofiber measuring the difference between the full corrected conifold object $P^H$ and its intersection-space projection $P^H_I$.  This is a statement in $D^b\operatorname{MHM}(X_0)$, not merely a numerical identity after taking cohomology.  After hypercohomology, the triangle gives a long exact sequence of mixed Hodge structures and hence a filtered atom-shadow statement, not necessarily a direct-sum identity unless the triangle splits.
\end{remark}

\subsection{Interpretation}
\label{subsec:specialization-projection-interpretation}

The conditional splitting identifies the nearby-smoothing atom shadow as a sum of two pieces.  The first piece,
\[
   \mathsf{HA}^{I}(X_0),
\]
is the intersection-space atom package and carries the type-IIB realization.  The second piece,
\[
   \mathfrak A_{\mathrm{Hod}}(\mathcal C^H_{\Sigma}),
\]
is supported at the ordinary double points and records the local correction needed to recover the nearby-smoothing object from the intersection-space summand.

The projection triangle strengthens this interpretation.  It says that the same singularity-supported complement also measures the difference between the corrected conifold object and its intersection-space projection:
\[
   P^H\longrightarrow P^H_I\longrightarrow \mathcal C^H_\Sigma\xrightarrow{+1}.
\]
Thus $\mathcal C^H_\Sigma$ is the finite-node Hodge-theoretic residue of the projection from corrected nearby/vanishing-cycle geometry to the intersection-space sector.

This gives a sheaf-theoretic interpretation of the intersection-space package as a direct summand of the smoothing realization whenever Hypothesis~\ref{hyp:S-specialization-splitting} holds.  It also clarifies why the construction is different from asserting a quantum theory on $I^{\bar m}X_0$: the comparison is made through mixed Hodge modules on $X_0$, nearby and vanishing cycles, and the projection triangle, not through stable maps into the spatial intersection space.
%-------------------------------------
\section{The IC--Intersection-Space Defect}
\label{sec:IC-IS-defect}

We now isolate the defect object associated with the comparison between the intersection-complex realization and the intersection-space realization.  The previous sections constructed the type-IIA/intersection-homology atom package
\[
   \mathsf{HA}^{IH}(X_0)
   =
   \operatorname{Atom}_{\mathrm{Hod}}
   \bigl(\mathbb H^*(X_0;IC^H_{X_0})\bigr)
\]
and the type-IIB/intersection-space atom package
\[
   \mathsf{HA}^{I}(X_0)
   =
   \operatorname{Atom}_{\mathrm{Hod}}
   \bigl(\mathbb H^*(X_0;\mathcal{IS}^{H}_{X_0})\bigr).
\]
Both are Hodge-realization atom shadows of mixed-Hodge-module objects on the same singular fiber $X_0$.  By Proposition~\ref{prop:exterior-agreement}, these objects agree on the smooth locus
\[
   U=X_0\setminus\Sigma.
\]
Thus their difference is controlled by the singular gluing data at the finite ordinary double point set $\Sigma$.

\subsection{The Grothendieck defect}
\label{subsec:grothendieck-defect}

The mixed-Hodge-module-level defect is the formal difference between the intersection-space mixed Hodge module and the intersection-complex mixed Hodge module.

\begin{definition}[IC--intersection-space Grothendieck defect]
\label{def:IC-IS-grothendieck-defect}
Assume Hypotheses H1--H3.  The IC--intersection-space Grothendieck defect of $X_0$ is
\[
   \Delta_{I/IC}(X_0)
   :=
   [\mathcal{IS}^{H}_{X_0}]-[IC^H_{X_0}]
   \in K_0(\operatorname{MHM}(X_0)).
\]
\end{definition}

Here $K_0(\operatorname{MHM}(X_0))$ denotes the Grothendieck group of Saito's category of mixed Hodge modules on $X_0$ \cite{Saito90,Saito_MHM,Saito_MixedHodgeModules}.  The defect is defined at the mixed-Hodge-module level before passing to hypercohomology or atom shadows.

\begin{proposition}[Exterior vanishing of the defect]
\label{prop:exterior-vanishing-defect}
Assume Hypotheses H1--H3.  The restriction of the defect to the smooth locus vanishes:
\[
   j^*\Delta_{I/IC}(X_0)=0
   \qquad
   \text{in }
   K_0(\operatorname{MHM}(U)),
\]
where $j:U\hookrightarrow X_0$ is the inclusion of the smooth locus.
\end{proposition}

\begin{proof}
By Proposition~\ref{prop:exterior-agreement}, the restrictions of the two mixed-Hodge-module objects to $U$ are both isomorphic to the normalized constant Hodge module:
\[
   j^*IC^H_{X_0}\simeq \mathbb Q^H_U[3],
   \qquad
   j^*\mathcal{IS}^{H}_{X_0}\simeq \mathbb Q^H_U[3].
\]
Therefore their classes in $K_0(\operatorname{MHM}(U))$ are equal:
\[
   [j^*\mathcal{IS}^{H}_{X_0}]
   =
   [j^*IC^H_{X_0}].
\]
Taking the difference gives
\[
   j^*\bigl([\mathcal{IS}^{H}_{X_0}]-[IC^H_{X_0}]\bigr)=0.
\]
This is exactly
\[
   j^*\Delta_{I/IC}(X_0)=0.
\]
\end{proof}

\begin{remark}[Singular support of the defect]
\label{rem:singular-support-defect}
Proposition~\ref{prop:exterior-vanishing-defect} says that the IC--intersection-space defect has no exterior contribution.  It is therefore controlled by the singular extension and gluing data at $\Sigma$, together with the middle-dimensional conifold contribution isolated by the rigid--vanishing filtration.
\end{remark}

\subsection{The atom-shadow defect}
\label{subsec:atom-shadow-defect}

Passing to hypercohomology and then to Hodge-realization atom shadows gives the corresponding atom-level defect.

\begin{definition}[IC--intersection-space atom defect]
\label{def:IC-IS-atom-defect}
Assume Hypotheses H1--H3.  The IC--intersection-space atom defect is the formal difference
\[
   \mathsf{Def}_{I/IC}(X_0)
   :=
   \mathsf{HA}^{I}(X_0)-\mathsf{HA}^{IH}(X_0).
\]
\end{definition}

This expression is not a morphism in a category of atoms.  It is a formal difference of Hodge-realization atom shadows induced by the two mixed-Hodge-module realizations
\[
   \mathcal{IS}^{H}_{X_0}
   \qquad
   \text{and}
   \qquad
   IC^H_{X_0}.
\]

\begin{proposition}[Atom defect as realization of the Grothendieck defect]
\label{prop:atom-defect-realizes-grothendieck-defect}
Assume Hypotheses H1--H3.  The atom defect $\mathsf{Def}_{I/IC}(X_0)$ is obtained from the Grothendieck defect $\Delta_{I/IC}(X_0)$ by applying hypercohomology and then the Hodge-realization atom-shadow construction.
\end{proposition}

\begin{proof}
By Definition~\ref{def:IC-IS-grothendieck-defect},
\[
   \Delta_{I/IC}(X_0)
   =
   [\mathcal{IS}^{H}_{X_0}]-[IC^H_{X_0}].
\]
Applying hypercohomology gives the formal difference of graded mixed Hodge structures
\[
   \mathbb H^*(X_0;\mathcal{IS}^{H}_{X_0})
   -
   \mathbb H^*(X_0;IC^H_{X_0}).
\]
Applying $\operatorname{Atom}_{\mathrm{Hod}}$ gives
\[
   \operatorname{Atom}_{\mathrm{Hod}}
   \bigl(\mathbb H^*(X_0;\mathcal{IS}^{H}_{X_0})\bigr)
   -
   \operatorname{Atom}_{\mathrm{Hod}}
   \bigl(\mathbb H^*(X_0;IC^H_{X_0})\bigr).
\]
By definition, these two terms are
\[
   \mathsf{HA}^{I}(X_0)
   \qquad
   \text{and}
   \qquad
   \mathsf{HA}^{IH}(X_0).
\]
Thus the resulting formal atom-level difference is
\[
   \mathsf{Def}_{I/IC}(X_0)
   =
   \mathsf{HA}^{I}(X_0)-\mathsf{HA}^{IH}(X_0).
\]
\end{proof}

\subsection{Relation with the projection triangle}
\label{subsec:defect-relation-projection-triangle}

The projection triangle gives a second, map-level source for the defect viewpoint.  Under Hypothesis S, the nearby-cycle object splits as
\[
   \psi_\pi(F)
   \simeq
   \mathcal{IS}^{H}_{X_0}\oplus\mathcal C^H_\Sigma.
\]
The projection of the variation morphism defines
\[
   P^H_I
   =
   \operatorname{Cone}(\operatorname{var}_I)[-1],
\]
and Theorem~\ref{thm:projection-triangle} gives a distinguished triangle
\[
   P^H\longrightarrow P^H_I\longrightarrow \mathcal C^H_\Sigma\xrightarrow{+1}.
\]

\begin{proposition}[Complement as projection defect]
\label{prop:complement-as-projection-defect}
Assume Hypotheses H1--H3 and Hypothesis S.  In the projection triangle
\[
   P^H\longrightarrow P^H_I\longrightarrow \mathcal C^H_\Sigma\xrightarrow{+1},
\]
the singularity-supported complement $\mathcal C^H_\Sigma$ is the cofiber measuring the difference between the corrected conifold object $P^H$ and its intersection-space projection $P^H_I$.
\end{proposition}

\begin{proof}
This is precisely the content of Theorem~\ref{thm:projection-triangle}.  That theorem applies the octahedral axiom to the composable morphisms
\[
   \phi_\pi(F)\xrightarrow{\operatorname{var}}\psi_\pi(F)
   \xrightarrow{\operatorname{pr}_I}\mathcal{IS}^{H}_{X_0}.
\]
The composition is
\[
   \operatorname{var}_I=\operatorname{pr}_I\circ\operatorname{var}.
\]
The cone of the projection $\operatorname{pr}_I$ is the singularity-supported complement, up to the cone convention already accounted for in Theorem~\ref{thm:projection-triangle}.  Hence the resulting distinguished triangle is
\[
   P^H\to P^H_I\to \mathcal C^H_\Sigma\xrightarrow{+1}.
\]
Thus $\mathcal C^H_\Sigma$ is the cofiber of the morphism from the corrected conifold object to its intersection-space projection.
\end{proof}

\begin{remark}[Two related defect objects]
\label{rem:two-related-defect-objects}
The Grothendieck defect $\Delta_{I/IC}(X_0)$ and the projection complement $\mathcal C^H_\Sigma$ are related but not identical objects.  The first measures the difference between the intersection-space realization and the intersection-complex realization.  The second measures the difference between the corrected conifold object and its intersection-space projection.  Both are finite-node Hodge-theoretic defect objects on the singular fiber, and both are controlled by the same conifold singular geometry.
\end{remark}

\subsection{BPS-shadow interpretation}
\label{subsec:BPS-shadow-interpretation}

The IC--intersection-space defect and the projection complement provide natural Hodge-theoretic candidates for future comparison with refined Donaldson--Thomas and BPS wall-crossing structures.  The present paper stops at the singular-fiber Hodge-realization level.  A DT/BPS interpretation requires an additional functorial passage to moduli stacks of sheaves, complexes, quiver representations, perverse coherent sheaves, or stability-condition objects.

Thus the role of the present paper is to identify the geometry-side handoff objects
\[
   \Delta_{I/IC}(X_0),
   \qquad
   \mathcal C^H_\Sigma,
   \qquad
   P^H\to P^H_I\to \mathcal C^H_\Sigma\xrightarrow{+1}.
\]
The moduli-stack realization and the comparison with refined or cohomological DT/BPS wall-crossing data are left to future work.
%%%%%%%%%%%%%%%%%%%%%%%%%%%%%%%%%%%%%%%%%%%%%%%%%%%%%%%%%%%%%%%%%%
\section{The Classical Nodal Quintic}
\label{sec:classical-nodal-quintic}

We conclude with the standard nodal quintic example.  The purpose of this section is not to rederive the topology of the nodal quintic, but to show how the intersection-space atom package records a concrete type-IIB middle contribution that is invisible to the intersection-homology package.  This example gives the numerical anchor for the preceding formalism and for the IC--intersection-space defect viewpoint.

\subsection{Setup and numerical input}
\label{subsec:nodal-quintic-setup}

Let $S$ denote the classical quintic conifold in $\mathbb P^4$ with $125$ ordinary double points.  One standard model is the singular member of the Dwork pencil
\cite{COGP,HubschBestiary2nd2024}
\[
   S
   =
   \left\{
   z_0^5+z_1^5+z_2^5+z_3^5+z_4^5
   -5z_0z_1z_2z_3z_4=0
   \right\}
   \subset \mathbb P^4 .
\]
At the conifold value, this hypersurface has $125$ isolated ordinary double points
\cite{COGP,HubschBestiary2nd2024,Banagl_IntersectionSpacesBook}.  Each singularity is locally analytically equivalent to
\[
   x_0^2+x_1^2+x_2^2+x_3^2=0,
\]
and hence is a weighted homogeneous isolated hypersurface singularity \cite{Milnor68,Clemens83}.

We use the following standard numerical input for this example:
\[
   \dim H_3(IS;\mathbb Q)=204,
   \qquad
   \dim IH_3(S;\mathbb Q)=2.
\]
Here $IS$ denotes Banagl's middle-perversity intersection space of the nodal quintic, and $IH_3(S;\mathbb Q)$ denotes middle-dimensional intersection homology.  These values are cited from the standard nodal-quintic and intersection-space computations
\cite{Banagl_IntersectionSpacesBook,BanaglBudurMaxim14,HubschBestiary2nd2024}.  They are the numerical source of the middle-degree difference
\[
   204-2=202.
\]

\begin{remark}[Use of numerical input]
\label{rem:nodal-quintic-numerical-input}
The present paper uses the nodal quintic as a test case for the Hodge-realization formalism.  The assertions that $S$ has $125$ ordinary double points, that
\[
   \dim H_3(IS;\mathbb Q)=204,
\]
and that
\[
   \dim IH_3(S;\mathbb Q)=2
\]
are taken as standard cited inputs.  The contribution of this section is to translate those numerical inputs into the language of intersection-space Hodge atom packages and the IC--intersection-space defect.
\end{remark}

\subsection{Verification of the sheaf-theoretic input}
\label{subsec:nodal-quintic-sheaf-theoretic-input}

Before applying the atom formalism, one must verify which hypotheses from the preceding sections are being used.  Locally, the nodal quintic satisfies the ordinary-double-point hypothesis: every singularity is an ordinary double point and hence a weighted homogeneous isolated hypersurface singularity.  Therefore the local Banagl--Budur--Maxim intersection-space complex, local mixed-Hodge-module lift, local Verdier self-duality, and local specialization input apply at each of the $125$ nodes by Theorem~\ref{thm:bbm-one-singularity} and Proposition~\ref{prop:nodal-quintic-local-hypotheses}.

The remaining issues are global.  They are exactly the hypotheses isolated earlier: the local intersection-space complexes must glue to a global perverse sheaf, and this global complex must admit a mixed-Hodge-module lift.  We state the resulting input explicitly.

\begin{proposition}[Sheaf-theoretic input for the nodal quintic]
\label{prop:sheaf-theoretic-input-nodal-quintic}
Assume Hypotheses~\ref{hyp:H2-multi-node-gluing} and~\ref{hyp:H3-MHM-lift} for the classical $125$-node quintic $S$.  Then there exists a mixed Hodge module
\[
   \mathcal{IS}^{H}_{S}\in\operatorname{MHM}(S)
\]
such that
\[
   \operatorname{rat}(\mathcal{IS}^{H}_{S})\simeq \mathcal{IS}_{S},
   \qquad
   \mathbb H^*(S;\mathcal{IS}_{S})
   \cong
   H^*(IS;\mathbb Q).
\]
Consequently, the intersection-space Hodge atom package
\[
   \mathsf{HA}^{I}(S)
   :=
   \operatorname{Atom}_{\mathrm{Hod}}
   \bigl(\mathbb H^*(S;\mathcal{IS}^{H}_{S})\bigr)
\]
is well-defined.
\end{proposition}

\begin{proof}
By Proposition~\ref{prop:nodal-quintic-local-hypotheses}, every singularity of $S$ is an ordinary double point.  Therefore the local Banagl--Budur--Maxim input applies at every node.  Since the node set is finite, one may choose pairwise disjoint Milnor neighborhoods around the nodes.

Hypothesis~\ref{hyp:H2-multi-node-gluing} gives a global perverse sheaf
\[
   \mathcal{IS}_{S}\in\operatorname{Perv}(S;\mathbb Q)
\]
whose hypercohomology computes the rational cohomology of Banagl's middle-perversity intersection space:
\[
   \mathbb H^*(S;\mathcal{IS}_{S})
   \cong
   H^*(IS;\mathbb Q).
\]
Hypothesis~\ref{hyp:H3-MHM-lift} gives a mixed Hodge module
\[
   \mathcal{IS}^{H}_{S}\in\operatorname{MHM}(S)
\]
with rationalization
\[
   \operatorname{rat}(\mathcal{IS}^{H}_{S})\simeq \mathcal{IS}_{S}.
\]
Thus the hypercohomology groups
\[
   \mathbb H^*(S;\mathcal{IS}^{H}_{S})
\]
carry mixed Hodge structures by Saito's theory \cite{Saito90,Saito_MHM,Saito_MixedHodgeModules}.  Applying Definition~\ref{def:intersection-space-hodge-atom-package} gives the well-defined atom package
\[
   \mathsf{HA}^{I}(S)
   =
   \operatorname{Atom}_{\mathrm{Hod}}
   \bigl(\mathbb H^*(S;\mathcal{IS}^{H}_{S})\bigr).
\]
\end{proof}

\begin{remark}[Conditional global input]
\label{rem:nodal-quintic-conditional-global-input}
Proposition~\ref{prop:sheaf-theoretic-input-nodal-quintic} is deliberately phrased under the finite-ODP gluing and MHM-lift hypotheses.  Locally, the nodal quintic lies exactly in the weighted homogeneous ordinary-double-point setting.  Globally, the point is to ensure that the $125$ local intersection-space complexes assemble into the mixed-Hodge-module object used to define the atom package.
\end{remark}

\subsection{Middle-degree comparison table}
\label{subsec:nodal-quintic-middle-table}

The present paper uses only the middle-degree comparison between intersection homology and intersection-space homology.  We therefore record the middle ranks directly rather than reproducing a full middle-degree Betti-number table for the singular quintic.  The cited numerical input is
\[
   \dim H_3(IS;\mathbb Q)=204,
   \qquad
   \dim IH_3(S;\mathbb Q)=2,
\]
so the middle-degree IC--intersection-space defect rank is \(204-2=202\).

\begin{table}[h]
\centering
\caption{Middle-degree comparison for the classical $125$-node quintic conifold.}
\begin{tabular}{c|c|c}
Theory & Middle group & Rank \\
\hline
Intersection homology / type-IIA & $IH_3(S;\mathbb Q)$ & $2$ \\
Intersection space / type-IIB & $H_3(IS;\mathbb Q)$ & $204$ \\
IC--intersection-space difference & $H_3(IS;\mathbb Q)-IH_3(S;\mathbb Q)$ & $202$
\end{tabular}
\label{tab:nodal-quintic-middle-comparison}
\end{table}

\begin{remark}[Middle-degree focus]
\label{rem:nodal-quintic-middle-degree-focus}
The present paper uses only the middle-degree comparison.  We therefore record the middle ranks directly rather than reproducing a full middle-degree Betti-number table for the singular quintic.  The key numerical input is
\[
   \dim H_3(IS;\mathbb Q)=204,
   \qquad
   \dim IH_3(S;\mathbb Q)=2.
\]
\end{remark}

\subsection{The rank-\texorpdfstring{$202$}{202} jump}
\label{subsec:nodal-quintic-202-jump}

We now translate the middle-degree rank comparison into the language of Hodge-realization atom packages.

\begin{theorem}[Nodal quintic atom jump]
\label{thm:nodal-quintic-atom-jump}
For the classical $125$-node quintic conifold $S$, assume Hypotheses~\ref{hyp:H2-multi-node-gluing} and~\ref{hyp:H3-MHM-lift}.  Using the standard numerical input
\[
   \dim H_3(IS;\mathbb Q)=204,
   \qquad
   \dim IH_3(S;\mathbb Q)=2,
\]
one has
\[
   \dim \operatorname{rat}\mathsf{HA}^{I}_3(S)
   -
   \dim \operatorname{rat}\mathsf{HA}^{IH}_3(S)
   =
   202.
\]
Equivalently,
\[
   \dim H_3(IS;\mathbb Q)
   -
   \dim IH_3(S;\mathbb Q)
   =
   204-2
   =
   202.
\]
\end{theorem}

\begin{proof}
By Definition~\ref{def:intersection-space-hodge-atom-package},
\[
   \mathsf{HA}^{I}_3(S)
   =
   \operatorname{Atom}_{\mathrm{Hod}}
   \bigl(\mathbb H^3(S;\mathcal{IS}^{H}_{S})\bigr).
\]
By Proposition~\ref{prop:betti-realization-HAI}, the rational realization of
\[
   \mathbb H^3(S;\mathcal{IS}^{H}_{S})
\]
identifies with
\[
   H^3(IS;\mathbb Q).
\]
Since all groups are taken with rational coefficients, the universal coefficient theorem identifies the rank of middle cohomology with the rank of middle homology.  Thus the middle rank of the cohomological intersection-space realization agrees with the cited middle rank of
\[
   H_3(IS;\mathbb Q).
\]
The standard nodal-quintic input gives
\[
   \dim H_3(IS;\mathbb Q)=204.
\]
Therefore
\[
   \dim \operatorname{rat}\mathsf{HA}^{I}_3(S)=204.
\]

Similarly,
\[
   \mathsf{HA}^{IH}_3(S)
   =
   \operatorname{Atom}_{\mathrm{Hod}}
   \bigl(\mathbb H^3(S;IC^H_S)\bigr).
\]
The rational realization of the intersection-complex hypercohomology is intersection cohomology:
\[
   \operatorname{rat}\mathbb H^3(S;IC^H_S)
   \cong
   IH^3(S;\mathbb Q)
\]
\cite{GoreskyMacPherson80,GoreskyMacPherson83,BBD82,Saito90}.
With rational coefficients, the corresponding middle intersection-cohomology and middle intersection-homology groups have the same rank.  Hence the cited value
\[
   \dim IH_3(S;\mathbb Q)=2
\]
gives the middle rank of the intersection-complex realization:
\[
   \dim \operatorname{rat}\mathsf{HA}^{IH}_3(S)=2.
\]
Subtracting gives
\[
   \dim \operatorname{rat}\mathsf{HA}^{I}_3(S)
   -
   \dim \operatorname{rat}\mathsf{HA}^{IH}_3(S)
   =
   204-2
   =
   202.
\]
\end{proof}

\subsection{Defect interpretation}
\label{subsec:nodal-quintic-defect-interpretation}

The nodal quintic example exhibits the central phenomenon of the paper in a single number.  The intersection-homology atom package sees a middle-dimensional rank-$2$ sector.  The intersection-space atom package sees a middle-dimensional rank-$204$ sector.  The difference,
\[
   202,
\]
is the type-IIB middle contribution retained by the intersection-space complex and absent from the intersection-homology package.

Equivalently, the atom jump measures the middle-degree rank-level shadow of the difference between the two mixed-Hodge-module realizations over the same singular fiber:
\[
   IC^H_S
   \qquad\text{and}\qquad
   \mathcal{IS}^H_S.
\]
The first realizes the type-IIA/intersection-homology sector; the second realizes the type-IIB/intersection-space sector.  Thus the $202$-jump is not merely a Betti-number discrepancy.  It is the middle-degree numerical realization of the IC--intersection-space defect
\[
   \Delta_{I/IC}(S)
   =
   [\mathcal{IS}^{H}_{S}]-[IC^H_S]
   \in K_0(\operatorname{MHM}(S)),
\]
after passing to rational middle-degree realizations.

\begin{corollary}[Middle defect rank for the nodal quintic]
\label{cor:nodal-quintic-middle-defect-rank}
For the classical nodal quintic, under the hypotheses and numerical input of Theorem~\ref{thm:nodal-quintic-atom-jump}, the middle-degree rational realization of the IC--intersection-space defect has rank
\[
   202.
\]
\end{corollary}

\begin{proof}
The middle-degree rational realization of the defect is, by definition, the formal difference between the middle-degree rational realization of the intersection-space package and the middle-degree rational realization of the intersection-complex package:
\[
   \dim \operatorname{rat}\mathsf{HA}^{I}_3(S)
   -
   \dim \operatorname{rat}\mathsf{HA}^{IH}_3(S).
\]
By Theorem~\ref{thm:nodal-quintic-atom-jump}, this difference is
\[
   204-2=202.
\]
\end{proof}

\begin{corollary}[IIB vanishing atom for the nodal quintic, conditional]
\label{cor:nodal-quintic-IIB-vanishing-atom}
Assume, in addition, Hypothesis MHS-B for $S$ and suppose that the rigid quotient contribution in Banagl's middle-degree exact sequence is identified with the rank-$2$ middle intersection-homology sector.  Then the middle-degree IIB vanishing atom has rank
\[
   \operatorname{rk}\mathsf{HA}^{I}_{\mathrm{van},3}(S)=202.
\]  
\end{corollary}

\begin{proof}
By Theorem~\ref{thm:rigid-vanishing-filtration}, the middle intersection-space atom package has associated graded
\[
   \operatorname{gr}\mathsf{HA}^{I}_3(S)
   =
   \mathsf{HA}^{I}_{\mathrm{van},3}(S)
   +
   \mathsf{HA}^{I}_{\mathrm{rig},3}(S).
\]
By Definition~\ref{def:middle-rigid-atom}, the rigid middle atom is the atom shadow of the quotient side of Banagl's middle-degree exact sequence.  Under the additional assumption in the statement, this rigid quotient contribution has rank $2$.  The full middle intersection-space rank is $204$.  Therefore the remaining associated-graded contribution has rank
\[
   204-2=202.
\]
This remaining contribution is, by Definition~\ref{def:IIB-vanishing-atom}, the middle-degree IIB vanishing atom.
\end{proof}

\begin{remark}[Why this example matters]
\label{rem:why-nodal-quintic-example-matters}
This example anchors the formalism.  It shows that $\mathsf{HA}^{I}(S)$ is not a cosmetic renaming of the intersection-homology atom package.  It detects a large middle-dimensional type-IIB contribution that is invisible to $\mathsf{HA}^{IH}(S)$.  Thus the Banagl--Budur--Maxim intersection-space complex produces a genuinely different Hodge-realization atom package from the intersection complex, even over the same singular Calabi--Yau fiber.
\end{remark}

\begin{remark}[Relation to the projection-triangle viewpoint]
\label{rem:nodal-quintic-relation-projection-triangle}
When Hypothesis~\ref{hyp:S-specialization-splitting} is also imposed for the nodal quintic, the projection triangle
\[
   P^H\longrightarrow P^H_I\longrightarrow \mathcal C^H_\Sigma\xrightarrow{+1}
\]
provides the corresponding corrected-conifold interpretation of the defect.  The present section uses only the middle-degree atom ranks.  The projection triangle gives the stronger mixed-Hodge-module statement that the singularity-supported complement $\mathcal C^H_\Sigma$ is the cofiber measuring the difference between the corrected conifold object and its intersection-space projection.
\end{remark}

%%%%%%%%%%%%%%%%%%%%%%%%%%%%%%%%%%%%%%%%%%%%%%%%%%%%%%%%%%%%%%%%%%%%
\section{Characteristic-Class and Integral Enhancements}
\label{sec:characteristic-integral-enhancements}

This section records two optional refinements of the intersection-space atom package.  The first is a characteristic-class shadow obtained by applying the motivic Hirzebruch class transformation to the mixed-Hodge-module class of $\mathcal{IS}^H_{X_0}$.  The second is a conditional integral enhancement, available only when the intersection-space realization is equipped with a compatible integral lattice.  Neither refinement is required for the construction of the rational Hodge-realization atom package
\[
   \mathsf{HA}^{I}(X_0).
\]

The guiding principle is the same as in the rest of the paper: all refinements are made at the level of the sheaf-theoretic or mixed-Hodge-module realization.  We do not assert a new characteristic-class theory or a canonical integral theory for the spatial intersection space itself.

\subsection{MHM characteristic-class shadow}
\label{subsec:MHM-characteristic-class-shadow}

The mixed-Hodge-module realization of the intersection-space complex allows one to attach characteristic classes by applying characteristic-class transformations for mixed Hodge modules \cite{BrasseletSchurmannYokura_HirzebruchClasses,Schurmann_TopologySingularSpaces,Saito90}.  We keep this construction at the level of a shadow: it records the class-theoretic image of $\mathcal{IS}^{H}_{X_0}$ but does not introduce a new theory of characteristic classes for the spatial intersection space itself.

Let
\[
   K_0(\operatorname{MHM}(X_0))
\]
denote the Grothendieck group of the category of mixed Hodge modules on $X_0$.  The motivic Hirzebruch class transformation
\[
   MHT_{y*}
\]
is additive on this Grothendieck group and takes values in the appropriate homology or Chow-theoretic target used for Hirzebruch classes \cite{BrasseletSchurmannYokura_HirzebruchClasses,Schurmann_TopologySingularSpaces}.  We use only this additivity and functorial class-theoretic behavior.

\begin{definition}[Intersection-space Hirzebruch shadow]
\label{def:intersection-space-Hirzebruch-shadow}
Assume Hypotheses~\ref{hyp:H1-local-BBM-input}--\ref{hyp:H3-MHM-lift}.  The intersection-space Hirzebruch shadow of $X_0$ is
\[
   IT^{I}_{y*}(X_0)
   :=
   MHT_{y*}\bigl([\mathcal{IS}^{H}_{X_0}]\bigr).
\]
Here
\[
   [\mathcal{IS}^{H}_{X_0}]\in K_0(\operatorname{MHM}(X_0))
\]
is the Grothendieck class of the intersection-space mixed Hodge module.
\end{definition}

The notation $IT^I_{y*}(X_0)$ is meant to emphasize that this is the intersection-space analogue, at the mixed-Hodge-module realization level, of the intersection-homology Hirzebruch class attached to the intersection complex.

For comparison, define the intersection-homology Hirzebruch shadow by
\[
   IT_{y*}(X_0)
   :=
   MHT_{y*}\bigl([IC^H_{X_0}]\bigr).
\]
This is the characteristic-class shadow attached to the type-IIA/intersection-homology package.

\begin{definition}[Characteristic-class defect]
\label{def:characteristic-class-defect}
The characteristic-class shadow of the IC--intersection-space defect is
\[
   \Delta T_{y*}^{I/IC}(X_0)
   :=
   IT^{I}_{y*}(X_0)-IT_{y*}(X_0).
\]
Equivalently,
\[
   \Delta T_{y*}^{I/IC}(X_0)
   :=
   MHT_{y*}\bigl([\mathcal{IS}^{H}_{X_0}]\bigr)
   -
   MHT_{y*}\bigl([IC^H_{X_0}]\bigr).
\]
\end{definition}

\begin{proposition}[Characteristic-class comparison]
\label{prop:characteristic-class-comparison}
Assume Hypotheses~\ref{hyp:H1-local-BBM-input}--\ref{hyp:H3-MHM-lift}.  Then
\[
   IT^{I}_{y*}(X_0)-IT_{y*}(X_0)
   =
   MHT_{y*}\bigl([\mathcal{IS}^{H}_{X_0}]-[IC^H_{X_0}]\bigr).
\]
Equivalently,
\[
   \Delta T_{y*}^{I/IC}(X_0)
   =
   MHT_{y*}\bigl(\Delta_{I/IC}(X_0)\bigr),
\]
where
\[
   \Delta_{I/IC}(X_0)
   :=
   [\mathcal{IS}^{H}_{X_0}]-[IC^H_{X_0}]
   \in K_0(\operatorname{MHM}(X_0)).
\]
\end{proposition}

\begin{proof}
By Definition~\ref{def:intersection-space-Hirzebruch-shadow},
\[
   IT^{I}_{y*}(X_0)
   =
   MHT_{y*}\bigl([\mathcal{IS}^{H}_{X_0}]\bigr).
\]
By definition of the intersection-homology Hirzebruch shadow,
\[
   IT_{y*}(X_0)
   =
   MHT_{y*}\bigl([IC^H_{X_0}]\bigr).
\]
The motivic Hirzebruch class transformation is additive on the Grothendieck group of mixed Hodge modules \cite{BrasseletSchurmannYokura_HirzebruchClasses,Schurmann_TopologySingularSpaces}.  Therefore
\[
\begin{aligned}
   IT^{I}_{y*}(X_0)-IT_{y*}(X_0)
   &=
   MHT_{y*}\bigl([\mathcal{IS}^{H}_{X_0}]\bigr)
   -
   MHT_{y*}\bigl([IC^H_{X_0}]\bigr)  \\
   &=
   MHT_{y*}\bigl([\mathcal{IS}^{H}_{X_0}]-[IC^H_{X_0}]\bigr).
\end{aligned}
\]
The class
\[
   [\mathcal{IS}^{H}_{X_0}]-[IC^H_{X_0}]
\]
is precisely $\Delta_{I/IC}(X_0)$.  This proves the claimed equality.
\end{proof}

\begin{remark}[Scope of the characteristic-class shadow]
\label{rem:scope-characteristic-class-shadow}
The class $IT^I_{y*}(X_0)$ is not asserted to be a characteristic class of the spatial intersection space as a topological space.  It is the motivic Hirzebruch class shadow of the mixed-Hodge-module realization $\mathcal{IS}^{H}_{X_0}$.  This is the appropriate level for comparison with $IC^H_{X_0}$, with the Hodge-realization atom package, and with the IC--intersection-space defect.
\end{remark}

\begin{corollary}[Exterior cancellation of the characteristic-class defect]
\label{cor:exterior-cancellation-characteristic-defect}
Assume Hypotheses~\ref{hyp:H1-local-BBM-input}--\ref{hyp:H3-MHM-lift}.  If
\[
   j^*\mathcal{IS}^{H}_{X_0}\simeq j^*IC^H_{X_0}
\]
on
\[
   U=X_0\setminus\Sigma,
\]
then the restriction of the Grothendieck defect
\[
   \Delta_{I/IC}(X_0)
   =
   [\mathcal{IS}^{H}_{X_0}]-[IC^H_{X_0}]
\]
to $U$ vanishes.  Consequently, the characteristic-class defect
\[
   \Delta T_{y*}^{I/IC}(X_0)
\]
is controlled by the singular gluing contribution at $\Sigma$ and by the associated middle-degree conifold correction.
\end{corollary}

\begin{proof}
By Proposition~\ref{prop:exterior-agreement}, both $\mathcal{IS}^{H}_{X_0}$ and $IC^H_{X_0}$ restrict to the normalized constant Hodge module on the smooth locus:
\[
   j^*\mathcal{IS}^{H}_{X_0}\simeq \mathbb Q^H_U[3],
   \qquad
   j^*IC^H_{X_0}\simeq \mathbb Q^H_U[3].
\]
Therefore
\[
   j^*\mathcal{IS}^{H}_{X_0}\simeq j^*IC^H_{X_0}.
\]
Passing to the Grothendieck group gives
\[
   j^*\bigl([\mathcal{IS}^{H}_{X_0}]-[IC^H_{X_0}]\bigr)=0.
\]
Thus the defect class has no exterior contribution.

By Proposition~\ref{prop:characteristic-class-comparison}, the characteristic-class defect is the image of this Grothendieck defect under $MHT_{y*}$.  Since the source defect vanishes on $U$, its class-theoretic image is controlled by the complement to the exterior stratum, namely the singular gluing contribution at $\Sigma$ and the associated middle conifold correction.  This proves the claim.
\end{proof}

\begin{remark}[Relation to the projection triangle]
\label{rem:characteristic-defect-projection-triangle}
When Hypothesis~\ref{hyp:S-specialization-splitting} holds, the projection triangle
\[
   P^H\longrightarrow P^H_I\longrightarrow \mathcal C^H_\Sigma\xrightarrow{+1}
\]
gives a second source of characteristic-class shadows.  Applying $MHT_{y*}$ to the Grothendieck classes of the three terms gives an additive class relation in the target of the motivic Hirzebruch transformation.  This relation belongs to the corrected nearby/vanishing-cycle side of the paper, while $\Delta T_{y*}^{I/IC}(X_0)$ records the IC--intersection-space defect.
\end{remark}

\subsection{Integral enhancement, conditional}
\label{subsec:integral-enhancement-conditional}

The atom package constructed in this paper is rational and Hodge-theoretic.  It is defined from
\[
   \mathbb H^*(X_0;\mathcal{IS}^{H}_{X_0}),
\]
whose rational realization is
\[
   H^*(I^{\bar m}X_0;\mathbb Q).
\]
In some situations, one may also have a compatible integral lattice.  Such a lattice is extra structure and is not part of the basic rational MHM package.

\begin{definition}[Conditional integral atom enhancement]
\label{def:conditional-integral-atom-enhancement}
Suppose that
\[
   \mathbb H^*(X_0;\mathcal{IS}^{H}_{X_0})
\]
is equipped with a compatible integral lattice $\Lambda_I$.  Define
\[
   \mathsf{HA}^{I,\mathbb Z}(X_0)
   :=
   \bigl(\mathsf{HA}^{I}(X_0),\Lambda_I\bigr).
\]
\end{definition}

The compatibility condition means that $\Lambda_I$ should recover the rational realization after tensoring with $\mathbb Q$:
\[
   \Lambda_I\otimes_{\mathbb Z}\mathbb Q
   \cong
   \operatorname{rat}\mathbb H^*(X_0;\mathcal{IS}^{H}_{X_0})
   \cong
   H^*(I^{\bar m}X_0;\mathbb Q),
\]
and should be compatible with any filtrations or pairings included in the enhanced atom package.

\begin{remark}[Additional structure]
\label{rem:integral-enhancement-additional-structure}
The integral enhancement is additional structure and is not required for the rational mixed-Hodge-module atom package.  The rational package is defined as soon as the mixed-Hodge-module lift $\mathcal{IS}^{H}_{X_0}$ exists.
\end{remark}

\begin{remark}[No canonical integral assertion]
\label{rem:no-canonical-integral-assertion}
Definition~\ref{def:conditional-integral-atom-enhancement} is conditional.  We do not assert that the mixed-Hodge-module lift $\mathcal{IS}^{H}_{X_0}$ canonically determines an integral lattice, nor do we assert that different spatial models of $I^{\bar m}X_0$ yield canonically identical integral structures.  The integral package is included only to record the additional data available in examples where such a lattice has been specified.
\end{remark}

\begin{proposition}[Forgetting the integral structure]
\label{prop:forgetting-integral-structure}
There is a natural forgetful operation
\[
   \mathsf{HA}^{I,\mathbb Z}(X_0)\longmapsto \mathsf{HA}^{I}(X_0)
\]
obtained by forgetting the lattice $\Lambda_I$.
\end{proposition}

\begin{proof}
By Definition~\ref{def:conditional-integral-atom-enhancement},
\[
   \mathsf{HA}^{I,\mathbb Z}(X_0)
   =
   \bigl(\mathsf{HA}^{I}(X_0),\Lambda_I\bigr).
\]
The forgetful operation sends the ordered pair
\[
   \bigl(\mathsf{HA}^{I}(X_0),\Lambda_I\bigr)
\]
to its first component
\[
   \mathsf{HA}^{I}(X_0).
\]
No additional structure or theorem is required.
\end{proof}

\begin{remark}[Relation to future integral refinements]
\label{rem:future-integral-refinements}
The conditional integral enhancement is included to indicate a possible bridge to integral and motivic refinements, including applications to integral Hodge-theoretic defect packages.  Those questions require additional input beyond the rational intersection-space MHM package and are not developed in this paper.
\end{remark}

%%%%%%%%%%%%%%%%%%%%%%%%%%%%%%%%%%%%%%%%%%%%%%%%%%%%%%%%%%%%%%%%%%%
\section{Outlook}
\label{sec:outlook}

We close by indicating several directions in which the intersection-space atom package, the projection triangle, and the IC--intersection-space defect could be extended.  These directions are not needed for the results of the present paper.  Their purpose is to clarify how the construction fits into a broader program: first isolate finite-node Hodge-theoretic defect objects on the singular fiber, and only then ask how those objects should be transported to mirror, motivic, integral, or moduli-stack settings.

\subsection{Toward genuine mirror-transition atom matching}
\label{subsec:outlook-mirror-transition}

The present paper compares the type-IIA and type-IIB atom packages over a single conifold singular fiber $X_0$.  A stronger result would compare the type-IIA atom package of one conifold degeneration with the type-IIB atom package of its mirror.  Such a statement would require a specified mirror pair, a precise matching of the relevant intersection-homology and intersection-space realizations, and compatibility with the mixed Hodge structures used to define atom shadows.

More concretely, one would need a mirror conifold pair
\[
   X_0
   \qquad\text{and}\qquad
   X_0^\vee
\]
together with a realization-level identification between the intersection-homology sector of one side and the intersection-space sector of the other.  Only after such a realization-level matching is fixed can one apply
\[
   \operatorname{Atom}_{\mathrm{Hod}}
\]
to obtain a genuine mirror-transition atom comparison.  This direction is compatible with Banagl's IIA/IIB interpretation of intersection homology and intersection spaces for conifolds \cite{Banagl_IntersectionSpacesBook,BanaglBudurMaxim14,Strominger95,HubschBestiary2nd2024}, but it lies beyond the internal single-fiber comparison carried out in this paper.

\subsection{Toward categorical atom shadows}
\label{subsec:outlook-categorical-atom-shadows}

Atoms are treated here as shadows of categorical objects, not as objects of an independent atom category.  The relevant categorical inputs are mixed Hodge modules, derived mixed Hodge modules, perverse sheaves, and the nearby/vanishing-cycle objects appearing in the corrected conifold package.  The atom package is extracted only after taking the Hodge realization of these objects.

A future formalism could make this systematic by assigning atom shadows to a broader class of categorical inputs, including mixed Hodge modules, perverse schobers, motivic realizations, and integral Hodge defect packages.  Such a formalism would need to preserve the level discipline used here: the categorical operations occur before atomization, and the atom package records the resulting Hodge-realization data.  In particular, direct sums, distinguished triangles, filtrations, and Verdier duality should first be formulated in the appropriate category, and only then passed to atom shadows through hypercohomology or another specified realization functor \cite{BBD82,KashiwaraSchapira90,Saito90,Saito_MHM,Schurmann_TopologySingularSpaces}.

\subsection{Toward quantum/F-bundle enhancement}
\label{subsec:outlook-quantum-F-bundle}

This paper deliberately avoids asserting a full A-model $F$-bundle on the spatial intersection space
\[
   I^{\bar m}X_0.
\]
A full enhancement would require additional quantum-cohomological input, including a quantum product, Dubrovin connection, and Euler-field spectral decomposition \cite{KKPY25,Givental96,Iritani09}.  These structures are not constructed here for the spatial intersection space.

A more realistic route is to compare the intersection-space atom shadow with genuine $F$-bundles coming from nearby smoothings and small resolutions:
\[
   X_s \rightsquigarrow X_0 \leftarrow \widetilde X.
\]
In this approach, the intersection-space package would appear as a Hodge-theoretic shadow of nearby smooth or resolved geometries rather than as a quantum theory intrinsic to the spatial intersection space.  The projection triangle
\[
   P^H\longrightarrow P^H_I\longrightarrow \mathcal C^H_\Sigma\xrightarrow{+1}
\]
suggests a precise object to compare with such smooth or resolved quantum packages: the singularity-supported complement $\mathcal C^H_\Sigma$ and the projected corrected object $P^H_I$ record how the corrected conifold object is seen through the intersection-space summand.

\subsection{Toward motivic, characteristic-class, and integral defect atoms}
\label{subsec:outlook-motivic-integral-defect-atoms}

The same atom-shadow method may be applied to motivic, characteristic-class, and integral refinements.  In the present paper, the characteristic-class shadow is obtained by applying the motivic Hirzebruch class transformation to the mixed-Hodge-module class
\[
   [\mathcal{IS}^H_{X_0}]
   \in K_0(\operatorname{MHM}(X_0)).
\]
The IC--intersection-space characteristic-class defect is then the image of
\[
   \Delta_{I/IC}(X_0)
   =
   [\mathcal{IS}^H_{X_0}]-[IC^H_{X_0}]
\]
under the same characteristic-class transformation.  This is a realization-level construction, not a new intrinsic characteristic-class theory for the spatial intersection space \cite{BrasseletSchurmannYokura_HirzebruchClasses,Schurmann_TopologySingularSpaces,Saito90}.

Integral refinements require additional data.  The mixed-Hodge-module package gives a rational Hodge realization, but it does not by itself canonically specify an integral lattice.  Thus an integral atom package requires a compatible lattice
\[
   \Lambda_I\subset \mathbb H^*(X_0;\mathcal{IS}^H_{X_0})
\]
whose rationalization recovers the rational intersection-space realization.  Such integral packages may be useful for future integral Hodge-theoretic defect constructions, but they require input beyond the rational mixed-Hodge-module framework used in this paper \cite{Deligne72,Saito90,Schurmann_TopologySingularSpaces}.

\subsection{Toward moduli-stack transport and refined DT/BPS wall crossing}
\label{subsec:outlook-DT-BPS-wall-crossing}

The projection triangle identifies a finite-node Hodge-theoretic object that should be relevant for later BPS and wall-crossing comparisons:
\[
   P^H\longrightarrow P^H_I\longrightarrow \mathcal C^H_\Sigma\xrightarrow{+1}.
\]
Equivalently, at the level of Grothendieck classes, the IC--intersection-space defect
\[
   \Delta_{I/IC}(X_0)
   =
   [\mathcal{IS}^H_{X_0}]-[IC^H_{X_0}]
\]
records the difference between the type-IIB/intersection-space realization and the type-IIA/intersection-complex realization.

The present paper stops at the singular-fiber Hodge-realization level; the moduli-stack realization is the additional datum needed before a comparison with refined or cohomological Donaldson--Thomas wall-crossing can be made. Such objects naturally live on, or are extracted from, moduli spaces or stacks of sheaves, complexes, quiver representations, perverse coherent sheaves, or stability-condition objects.  Thus a refined DT/BPS comparison requires an additional functorial passage from finite-node singular-fiber geometry to moduli-stack geometry.

The expected future problem is therefore not simply to compare the rank of $\mathcal C^H_\Sigma$ with a local conifold BPS number.  Rather, one should first construct or identify a stack-side object associated to the defect triangle, for example on a moduli stack
\[
   \mathfrak M_\gamma
\]
of objects of charge $\gamma$.  Only after this stack-theoretic landing is specified can one compare the resulting realization with cohomological Donaldson--Thomas invariants, BPS sheaves, and wall-crossing factors \cite{KontsevichSoibelman08,JoyceSong12,Szendroi08,NagaoNakajima11,DavisonMeinhardt20}.

In this sense, the role of the present paper is to identify the geometry-side handoff object.  The future DT/BPS program should determine where
\[
   \mathcal C^H_\Sigma,
   \qquad
   P^H_I,
   \qquad
   \Delta_{I/IC}(X_0)
\]
land in the relevant moduli-stack framework, and how their realizations compare with refined or cohomological DT wall-crossing data.  This is the natural next step toward extracting BPS spectra, wall-crossing transformations, halo structures, and Fock-space packages from finite-node conifold geometry.

%------------------------------------------
%
%                   References
%

\printbibliography

\end{document}